\documentclass[12pt]{amsart}

\usepackage{amscd,latexsym,amsthm,amsfonts,amssymb,amsmath,amsxtra,mathdots}
\usepackage[mathscr]{eucal}
\renewcommand\theequation{\thesection.\arabic{equation}}

\newcommand{\BA}{{\mathbb {A}}}

\newcommand{\BC}{{\mathbb {C}}}

\newcommand{\BG}{{\mathbb {G}}}

\newcommand{\BR}{{\mathbb {R}}}

\newcommand{\CA}{{\mathcal {A}}}

\newcommand{\CC}{{\mathcal {C}}}
\newcommand{\CE}{{\mathcal {E}}}
\newcommand{\CF}{{\mathcal {F}}}

\newcommand{\CI}{{\mathcal {I}}}
\newcommand{\CJ}{{\mathcal {J}}}

\newcommand{\CL}{{\mathcal {L}}}
\newcommand{\CM}{{\mathcal {M}}}
\newcommand{\CN}{{\mathcal {N}}}
\newcommand{\CO}{{\mathcal {O}}}
\newcommand{\CP}{{\mathcal {P}}}

\newcommand{\CW}{{\mathcal {W}}}

\newcommand{\CZ}{{\mathcal {Z}}}

\newcommand{\FE}{{\mathfrak {E}}}

\newcommand{\Fb}{{\mathfrak {b}}}
\newcommand{\Fc}{{\mathfrak {c}}}
\newcommand{\Fd}{{\mathfrak {d}}}

\newcommand{\Fg}{{\mathfrak {g}}}

\newcommand{\Fl}{{\mathfrak {l}}}
\newcommand{\Fm}{{\mathfrak {m}}}
\newcommand{\Fn}{{\mathfrak {n}}}

\newcommand{\Fp}{{\mathfrak {p}}}

\newcommand{\Fr}{{\mathfrak {r}}}

\newcommand{\RI}{{\mathrm {I}}}
\newcommand{\RJ}{{\mathrm {J}}}

\newcommand{\RO}{{\mathrm {O}}}

\newcommand{\RU}{{\mathrm {U}}}

\newcommand{\Ad}{{\mathrm{Ad}}}
\newcommand{\Aut}{{\mathrm{Aut}}}
\newcommand{\aut}{{\mathrm{aut}}}

\newcommand{\Asai}{{\mathrm{As}}}

\newcommand{\cusp}{{\mathrm{cusp}}}
\newcommand{\Cent}{{\mathrm{Cent}}}

\newcommand{\disc}{{\mathrm{disc}}}

\newcommand{\diag}{{\mathrm{diag}}}

\newcommand{\el}{{\mathrm{ell}}}

\newcommand{\Gal}{{\mathrm{Gal}}}
\newcommand{\GL}{{\mathrm{GL}}}

\newcommand{\GSp}{{\mathrm{GSp}}}

\newcommand{\Hom}{{\mathrm{Hom}}}

\newcommand{\Ind}{{\mathrm{Ind}}}

\newcommand{\Isom}{{\mathrm{Isom}}}

\newcommand{\Int}{{\mathrm{Int}}}

\newcommand{\Mat}{{\mathrm{Mat}}}

\newcommand{\out}{{\mathrm{Out}}}

\newcommand{\PGL}{{\mathrm{PGL}}}

\renewcommand{\Re}{{\mathrm{Re}}}

\newcommand{\Res}{{\mathrm{Res}}}

\newcommand{\simp}{{\mathrm{sim}}}
\newcommand{\SL}{{\mathrm{SL}}}

\newcommand{\SO}{{\mathrm{SO}}}
\newcommand{\GO}{{\mathrm{GO}}}
\newcommand{\GSO}{{\mathrm{GSO}}}
\newcommand{\SU}{{\mathrm{SU}}}

\newcommand{\Sp}{{\mathrm{Sp}}}

\newcommand{\sm}{{\mathrm{ss}}}
\newcommand{\st}{{\mathrm{st}}}
\newcommand{\Span}{{\mathrm{Span}}}
\newcommand{\subr}{{\mathrm{subr}}}

\newcommand{\tr}{{\mathrm{tr}}}

\newcommand{\unit}{{\mathrm{unit}}}

\newcommand{\ud}{\,\mathrm{d}}

\newcommand{\ovl}{\overline}
\newcommand{\udl}{\underline}
\newcommand{\wt}{\widetilde}

\newcommand{\apair}[1]{\left\langle {#1} \right\rangle}

\newcommand{\cpair}[1]{\left\{{#1}\right\}}
\newcommand{\ppair}[1]{\left( {#1} \right)}

\newcommand{\ol}{\overline}

\def\bks{{\backslash}}

\def\eps{{\epsilon}}

\def\lam{{\lambda}}

\def\sym{{\rm sym}}
\def\sig{{\sigma}}

\newtheorem{thm}{Theorem}[section]

\newtheorem{conj}[thm]{Conjecture}
\newtheorem{cor}[thm]{Corollary}

\newtheorem{lem}[thm]{Lemma}

\newtheorem{prop}[thm]{Proposition}

\newtheorem{prin}[thm]{Principle}

\newtheorem{ques/conj}[thm]{Question/Conjecture}
\newtheorem{rmk}[thm]{Remark}

\makeatletter

\newcommand{\Rmnum}[1]{\expandafter\@slowromancap\romannumeral #1@}
\makeatother

\begin{document}
\renewcommand{\theequation}{\arabic{equation}}
\numberwithin{equation}{section}

\title[Arthur Parameters and Cuspidal Automorphic Modules]
{Arthur Parameters and Cuspidal Automorphic Modules of Classical Groups}

\author{Dihua Jiang}
\address{School of Mathematics,
University of Minnesota,
Minneapolis, MN 55455, USA}
\email{dhjiang@math.umn.edu}

\author{Lei Zhang}
\address{Department of Mathematics,
National University of Singapore,
Singapore 119076}
\email{matzhlei@nus.edu.sg}

\subjclass[2010]{Primary 11F70, 22E50; Secondary 11F85, 22E55}

\date{\today}


\thanks{The research of the first named author is supported in part by the NSF Grants DMS--1301567, DMS--1600685 and DMS--1901802; 
and that of the second named author is supported in part by the start-up grant and AcRF Tier 1 grant R-146-000-237-114 of National University of Singapore.}

\keywords{Arthur Parameters, Cuspidal Automorphic Modules, Bessel-Fourier Coefficients, Twisted Automorphic Descent, Classical Groups, Global Gan-Gross-Prasad Conjecture}

\begin{abstract}
The endoscopic classification via the stable trace formula comparison provides certain character relations between irreducible cuspidal
automorphic representations of classical groups and their global Arthur parameters, which are certain automorphic representations
of general linear groups. It is a question of J. Arthur and W. Schmid that asks: {\sl How to construct concrete modules for irreducible
cuspidal automorphic representations of classical groups in term of their global Arthur parameters?}
In this paper, we formulate a general construction of concrete modules, using Bessel periods, for cuspidal automorphic representations
of classical groups with generic global Arthur parameters. Then we establish the theory for orthogonal and unitary groups, based on
certain well expected conjectures. Among the consequences of the theory in this paper is that the global Gan-Gross-Prasad conjecture for those
classical groups is proved in full generality in one direction and with a global assumption in the other direction.
\end{abstract}

\maketitle

\tableofcontents


\section{Introduction}


Let $F$ be a number field and $\BA$ be the ring of the adeles of $F$. Let $G$ be a classical group defined over $F$.
The theory of endoscopic classification gives a parameterization of the irreducible automorphic representations of $G(\BA)$
occurring in the discrete spectrum of all square-integrable automorphic functions on $G(\BA)$, up to global Arthur packets,
by means of global Arthur parameters. These parameters are formal sums of
certain irreducible square-integrable automorphic representations of general linear groups.
This fundamental theory has been established by J. Arthur in \cite{A13} for $G$ to be either symplectic groups or $F$-quasisplit special orthogonal groups,
with outline on general orthogonal groups in \cite[Chapter 9]{A13}. Following the fundamental work of Arthur (\cite{A13}), several authors
made progress for more general classical groups. C.-P. Mok established the theory for $F$-quasisplit unitary groups (\cite{Mk15}). More recently,
Kaletha, Minguez, Shin, and White in \cite{KMSW} made progress on more general unitary groups. We refer to the work of B. Xu (\cite{Xu})
for progress on the cases of similitude classical groups $\GSp_{2n}$ and $\GO_{2n}$.
We remark that all those works depend on the stabilization of the
twisted trace formula, which has been achieved through a series of works of C. M\oe glin and J.-L. Waldspurger that are now given in their books
(\cite{MW-Book1} and \cite{MW-Book2}).

In Problem No.\ $5$ in the {\sl Open Problems in honor of W. Schmid} (\cite{A-P5}), Arthur explains that the trace formula method
establishes certain {\sl character relation} between irreducible cuspidal automorphic representations of classical groups and their
global Arthur parameters.
It was Schmid who asks: {\sl ``What about modules...?"}. This means how to construct a concrete
module for any irreducible cuspidal automorphic representation in terms of its global Arthur parameter.
In \cite{A-P5}, Arthur posed this question and pointed out that the work of the first named author (\cite{J14}) has
the potential to give an answer to this question.

Our objective is to formulate, in the spirit of the constructive theory described in \cite{J14} and also \cite{J17},
a general construction (Principle \ref{prin}) of concrete modules for cuspidal automorphic representations of
general classical groups, which provides an answer to the question of Arthur-Schmid.

In this paper we establish the theory of concrete modules (Conjecture \ref{pmc}), under certain well expected conjectures (Conjecture \ref{bpconj}, for instance), for cuspidal automorphic representations with generic global Arthur parameters (Theorem \ref{th-mcgeneral}).
The key idea in the theory is to introduce the method of {\sl twisted automorphic descents}, which
extends the method of automorphic descents of Ginzburg-Rallis-Soudry (\cite{GRS11}) from $F$-quasisplit classical groups to general classical groups,
and from generic cuspidal automorphic representations to general cuspidal automorphic representations with generic global Arthur parameters.

One of the main technical issues in the method
is to establish the global non-vanishing of the twisted automorphic descents that are constructed from the given data. This is treated
by establishing the {\sl reciprocal non-vanishing for Bessel periods} (Theorem \ref{th-rnbp}), which depends heavily on the extension of the global and
local theory of the global zeta integrals that represent the tensor product $L$-functions to the generality considered in this paper from the
work of Ginzburg, Piatetski-Shapiro and Rallis (\cite{GPSR97}), the work of the current authors (\cite{JZ14}), and the recent work of Soudry
(\cite{S-I} and \cite{S-II}). Those previously done works mainly treat the $F$-quasisplit classical groups. The extension of the global theory is discussed
in Section \ref{sec-bpgzi} of this paper, and that of the local theory is given in our joint work with Soudry in \cite{JSdZ}.
Another technical issue is to prove the irreducibility of the concrete modules constructed via the twisted automorphic descents, which
is carried out by using the local Gan-Gross-Prasad conjecture (Conjecture \ref{smoconj}) as input.
As a consequence, we are able to establish one direction of the global
Gan-Gross-Prasad conjecture in full generality (Theorem \ref{gggp1}), while the other direction of the conjecture with a global assumption (Theorem \ref{gggp2}), except some special cases (Corollary \ref{cor-gggpm1}, and also \cite{JLXZ}), where such a global assumption can be established.

The global Gan-Gross-Prasad conjecture that we refer to is Conjectures 24.1 (and Conjecture 26.1 for a different formulation) in \cite{GGP12}.
It was first made by B. Gross and D. Prasad in \cite{GP92} and \cite{GP94} for orthogonal groups, and
was reformulated in full generality for all classical groups, including the metaplectic groups, by Gan, Gross and Prasad in \cite{GGP12}.
The progress towards the proof of the global Gan-Gross-Prasad conjecture can be traced back to the pioneering
work of Harder-Langlands-Rapoport on the Tate conjecture for Hilbert-Blumenthal modular surfaces (\cite{HLR}), and has been well explained
in \cite{GGP12}, \cite{GGP}, and also in \cite{Gan-AMV}.

It is important to point out that the work of W. Zhang (\cite{ZW-1} and \cite{ZW-2}) established the global Gan-Gross-Prasad conjecture for a special family of unitary groups with certain global and local constraints. The approach taken up in \cite{ZW-1} and \cite{ZW-2} is to use the relative trace formula developed by H. Jacquet and S. Rallis in \cite{JR11} for unitary groups.
However, such a relative trace formula that can be used to attack the global Gan-Gross-Prasad conjecture for orthogonal groups is so far
not known to be available.
The global Gan-Gross-Prasad conjecture for generic cuspidal automorphic representations with simple global Arthur parameters
was considered in \cite{GJR04}, \cite{GJR05}, and \cite{GJR09} for symplectic and metaplectic groups, orthogonal groups, and unitary groups, respectively.
The method is a combination of the Bessel or Fourier-Jacobi periods of certain residual representations with the Arthur truncation method. It was recently
discovered that there is a technical gap in the argument towards the end of the proof, which needs to be filled up.
A similar approach with the Arthur truncation replaced by the Jacquet-Lapid-Rogawski truncation is applied to the case of $\RU_{n+1}\times\RU_n$
by A. Ichino and S. Yamana in \cite{IY}.
We refer to Section \ref{sec-gggp} for a more detailed account.

The approach taken up in this paper treats the global Gan-Gross-Prasad conjecture uniformly for unitary groups and orthogonal groups, and can be used to take
care of the symplectic group and metaplectic group situation by using the Fourier-Jacobi periods (\cite{JZ-Howe}). It avoids the technical difficulties
that occur in the work (\cite{GJR04}, \cite{GJR05}, and \cite{GJR09}), which seem hopeless to be smoothly handled
when one considers general cuspidal automorphic representations with generic global
Arthur parameters and general classical groups. More importantly, the approach in this paper is much naturally related to the theory
of twisted automorphic descents and the general Rankin-Selberg method, so that one may regard the global Gan-Gross-Prasad conjecture as part of the
theory developed in our work. Finally, the results on the global Gan-Gross-Prasad conjecture in this paper (Theorems \ref{gggp1} and \ref{gggp2})
do not assume that the cuspidal multiplicity should be one, while this cuspidal multiplicity one assumption was taken for the
global Gan-Gross-Prasad conjecture in \cite{GGP12}. This cuspidal multiplicity one issue was also discussed by H. Xue in \cite[Section 6]{XueH17}.

We also refer to \cite{ZW-1} and \cite{ZW-2} for
a beautiful explanation of the relation between the Gan-Gross-Prasad conjecture and certain important problems in arithmetic and geometry,
and for a more complete account of the progress on lower rank examples and other special cases towards the global conjecture and its refinement.

It is worthwhile to mention that the basic theoretic framework and technical results developed in this paper have been used in some recent work (\cite{JLX} and \cite{JZ}) to study the 
{\sl  automorphic branching problem and its reciprocal problem}, 
and to establish certain cases of the global Gan-Gross-Prasad conjecture for {\sl non-tempered} global Arthur parameters, which has been recently formulated as \cite[Conjecture 9.1]{GGP-ng}. 

\subsection{Main ideas and arguments in the theory}
In order to illustrate the main ideas and arguments of the theory in this introduction, we take $G$ to be an odd special orthogonal group. The general case
will be discussed in the main body of this paper.

We denote by $G_n^*=\SO(V^*,q^*)$ the $F$-split odd special orthogonal group of $2n+1$ variables.
Let $G_n=\SO(V,q)$ be the odd special orthogonal group defined by a $2n+1$ dimensional non-degenerate quadratic space $(V,q)$ over $F$.
Then $G_n$ is a pure inner form of $G_n^*$ over $F$, in the
sense of Vogan (in \cite{V93} and also in \cite{GGP12}, \cite{Klocal} and \cite{Kglobal}).
Following the work of Arthur (\cite{A13}, Chapter 9, in particular),
the discrete spectrum of $G_n$ are parameterized by the $G_n$-relevant, global Arthur parameters of $G_n^*$, the set of which is
denoted by $\wt{\Psi}_2(G_n^*)_{G_n}$. The global Arthur parameters of $G_n^*$ are multiplicity-free formal sums of the type
\begin{equation}\label{ap}
\psi=(\tau_1,b_1)\boxplus\cdots\boxplus(\tau_r,b_r)\in\wt{\Psi}_2(G_n^*),
\end{equation}
where $\tau_i$ is an irreducible unitary self-dual cuspidal automorphic representation of $\GL_{a_i}(\BA)$ for $i=1,2,\cdots,r$,
having the property that
when $\tau_i$ is of orthogonal type, the integer $b_i$ must be even, and when $\tau_i$ is of symplectic type, the integer $b_i$ must be odd.

Following \cite{A13}, a global Arthur parameter $\psi$ is called {\sl generic} if $b_i=1$ for $i=1,2,\cdots,r$. The subset of the generic
parameters is denoted by $\wt{\Phi}_2(G_n^*)$ and that of the $G_n$-relevant ones is denoted by $\wt{\Phi}_2(G_n^*)_{G_n}$.
Hence the generic global Arthur parameters are of the form
\begin{equation}\label{gap}
\phi=(\tau_1,1)\boxplus\cdots\boxplus(\tau_r,1).
\end{equation}
It follows that for a generic global Arthur parameter $\phi$ in \eqref{gap}, the cuspidal automorphic
representations $\tau_1,\cdots,\tau_r$ are all of
symplectic type and $\tau_i$ is not equivalent to $\tau_j$ if $i\neq j$.

By \cite{A13}, in particular, \cite[Chapter 9]{A13},
for any $\pi\in\CA_\cusp(G_n)$, the set of equivalence classes of irreducible automorphic representations of $G_n$ that occur in the
cuspidal spectrum, there is a $G_n$-relevant, global Arthur parameter $\psi\in\wt{\Psi}_2(G_n^*)$, such that
$\pi\in\wt{\Pi}_\psi(G_n)$, the global Arthur packet of $G_n$ associated to $\psi$. One has the following diagram:
\begin{equation}\label{diag1}
\begin{matrix}
                & &\wt{\Psi}_2(G_n^*)& & &\\
                &&&&&\\
                &&\psi&&&\\
                &&&&&\\
                &\swarrow& &\searrow&&\\
                &&&&&\\
  \CA_\disc(G_n)\cap \wt{\Pi}_\psi(G_n)&       & {\iff}        &           & \CA_\disc(G_n^*) \cap \wt{\Pi}_\psi(G_n^*)\\
\end{matrix}
\end{equation}
where $\CA_\disc(G_n)$ is the set of equivalence classes of irreducible automorphic representations of $G_n(\BA)$ that occur in the discrete spectrum.

When a parameter $\psi\in\wt{\Psi}_2(G_n^*)$ is generic, i.e. $\psi=\phi$ as given in \eqref{gap},
the global packet $\wt{\Pi}_\phi(G_n^*)$ contains an irreducible generic cuspidal automorphic representation $\pi_0$ of $G_n^*(\BA)$.
This $\pi_0$ can be constructed by the automorphic descent of Ginzburg, Rallis and Soudry in \cite{GRS11} and
in \cite{JS03}. This construction produces a
concrete module for $\pi_0$ by using only the generic global
Arthur parameter $\phi$. However, it remains a {\sl big problem} to construct other cuspidal members in the global packet $\wt{\Pi}_\phi(G_n^*)$, and
even more generally, to construct all cuspidal members in $\wt{\Pi}_\psi(G_n)$ for all pure inner forms $G_n$ of $G_n^*$.

It seems clear from Diagram \eqref{diag1} that one has to take more invariants of $\pi$ into consideration,
in order to develop a reasonable theory that constructs
concrete modules of all cuspidal members in $\wt{\Pi}_\psi(G)$ for general classical groups $G$.
One of the natural choices is to utilize the structure of Fourier coefficients of
cuspidal automorphic representations $\pi$, in addition to the global Arthur parameters $\psi$. We use $\CF(\pi,G)$ to denote a certain
piece of information about the structure of Fourier coefficients of $\pi$. Here is the principle of the theory.

\begin{prin}[Concrete Modules]\label{prin}
Let $G^*$ be an $F$-quasisplit classical group and $G$ be a pure inner form of $G^*$. For an irreducible cuspidal automorphic representation
$\pi$ of $G(\BA)$, assuming that $\pi$ has a $G$-relevant global Arthur parameter $\psi\in\wt{\Psi}_2(G^*)$, there exists a datum
$\CF(\pi,G)$ such that one is able to construct a concrete irreducible module $\CM(\psi, \CF(\pi,G))$, depending on the data $(\psi, \CF(\pi, G))$,
with the property that
$$
\pi\cong\CM(\psi, \CF(\pi, G)).
$$
Moreover, if $\pi$ occurs in the cuspidal spectrum of $G$ with multiplicity one, then $\pi=\CM(\psi, \CF(\pi, G))$.
\end{prin}

We remark that if $G=G^*$ is $F$-quasisplit and $\pi$ is generic, the concrete module expected in Principle \ref{prin} should
coincide with the module constructed from the automorphic descents of Ginzburg-Rallis-Soudry in \cite{GRS11}.
This will be explained in Corollary \ref{generic}.

We still take $G_n$ to be an odd special orthogonal group.
For the case when the global Arthur parameter $\psi$ is generic, we propose the {\sl Main Conjecture} (Conjecture \ref{pmc}) of the theory
developed in this paper that specifies
the {\sl Principle of Concrete Modules} (Principle \ref{prin}) with the datum $\CF(\pi,G_n)$ explicitly given in Conjecture \ref{bpconj}.
The nature of Conjecture \ref{bpconj} will be briefly discussed in Section \ref{sec-fcpt} and will be considered in our future
work.
With $\CF(\pi, G_n)$ as described in Conjecture \ref{bpconj}, and with the generic global Arthur parameter $\phi$ for $\pi$,
the construction of the concrete module $\CM(\phi, \CF(\pi, G_n))$ for the given $\pi$ is carried out by the {\sl twisted automorphic descent} as
illustrated in Diagram \eqref{diag}.

One of the key results in this paper is Theorem \ref{th-rnbp}, which gives a {\sl reciprocal non-vanishing for Bessel periods}.
Such a non-vanishing property is proved using a refined theory of the global zeta integrals for the tensor product $L$-functions for $G_n$
and a general linear group. The global theory of the global zeta integrals goes back to
the pioneering work of Ginzburg, Piatetski-Shapiro and Rallis for orthogonal groups (\cite{GPSR97}),
which has been extended to a more general setting,
including unitary groups by the authors of this paper in \cite{JZ14}.
We established the global results of the global zeta integrals for the most general situation
in Section \ref{sec-bpgzi}. In order to obtain Theorem \ref{th-rnbp}, we need the explicit unramified
computation of the local zeta integrals. This is done in \cite{JSdZ}, which extends the work of Soudry (\cite{S-I} and \cite{S-II})
for split orthogonal groups to the generality considered in this paper.

By using the {\sl reciprocal non-vanishing for Bessel periods}, we are able to show that certain Fourier coefficient of a residual representation,
which is denoted by $\CF^{\CO_{\kappa_0}}(\CE_{\tau\otimes\sigma})$, is nonzero. The notation is referred to Theorem \ref{th-rnbp}.
This is the candidate for the concrete module of $\pi$, as explained in the main conjecture of the theory (Conjecture \ref{pmc}).
Conjecture \ref{pmc} for $G_n$ asserts that $\CF^{\CO_{\kappa_0}}(\CE_{\tau\otimes\sigma})$ is an
irreducible cuspidal automorphic representation of $G_n(\BA)$ that is isomorphic to the given $\pi$.
We note that when $G_n$ is an even special orthogonal group, this assertion has to be modified due to the extra outer involution.
We refer to Conjecture \ref{pmc} for detail.

One of the main results of this paper (Theorem \ref{th-mcgeneral}) is to prove that Conjecture \ref{pmc} holds
under the assumption of Conjecture \ref{bpconj} and
Conjecture \ref{smoconj},
In two special cases, the results are stronger as given in
Corollary \ref{generic} and Corollary \ref{sboc}. It is worthwhile to mention that by a different argument, this theory recovered the
classical Jacquet-Langlands correspondence for $\PGL(2)$ in \cite{JLXZ}.

We remark that the irreducibility of $\CF^{\CO_{\kappa_0}}(\CE_{\tau\otimes\sigma})$ is deduced from Conjecture \ref{smoconj}, which is the local Gan-Gross-Prasad conjecture for local Vogan packets. The recent progress towards this conjecture is recorded as Theorem \ref{mwslmo}, according to
the work of M\oe glin-Waldspurger (\cite{MW12}), the work of R. Beuzart-Plessis (\cite{BP12} and \cite{BP}), the work of Gan-Ichino (\cite{GI}),
the work of H. He (\cite{He}), and the PhD Thesis of Zhilin Luo (\cite{Luo-thesis}).

Based on our theory, Theorem \ref{gggp1}
proves one direction of the global Gan-Gross-Prasad conjecture in full generality for the classical groups considered in this paper, while
Theorem \ref{gggp2} proves
the other direction of the conjecture with a global assumption (Conjecture \ref{gnvc}), which is about certain structure of
Fourier coefficients of the relevant residual representations. We refer to
\cite[Section 4]{J14} and \cite{JL-Cogdell} for discussion of the general issue related to the conjecture.

\subsection{Structure of this paper}
A more detailed description of the content in each section is in order. In Section \ref{sec-ccg}, we discuss the family of classical groups
considered in this paper and recall their basic structures. The global Arthur parameters and the discrete spectrum for those classical groups
are discussed in Section \ref{sec-dsap}. We recall from \cite{J14} and \cite{JL-Cogdell} the general notion of Fourier coefficients of automorphic
forms associated to the partitions or nilpotent orbits in Section \ref{sec-fcpt} and give a more detailed account for the special type of
Fourier coefficients, which is often called the {\sl Bessel-Fourier coefficients}. Based on the {\sl tower property} for Bessel-Fourier coefficients
of cuspidal automorphic forms (Proposition \ref{cfc}), we state Conjecture \ref{bpconj}. This is our starting point in the theory of construction
of concrete modules for irreducible cuspidal automorphic representations for general classical groups via the twisted automorphic descents. In Section \ref{sec-pif}, we show
(Proposition \ref{piform}) the construction illustrated by Diagram \eqref{diag} covers all the classical groups considered in this paper
as described in Section \ref{sec-ccg}.

The local Gan-Gross-Prasad conjecture (as in Conjecture \ref{smoconj}) is one of the key inputs in the proof of the irreducibility of
the constructed modules. We recall from \cite{GGP12} for the cases considered in this paper in Section \ref{sec-ggp}, and state
Conjecture \ref{smoconj}, which is needed for one of the main results in the paper (Theorem \ref{th-mcgeneral}).
The known cases of Conjecture \ref{smoconj} are stated in Theorem \ref{mwslmo}.
In Section \ref{sec-bpgzi}, we consider a family of global zeta integrals, which represent the tensor product
$L$-functions for the classical groups defined in Section \ref{sec-ccg} and the general linear groups. We show that they can be written as
an Euler product of local zeta integrals (Theorem \ref{thm:j>l} and Theorem \ref{thm:j<l}). With the explicit results on the unramified calculation
of the local zeta integrals in terms of the local $L$-factors (Theorem \ref{urmL}), the global zeta integral can be written in a formula in \eqref{formula1}.
Based on what was discussed in Section \ref{sec-bpgzi}, we establish in Section \ref{sec-rnbp}
the necessary analytic properties of the local zeta integrals in Section \ref{sec-nlzi}, which are needed to establish the
{\sl reciprocal non-vanishing for Bessel periods} (Theorem \ref{th-rnbp}). While some of the properties of the local zeta integrals can be deduced
from the global argument based on the formula in \eqref{formula1} for the global zeta integrals, one of the most technical local results is Proposition
\ref{nlzip:nonzero}, which asserts a general non-vanishing of the local zeta integrals for the data with certain global constraints, and will be proved
in Appendix \ref{A}.
It is also important to mention that Theorem \ref{nlio}
on the analytic properties of the normalized local intertwining operators is another key input in this theory. We will prove
Theorem \ref{nlio} in Appendix \ref{B}.
As a consequence, we obtain in Theorem \ref{gggp1} one direction of the global Gan-Gross-Prasad conjecture in full generality.
With Conjecture \ref{bpconj}
and Theorem \ref{th-rnbp}, in addition to Theorem \ref{nlio}, we are able to obtain the non-vanishing of the Fourier coefficient of
the particular residual representation $\CF^{\CO_{\kappa_0}}(\CE_{\tau\otimes\sigma})$, which is one of the key points in the theory.
The basic properties of $\CF^{\CO_{\kappa_0}}(\CE_{\tau\otimes\sigma})$ are established in Section \ref{sec-ccam}, which are similar to those in
the automorphic descents of Ginzburg, Rallis and Soudry (\cite{GRS11}) and in our previous work joint with Liu and Xu (\cite{JLXZ}).
As a consequence, we obtain results towards another direction of the global Gan-Gross-Prasad conjecture (Theorem \ref{gggp2}).

Diagram \eqref{diag} illustrates the main idea and process of the construction of concrete modules for irreducible
cuspidal automorphic representations of $G$ that have generic global Arthur parameters. In the framework of the construction given by
Diagram \eqref{diag}, we state the main conjecture of the theory (Conjecture \ref{pmc}). As one of the main results of this paper, we prove
Theorem \ref{th-mcgeneral} that Conjecture \ref{pmc} holds under the assumption of Conjecture \ref{bpconj} and
Conjecture \ref{smoconj}. In two special cases, Conjecture \ref{bpconj} is trivial or can be easily verified. Hence
we can have stronger results for those two special cases. The one attached to the regular partition (Corollary \ref{generic}) is essentially the automorphic
descents in \cite{GRS11}, and the other attached to the subregular partition (Corollary \ref{sboc}) is new, and
is a generalization of the construction considered
in \cite{JLXZ}.

There are two appendices following the main body of this paper. Appendix \ref{A} proves Proposition \ref{nlzip:nonzero} in a more general setting.
We put this as one of the two appendices so as to ensure a smoother logic flow in the main body of this paper.
Appendix \ref{B} proves Theorem \ref{nlio}. We leave this out of the main body because the proof needs different preparation, although
there is a possibility to put it in Section \ref{sec-ggp}.

Finally, we would like to thank J. Arthur and W. Schmid for asking and posing this very interesting and important problem in 2013,
which stimulates and encourages us to carry out the work in this paper. We hope the main results and conjectures in this paper to be
helpful towards  the understanding of the nature of their problem.
We are grateful to D. Soudry for his help in finding the proof presented in Appendix \ref{A}, which works uniformly for
all local places.
We would also like to thank
C. M\oe glin, F. Shahidi, D. Vogan, and B. Xu for very helpful conversation
about the proof of the results in Appendix \ref{B},
and thank W. T. Gan for his helpful comments and suggestions on several issues on the theory considered here.
Last, but not least, we would like to thank P. Sarnak for his interest in
and encouraging comments on the theory and results developed
in this paper, and to thank the referee for very important and useful comments and suggestions, which well improved the exposition of the paper.

\section{Discrete Spectrum and Fourier Coefficients}\label{dsfc}

\subsection{Certain classical groups}\label{sec-ccg}
The classical groups considered in this paper are unitary groups and special orthogonal groups that are explicitly defined below.

Let $F$ be a number field and $\BA=\BA_F$ be the ring of adeles of $F$. Let $F(\sqrt{\varsigma})$ be a quadratic field extension of $F$,
with $\varsigma$ is a non-square in $F^\times$. Let $E$ be either $F$ or $F(\sqrt{\varsigma})$, and consider the Galois group
$\Gamma_{E/F}=\Gal(E/F)$.
It is trivial if $E=F$, and has a unique non-trivial element $\iota$ \label{pg:iota} if $E=F(\sqrt{\varsigma})$.
Let $(V,q)$ be an $\Fn$-dimensional non-degenerate vector space over $E$, which is Hermitian if $E=F(\sqrt{\varsigma})$ and
is symmetric (or quadratic) if $E=F$. Denote by $G_n=\Isom(V,q)^\circ$ the identity connected component of
the isometry group of the space $(V,q)$, with $n=[\frac{\Fn}{2}]$.
Let $G_n^*=\Isom(V^*,q^*)^\circ$ be an $F$-quasisplit group of the same type, so that
$G_n$ is a pure inner form of $G_n^*$ over the field $F$, following \cite{V93} and \cite{GGP12}.

Let $(V_0,q)$ be the $F$-anisotropic kernel of $(V,q)$ with dimension $\Fd_0=\Fn-2\Fr$, where the $F$-rank $\Fr=\Fr_\Fn=\Fr(G_n)$ of $G_n$ is
the same as the Witt index of $(V,q)$.
Let $V^{+}$  be a maximal totally isotropic subspace of $(V,q)$, with $\{e_{1},\dots,e_{\Fr}\}$ being its basis.
Choose $E$-linearly independent vectors $\{e_{-1},\dots,e_{-\Fr}\}$ in $(V,q)$
such that
$$
q(e_{i},e_{-j})=\delta_{i,j}
$$
for all $1\leq i,j\leq\Fr$.
Denote by $V^{-}=\Span\{e_{-1},\dots,e_{-\Fr}\}$ the dual space of $V^+$.
Then $(V,q)$ has the following polar decomposition
$$
V=V^{+}\oplus V_{0}\oplus V^{-},
$$
where $V_{0}=(V^{+}\oplus V^{-})^{\perp}$ is an $F$-anisotropic kernel of $(V,q)$.
We choose an orthogonal basis $\{e'_{1},\dots,e'_{\Fd_0}\}$ of $V_{0}$ with the property that
$$
q(e'_{i},e'_{i})=d_{i},
$$
where $d_{i}$ is nonzero for all $1\leq i\leq \Fd_0$. Set $G_{d_0}=\Isom(V_{0},q)^\circ$ with $d_0=[\frac{\Fd_0}{2}]$,
which is anisotropic over $F$ and is regarded as an $F$-subgroup of $G_n$.

We put the above bases together in the following order to form a basis of $(V,q)$:
\begin{equation}\label{bs}
e_{1},\dots,e_{\Fr},e'_{1},\dots,e'_{\Fd_0},e_{-\Fr},\dots,e_{-1},
\end{equation}
and fix the following full isotropic flag in $(V,q)$:
$$
\Span\{e_{1}\}\subset\Span\{e_{1},e_{2}\}\subset
\cdots\subset
\Span\{e_{1},\dots,e_{\Fr}\},
$$
which defines a minimal parabolic $F$-subgroup $P_0$.
Moreover, $P_0$ contains a maximal $F$-split torus $S$, consisting of elements
$$
\diag\{t_{1},\dots,t_{\Fr},1,\dots,1,{t}^{-1}_{\Fr},\dots,t^{-1}_{1}\},
$$
with $t_i\in F^\times$ for $i=1,2,\cdots,\Fr$. Then the centralizer $Z(S)$ in $G_n$ is $\Res_{E/F}S\times G_{d_0}$, the Levi subgroup
of $P_0$, where $\Res_{E/F}S$ is the Weil restriction of $S$ from $E$ to $F$.
Then $P_0$ has the Levi decomposition:
$$
P_0=(\Res_{E/F}S\times G_{d_0})\ltimes N_0
$$
where $N_0$ is the unipotent radical of $P_0$.
Also, with respect to the order of the basis in \eqref{bs}, the group $G_n$ is also defined by the following symmetric matrix:
\begin{equation} \label{eq:J}
J_{\Fr}^\Fn=\begin{pmatrix}
&&1\\&J_{\Fr-1}^{\Fn-2}&\\1&&
\end{pmatrix}_{\Fn\times\Fn}
\text{ and }
J_{0}^{\Fd_0}=\diag\{d_{1},\dots,d_{\Fd_0}\}
\end{equation}
as defined inductively.

Let $_{F}\!\Phi(G_n,S)$ be the root system of $G_n$ over $F$.
Let $_{F}\!\Phi^{+}(G_n,S)$ be the positive roots corresponding to the minimal parabolic $F$-subgroup $P_0$,
and $_{F}\!\Delta=\{\alpha_{1},\dots,\alpha_{\Fr}\}$ be a set of simple roots in $_{F}\!\Phi^{+}(G_n,S)$.
When $G_n$ is an orthogonal group, the root system $_{F}\!\Phi(G_n,S)$ is of type $B_{\Fr}$ unless $\Fn=2\Fr$,
in which case it is of type $D_{\Fr}$.
When $G_n$ is a unitary group, the root system $_{F}\!\Phi(G_n,S)$ is non-reduced of type $BC_{\Fr}$ if $2\Fr<\Fn$;
otherwise, $_{F}\!\Phi(G_n,S)$ is of type $C_{\Fr}$.

For a subset $J\subset\{1,\dots,\Fr\}$, let $_{F}\!\Phi_{J}$ be the root subsystem  of $_{F}\!\Phi(G_n,S)$
generated by the simple roots $\{\alpha_{j}\colon j\in J\}$.
Let $P_{J}=M_{J}U_{J}$ be the standard parabolic $F$-subgroup of $G_n$,
whose Lie algebra consists of all roots spaces $\Fg_{\alpha}$ with $\alpha\in {_{F}\!\Phi^{+}(G_n,S)}\cup {_{F}\!\Phi_{J}}$.
For instance,  if set $\hat{i}:=\{1,\dots,\Fr\}\setminus\{i\}$,
then $P_{\hat{i}}=M_{\hat{i}}U_{\hat{i}}$ is the standard maximal parabolic $F$-subgroup of $G_n$,
which stabilizes the rational isotropic space $V^+_i$,
where $V^{\pm}_{i}:=\Span\{e_{\pm1},\dots,e_{\pm i}\}$.
Here $U_{\hat{i}}$ is the unipotent radical of $P_{\hat{i}}$ and
the Levi subgroup $M_{\hat{i}}$ is isomorphic to $G_{E/F}(i)\times G_{n-i}$. Following the notation of \cite{A13} and \cite{Mk15},
$G_{E/F}(i):=\Res_{E/F}\GL_i$ denotes the Weil restriction of $E$-group $\GL_{i}$ restricted to $F$.
Write  $V_{(i)}=(V^{+}_{i}\oplus V^{-}_{i})^{\perp}$ and hence $V_{(\Fr)}=V_0$ is the $F$-anisotropic kernel of $(V,q)$.

We recall simply from \cite{GGP12} the classification of pure inner $F$-forms of $F$-quasisplit classical groups $G_n^*$ for
a local field and then for a number field.

For a local field $F$ of characteristic zero, we recall the notion of a {\sl relevant pair} of classical groups.
As above, we let $G_n:=\Isom(V,q)^\circ$ be defined for an $\Fn$-dimensional non-degenerate space $(V,q)$ with $n=[\frac{\Fn}{2}]$.
Take an $\Fm$-dimensional non-degenerate subspace $(W,q)$ of $(V,q)$ with the property that the orthogonal complement $(W^\perp,q)$ is $F$-split and
has an odd dimension. Define $H_m:=\Isom(W,q)^\circ$ with $m=[\frac{\Fm}{2}]$.
By \cite[Section 2]{GGP12}, the pair $(G_n,H_m)$ forms a {\sl relevant} pair.

If $G_n':=\Isom(V',q')^\circ$ and $H_m':=\Isom(W',q')^\circ$ form another relevant pair, and
if $G_n'$ and $H_m'$ are pure inner $F$-form of $G_n$ and $H_m$, respectively, the product
$G_n'\times H_m'$ is defined to be {\sl relevant} to the product
$G_n\times H_m$
if the orthogonal complement $((W')^\perp,q')$ is equivalent to the orthogonal complement $(W^\perp,q)$, as Hermitian vector spaces.
From \cite[Lemma 2.2, Part (i)]{GGP12}, one can have an easy list of all $F$-relevant pairs $(G_n,H_m)$ whose product
$G_n\times H_m$ is relevant to the $F$-quasisplit product $G_n^*\times H_m^*$.

For a number field $F$, $G_n$ is a pure inner $F$-form of an $F$-quasisplit $G_n^*$ if
it is obtained by inner twisting by elements in the pointed set $H^1(F,G_n)$. It follows that
at every local place $\nu$, $G_n$ is a pure inner $F_\nu$-form of $G_n^*$.
The notion of {\sl relevance} is defined in the same way. We will come back to this in Section \ref{sec-ggp} when
we discuss Vogan packets and the Gan-Gross-Prasad conjectures.

\subsection{Discrete spectrum and Arthur packets}\label{sec-dsap}
For a reductive algebraic group $G$ defined over $F$, denote by $\CA_\disc(G)$ the set of equivalence classes of irreducible
unitary representations $\pi$ of $G(\BA)$ occurring in the discrete spectrum $L^2_\disc(G)$ of $L^2(G(F)\bks G(\BA)^1)$, when $\pi$ is restricted to
$G(\BA)^1$. Also denote by $\CA_\cusp(G)$ for the subset of $\CA_\disc(G)$, whose elements occur in the cuspidal spectrum $L^2_\cusp(G)$.
The theory of endoscopic classification for classical groups $G_n$ is to parameterize the set $\CA_\disc(G_n)$ by means of the global Arthur parameters,
which can be realized as certain automorphic representations of general linear groups. We recall from the work of Arthur (\cite{A13}),
the work of Mok (\cite{Mk15}) and the work of Kaletha, Minguez, Shin, and White (\cite{KMSW}) the theory for the (special) orthogonal groups and
the unitary groups considered in this paper.

First, we take an $F$-quasisplit classical group $G_n^*$, of which $G_n$ is a pure inner $F$-form. Both $G_n^*$ and $G_n$ share the same
$L$-group ${^LG_n^*}={^LG_n}$. Define $\Fn^\vee$ to be $\Fn$ if $G_n$ is
a unitary group or an even special orthogonal group; and to be $\Fn-1$ if $G_n$ is an odd special orthogonal group.
This number $\Fn^\vee$ is denoted by $N$ in \cite{A13}, \cite{Mk15} and \cite{KMSW}.

Following \cite{A13}, \cite{Mk15} and \cite{KMSW}, we denote by $\wt{\CE}_\simp(N)$ (with $N=\Fn^\vee)$ the set of the equivalence classes of simple twisted
endoscopic data. Each member in $\wt{\CE}_\simp(N)$ is represented by a triple $(G,s,\xi)$, where $G$ is an $F$-quasisplit classical group,
$s$ is a semi-simple
element as described in \cite[Page 11]{A13} and \cite[Page 16]{Mk15}, and $\xi$ is the $L$-embedding
$$
{^LG}\rightarrow {^LG_{E/F}(N)}.
$$
Note that when $G$ is an $F$-quasisplit unitary group, the $L$-embedding $\xi=\xi_{\chi_\kappa}$ depends on $\kappa=\pm1$.
As in \cite[Page 18]{Mk15}, for a simple twisted endoscopic datum $(\RU_{E/F}(N), \xi_{\chi_\kappa})$ of $G_{E/F}(N)$,
the sign $(-1)^{N-1}\cdot\kappa$ is called the {\sl parity} of the datum.
The set of global Arthur parameters for $G_n^*$ is denoted by $\wt{\Psi}_2(G_n^*,\xi)$, or simply by $\wt{\Psi}_2(G_n^*)$ if the
$L$-embedding $\xi$ is well understood in the discussion.

In order to explicate the structure of the parameters in $\wt{\Psi}_2(G_n^*,\xi)$, we first recall from \cite{A13} and \cite{Mk15} the
description of the conjugate self-dual, elliptic, global Arthur parameters for $G_{E/F}(N)$, the set of which is denoted by $\wt{\Psi}_\el(N)$.
We refer to \cite{A13}, \cite{Mk15} and also \cite{KMSW} for detailed discussion about general global Arthur parameters. The elements of
$\wt{\Psi}_\el(N)$ are denoted by $\psi^N$, which have the form
\begin{equation}\label{ellgap}
\psi^N=\psi_1^{N_1}\boxplus\cdots\boxplus\psi_r^{N_r}
\end{equation}
with $N=\sum_{i=1}^rN_i$. The formal summands $\psi_i^{N_i}$ are {\sl simple} parameters of the form
$$
\psi_i^{N_i}=\mu_i\boxtimes\nu_i
$$
with $N_i=a_ib_i$, where $\mu_i=\tau_i\in\CA_\cusp(G_{E/F}(a_i))$ and $\nu_i$ is a $b_i$-dimensional representation of $\SL_2(\BC)$.
Following the notation used in our previous paper \cite{J14}, we also denote
$$
\psi_i^{N_i}=(\tau_i,b_i)
$$
for $i=1,2,\cdots,r$.
A global parameter $\psi^N$ is called {\sl conjugate self-dual} if each simple parameter $\psi_i^{N_i}$
occurs in the decomposition of $\psi^N$ is conjugate self-dual in the sense that $\tau_i$ is conjugate self-dual. An irreducible
cuspidal automorphic representation $\tau$ of $G_{E/F}(a)$ is called {\sl conjugate self-dual} if $\tau\cong\tau^*$, where
$\tau^*=\iota(\tau)^\vee$ the contragredient of $\iota(\tau)$ with $\iota$ being the non-trivial element in $\Gamma_{E/F}$ if $E\neq F$;
otherwise, $\iota=1$.
The global parameter $\psi^N$ is called {\sl elliptic} if it is conjugate self-dual and its decomposition
into the simple parameters is multiplicity free, i.e. $\psi_i^{N_i}$ and $\psi_j^{N_j}$ are not equivalent if $i\neq j$
in the sense that either $\tau_i$ is not equivalent to $\tau_j$ or $b_i\neq b_j$.
A global parameter $\psi^N$ in $\wt{\Psi}_\el(N)$ is called {\sl generic} if $b_i=1$ for $i=1,2,\cdots,r$.
The set of generic, elliptic, global Arthur parameters for $G_{E/F}(N)$ is denoted by $\wt{\Phi}_\el(N)$. Hence elements $\phi$ in
$\wt{\Phi}_\el(N)$ are of the form:
\begin{equation}\label{ellggap}
\phi^N=(\tau_1,1)\boxplus\cdots\boxplus(\tau_r,1).
\end{equation}
When $r=1$, the parameters are called {\sl simple}. The corresponding sets are denoted by $\wt{\Psi}_\simp(N)$ and
$\wt{\Phi}_\simp(N)$, respectively. It is clear that the set $\wt{\Phi}_\simp(N)$ is in one-to-one correspondence with the set
of equivalence classes of the conjugate self-dual, irreducible cuspidal automorphic representations of $G_{E/F}(N)(\BA_F)$.
By \cite[Theorem 1.4.1]{A13} and \cite[Theorem 2.4.2]{Mk15}, for a simple parameter $\phi=\phi^a=(\tau,1)$ in $\wt{\Phi}_\simp(a)$
there exists a unique endoscopic datum $(G_\phi,s_\phi,\xi_\phi)$, such that the parameter $\phi^a$ descends to a global parameter
for $(G_\phi,\xi_\phi)$ in sense that there exists an irreducible automorphic representation $\pi$ in $\CA_2(G_\phi)$, whose
Satake parameters are determined by the Satake parameters of $\phi^a$.

When $E\neq F$, $G_\phi=\RU_{E/F}(a)$ is a unitary group,
the $L$-embedding carries a sign $\kappa_a$, which
determines the nature of the {\sl base change} from the unitary group $\RU_{E/F}(a)$ to $G_{E/F}(a)$.
By \cite[Theorem 2.5.4]{Mk15}, the (partial) $L$-function
$$
L(s,(\tau,1),\Asai^{\eta_{(\tau,1)}})
$$
has a (simple) pole at $s=1$ with the sign $\eta_{(\tau,1)}=\kappa_a\cdot(-1)^{a-1}$
(see also \cite[Theorem 8.1]{GGP12} and \cite[Lemma 2.2.1]{Mk15}).
Then the irreducible cuspidal automorphic representation $\tau$ or equivalently the simple generic parameter $(\tau,1)$
is called {\it conjugate orthogonal} if $\eta_{(\tau,1)}=1$ and {\it conjugate symplectic} if $\eta_{(\tau,1)}=-1$, following the terminology of
\cite[Section 3]{GGP12} and \cite[Section 2]{Mk15}.
Here
$L^S(s,(\tau,1),\Asai^+)$ is the (partial) Asai $L$-function of $\tau$ and $L^S(s,(\tau,1),\Asai^-)$ is the (partial) Asai $L$-function of
$\tau\otimes\omega_{E/F}$, where $\omega_{E/F}$ is the quadratic character associated to $E/F$ by the global class field theory.

The sign of a simple global Arthur parameter $\psi=\psi^{ab}=(\tau,b)\in\wt\Psi_2(ab)$ can be calculated following \cite[Section 2.4]{Mk15}.
Fix the sign $\kappa_a$ as before for the
endoscopic datum $(\RU_{E/F}(a),\xi_{\chi_{\kappa_a}})$, the sign of $(\tau,1)$ is
$
\eta_{(\tau,1)}=\eta_\tau=\kappa_a(-1)^{a-1}.
$
Hence the sign of $(\tau,b)$ is given by
$$
\eta_{(\tau,b)}=\kappa_a(-1)^{a-1+b-1}=\kappa_a(-1)^{a+b}=\eta_\tau(-1)^{b-1}.
$$
As in \cite[Equation (2.4.9)]{Mk15}, define
$
\kappa_{ab}:=\kappa_a(-1)^{ab-a-b+1}.
$
Then we have $\kappa_{ab}(-1)^{ab-1}=\eta_\tau(-1)^{b-1}=\eta_{(\tau,b)}$ and hence $\kappa_{ab}=\eta_\tau(-1)^{(a-1)b}$, which
gives the endoscopic datum $(\RU_{E/F}(ab),\xi_{{\kappa_{ab}}})$. More generally, for an elliptic parameter $\psi^N$ as in \eqref{ellgap},
following from \cite[Section 2.1]{Mk15}, each simple parameter $\psi_i^{N_i}$ determines the simple twisted endoscopic datum
$(\RU_{E/F}(N_i),\xi_{\chi_{\kappa_i}})$ with
$
\kappa_i=(-1)^{N-N_i}=\eta_{\tau_i}(-1)^{(a_i-1)b_i},
$
and hence determines the parity of the $\tau_i\in\CA_\cusp(G_{E/F}(a_i))$ for the simple parameter $\psi_i^{N_i}=(\tau_i,b_i)$.

When $E=F$, the notion of conjugate self-dual becomes just self-dual in the usual sense. A self-dual $\tau\in\CA_\cusp(a)$ is called {\it of symplectic type} if the (partial) exterior square $L$-function
$L^S(s,\tau,\wedge^2)$ has a (simple) pole at $s=1$; otherwise, $\tau$ is called {\it of orthogonal type}. In the latter case,
the (partial) symmetric square $L$-function $L^S(s,\tau,\sym^2)$ has a (simple) pole at $s=1$.

More generally, from
\cite[Section 1.4]{A13} and \cite[Section2.4]{Mk15}, for any parameter $\psi^N$ in $\wt{\Psi}_\el(N)$, there is a twisted elliptic
endoscopic datum $(G,s,\xi)\in\wt{\CE}_\el(N)$ such that the set of the global parameters $\wt{\Psi}_2(G,\xi)$ can be identified
as a subset of $\wt{\Psi}_\el(N)$. We refer to \cite[Section 1.4]{A13}, \cite[Section2.4]{Mk15}, and \cite[Section 1.3]{KMSW} for
more constructive description of the parameters in $\wt{\Psi}_2(G,\xi)$.
The elements of $\wt{\Psi}_2(G_n^*,\xi)$, with $N=\Fn^\vee$ and $n=[\frac{\Fn}{2}]$, are of the form
\begin{equation}\label{aps}
\psi=(\tau_1,b_1)\boxplus\cdots\boxplus(\tau_r,b_r).
\end{equation}
Here $N=N_1+\cdots+N_r$ and $N_i=a_i\cdot b_i$, and $\tau_i\in\CA_\cusp(G_{E/F}(a_i))$ and $b_i$ represents the $b_i$-dimensional
representation of $\SL_2(\BC)$. Note that each simple parameter $\psi_i=(\tau_i,b_i)$ belongs to $\wt{\Psi}_2(G_{n_i}^*,\xi_i)$ with
$n_i=[\frac{\Fn_i}{2}]$ and $N_i=\Fn_i^\vee$, for $i=1,2,\cdots,r$; and for $i\neq j$, $\psi_i$ is not equivalent to $\psi_j$. The parity for $\tau_i$ and $b_i$
is discussed as above. The subset of generic elliptic global Arthur parameters in $\wt{\Psi}_2(G_n^*,\xi)$ is denoted by
$\wt{\Phi}_2(G_n^*,\xi)$, whose elements are in the form of \eqref{ellggap}. 

Without lose of generality and for convenience, we choose $\xi$ with sign $\kappa=1$ {\sl throughout this paper}, which is consistent with the choices in  
the Gan-Gross-Prasad conjecture (\cite[Section, Page 35]{GGP}) and in 
the automorphic descents of Ginzburg-Rallis-Soudry (\cite[Page 55]{GRS11}).
That is, when $G^*_n$ is an odd unitary group, its parameters are conjugate orthogonal;
when $G^*_n$ is an even unitary group, its parameters are conjugate symplectic.

The following is   a simplified version of the endoscopic classification
for classical groups established in \cite{A13}, \cite{Mk15}, and \cite{KMSW}.

\begin{thm}[Endoscopic Classification]\label{ds}
For any $\pi\in\CA_\disc(G_n)$, there is a $G_n$-relevant global Arthur parameter $\psi\in\wt{\Psi}_2(G_n^*,\xi)$,
such that $\pi$ belongs to the global Arthur packet, $\wt{\Pi}_{\psi}(G_n)$, attached to the global Arthur parameter $\psi$.
\end{thm}

Following \cite{A13}, \cite{Mk15} and \cite{KMSW},
when $G_n$ is {\sl not} an even special orthogonal group, the multiplicity of $\pi\in\CA_\disc(G_n)$ realizing in the discrete spectrum $L_\disc^2(G_n)$ is
expected to be one. However, when $G_n$ is an even special orthogonal group, the discrete multiplicity of $\pi\in\CA_\disc(G_n)$ could be two.
The multiplicities of the discrete automorphic representations of classical groups depend on the multiplicity property of local Arthur packets,
which is known for the $p$-adic and complex cases for general local Arthur parameters. However,
For the generic local Arthur packets that is what we need in this paper, the multiplicity property holds for all local fields
Hence the expected multiplicities of the automorphic representations in generic global Arthur packets are known.
In the following, we may fix a realization of $\pi\in\CA_\disc(G_n)$ in the discrete spectrum $L_\disc^2(G_n)$, which will be denoted by $\CC_\pi$,
especially when the discrete multiplicity of $\pi$ is two.

Recall the notation from the definition of \cite[Chapter 8]{A13} that
\begin{equation}\label{wtO}
\wt{\RO}(G_n):=\wt{\out}_N(G_n):=\wt{\Aut}_N(G_n)/\wt{\Int}_N(G_n)
\end{equation}
is regarded as the diagonal subgroup of $\wt{\out}_N(G_n(\BA))$.
When $G_n$ is an even special orthogonal group,
one may take $\varepsilon\in \RO_{2n}(F)$ with $\det\varepsilon=-1$ and $\varepsilon^2=I_{2n}$, \label{pg:eps}, such that
the action of $\wt{\RO}(G_n)$ on $\pi$ can be realized as the $\varepsilon$-conjugate on $\pi$,
i.e., $\pi^\varepsilon(g)=\pi(\varepsilon g\varepsilon^{-1})$. Hence the $\wt{\RO}(G_n)$-orbit of $\pi$ has one or two elements.
If $\wt{\RO}(G_n)$ acts freely on $\pi$, following the notation in \cite{A13}, we denote the $\wt{\RO}(G_n)$-torsor of $\pi$ by $\{\pi,\pi_\star\}$. \label{pg:star}
When $G_n$ is not an even special orthogonal group, the group $\wt{\RO}(G_n)$ is trivial, so is its action. Hence in this case,
the $\wt{\RO}(G_n)$-orbit of $\pi$ contains only $\pi$ itself.

When $G_n$ is an even special orthogonal group, an elliptic global Arthur parameter $\psi^N$ as in \eqref{ellgap} may
descend to two different global Arthur parameters $\psi$ and $\psi_\star$ for $G_n$, which form an $\wt{\RO}(G_n)$-orbit.
If the $\wt{\RO}(G_n)$-orbit of $\psi^N$ is an $\wt{\RO}(G_n)$-torsor $\{\psi,\psi_\star\}$, then they define
different global Arthur packets and different global Vogan packets. However, following \cite{A13}, their tensor product $L$-functions with any
cuspidal automorphic representations of general linear groups are the same. We refer to Chapter 8 of \cite{A13} and Section 6 of \cite{AG} for more detailed discussion.

In the rest of this paper, when we say that $\psi^N$ is a global Arthur parameter of an even special orthogonal group $G_n$,
we really mean that $\psi^N$ is identified with either $\psi$ or $\psi_\star$, through a specific twisted endoscopic datum.

\subsection{Fourier coefficients and partitions}\label{sec-fcpt}
For an $F$-quasisplit classical group $G_n^*$ defined by an $\Fn$-dimensional non-degenerate space $(V^*,q^*)$ with the Witt index
$n=[\frac{\Fn}{2}]$, the relation between Fourier coefficients of automorphic forms $\varphi$ of $G_n^*(\BA)$ and the partitions of
type $(\Fn,G_n^*)$ has been discussed with details in \cite{J14} and also in \cite{JL-Cogdell}. We denote by $\CP(\Fn,G_n^*)$ the set of
all partitions of type $(\Fn,G_n^*)$. The set $\CP(\Fn,G_n^*)$ parameterizes the set of all $F$-stable nilpotent adjoint orbits in the
Lie algebra $\Fg_n^*(F)$ of $G_n^*(F)$, and hence each partition $\udl{p}\in\CP(\Fn,G_n^*)$ defines an $F$-stable nilpotent adjoint orbit $\CO_{\udl{p}}^\st$.
For an $F$-rational orbit $\CO_{\udl{p}}\in\CO_{\udl{p}}^\st$, the datum $(\udl{p},\CO_{\udl{p}})$ determines
a datum $(V_{\udl{p}},\psi_{\CO_{\udl{p}}})$
for defining Fourier coefficients as explained in \cite{J14} and \cite{JL-Cogdell}. Here $V_{\udl{p}}$ is a unipotent subgroup of $G_n^*$ and
$\psi_{\CO_{\udl{p}}}$ is a non-degenerate character of $V_{\udl{p}}(\BA)$, which is trivial on $V_{\udl{p}}(F)$ and determined by a given non-trivial character $\psi_F$ of $F\bks\BA$.

For an automorphic form $\varphi$ on $G_n^*(\BA)$, the $\psi_{\CO_{\udl{p}}}$-Fourier coefficient of $\varphi$ is defined by the
following integral:
\begin{equation}\label{fcg}
\CF^{\psi_{\CO_{\udl{p}}}}(\varphi)(g)
:=
\int_{V_{\udl{p}}(F)\bks V_{\udl{p}}(\BA)}\varphi(vg)\psi_{\CO_{\udl{p}}}^{-1}(v)dv.
\end{equation}
Let $N_{G_n^*}(V_{\udl{p}})^\sm$ be the connected component of the semi-simple part of the normalizer of the subgroup $V_{\udl{p}}$ in $G_n^*$. Define
\begin{equation}\label{stab}
H^{\CO_{\udl{p}}}:=\Cent_{N_{G_n^*}(V_{\udl{p}})^\sm}(\psi_{\CO_{\udl{p}}})^\circ,
\end{equation}
the identity connected component of the stabilizer.
It is clear that the $\psi_{\CO_{\udl{p}}}$-Fourier coefficient of $\varphi$, $\CF^{\psi_{\CO_{\udl{p}}}}(\varphi)(g)$, is left
$H^{\CO_{\udl{p}}}(F)$-invariant, smooth when restricted on $H^{\CO_{\udl{p}}}(\BA)$, and of moderate growth on a Siegel set of
$H^{\CO_{\udl{p}}}(\BA)$.

For any $\pi\in\CA_\disc(G_n^*)$, we denote by $\CC_\pi$ a realization of $\pi$ in the discrete spectrum $L^2_\disc(G_n^*)$.
We define $\CF^{{\CO_{\udl{p}}}}(\CC_\pi)$ (or simply $\CF^{{\CO_{\udl{p}}}}(\pi)$ when no confusion is caused) to be the space spanned by all
$\CF^{\psi_{\CO_{\udl{p}}}}(\varphi_\pi)$ with $\varphi_\pi$ running in the space of $\CC_\pi$, and call $\CF^{{\CO_{\udl{p}}}}(\CC_\pi)$
a {\sl $\psi_{\CO_{\udl{p}}}$-Fourier module} of $\pi$. We note that if the discrete multiplicity of $\pi$ is one, it has a unique
$\psi_{\CO_{\udl{p}}}$-Fourier module.
For a given $\pi\in\CA_\disc(G_n^*)$, we denote by $\Fp(\CC_\pi)$ (or simply $\Fp(\pi)$) the subset of
$\CP(\Fn,G_n^*)$ consisting all partitions $\udl{p}$ with the property that the $\psi_{\CO_{\udl{p}}}$-Fourier module,
$\CF^{{\CO_{\udl{p}}}}(\CC_\pi)$, is nonzero for some choice
of the $F$-rational orbit $\CO_{\udl{p}}$  in the $F$-stable orbit $\CO_{\udl{p}}^\st$, and denote by $\Fp^m(\pi)$ (short for $\Fp^m(\CC_\pi)$)
the subset of all maximal members in $\Fp(\pi)$.
In the rest of this paper, we may write $\CF^{{\CO_{\udl{p}}}}(\pi)$ to be $\CF^{{\CO_{\udl{p}}}}(\CC_\pi)$
and $\Fp^m(\pi)$ to be $\Fp^m(\CC_\pi)$ for a discrete realization $\CC_\pi$ of $\pi$.

For a pure inner $F$-form $G_n$ of $G_n^*$, a partition $\udl{p}$ in the set $\CP(\Fn,G_n^*)$ is called {\sl $G_n$-relevant} if
the unipotent subgroup $V_{\udl{p}}$ of $G_n$ as algebraic groups over the algebraic closure $\overline{F}$ is actually defined over $F$.
We denote by $\CP(\Fn,G_n^*)_{G_n}$ the subset of the set $\CP(\Fn,G_n^*)$ consisting of all $G_n$-relevant partitions of
type $(\Fn,G_n^*)$. It is easy to see that the above discussion about Fourier coefficients and Fourier modules
can be applied to all $\pi\in\CA_\disc(G_n)$ and
all $\udl{p}\in\CP(\Fn,G_n^*)_{G_n}$, without change.

Following R. Howe (\cite{H79} and \cite{H81}), N. Kawanaka (\cite{K87}), M\oe glin and Waldspurger
(\cite{MW87}), and M\oe glin (\cite{M96}), one expects that the partitions $\udl{p}$ in $\Fp^m(\pi)$, the $F$-rational orbits
$\CO_{\udl{p}}$ in the $F$-stable orbits $\CO_{\udl{p}}^\st$, and the automorphic spectrum of the Fourier modules
$\CF^{{\CO_{\udl{p}}}}(\pi)$ as representations of $H^{\CO_{\udl{p}}}(\BA)$ carry fundamental information about the given
automorphic representation $\pi$ of $G_n(\BA)$. However, it is usually not easy to obtain explicit information about those
data from the given $\pi$. In reality, we may consider certain special pieces of those data which may already carry enough information for
us to understand the given representation $\pi$ in the theory discussed in this paper.

We consider a family of partitions of type $(\Fn,G_n^*)$, which leads to the so called Bessel-Fourier coefficients of automorphic forms
on $G_n(\BA)$. These partitions are of the form
\begin{equation}\label{bfpt}
\udl{p}_{\ell}=[(2\ell+1)1^{\Fn-2\ell-1}].
\end{equation}
They are of type $(\Fn,G_n^*)$. The partition $\udl{p}_{\ell}$ is $G_n$-relevant if $\ell$ is less than or equal to the $F$-rank $\Fr$ of $G_n$.
For example, if $G_n$ is $F$-anisotropic, then the only $G_n$-relevant partition is the trivial partition $\udl{p}_0=[1^\Fn]$.
For $\pi\in\CA_\disc(G_n)$, and for a partition $\udl{p}_{\ell}\in\CP(\Fn,G_n^*)_{G_n}$,
the Fourier module $\CF^{{\CO_{\udl{p}_\ell}}}(\pi)$ will be called the {\sl $\ell$-th Bessel module} of $\pi$. As explained before, the
$\ell$-th Bessel module $\CF^{{\CO_{\udl{p}_\ell}}}(\pi)$ consists of moderately increasing automorphic functions on
$H^{\CO_{\udl{p}_\ell}}(\BA)$, and is a representation of $H^{\CO_{\udl{p}_\ell}}(\BA)$ by the right translation.

To simplify the notation, we set $\psi_{\CO_\ell}:=\psi_{\CO_{\udl{p}_\ell}}$, $H^{\CO_\ell}:=H^{\CO_{\udl{p}_\ell}}$, and
$\CF^{{\CO_\ell}}(\pi):=\CF^{{\CO_{\udl{p}_\ell}}}(\pi)$.
In this case, the $F$-algebraic group
$H^{\CO_\ell}$ is the classical group $H_{\ell^-}^{\CO_\ell}=\Isom(W^{\CO_\ell},q)^\circ$, where $(W^{\CO_\ell},q)$ is an $\Fl^-$-dimensional
non-degenerate subspace of $(V,q)$ with the properties:
\begin{itemize}
\item $\Fl^-=\Fn-2\ell-1$ and $\ell^-=[\frac{\Fl^-}{2}]$,
\item the product $G_n\times H_{\ell^-}^{\CO_\ell}$ is relevant in the
sense of the Gan-Gross-Prasad conjecture (\cite{GGP12}), and
\item the product $G_n\times H_{\ell^-}^{\CO_\ell}$ is a pure inner $F$-form of an $F$-quasisplit $G_n^*\times H_{\ell^-}^*$.
\end{itemize}
We refer to Section \ref{sec-pif} for more detailed discussion. One may extend the proof of \cite[Theorem 7.3]{GRS11} to the current case and prove
the cuspidality of the maximal Bessel module of $\pi$.

\begin{prop}[Cuspidality of Bessel Modules]\label{cfc}
For any $\pi$ belonging to $\CA_\cusp(G_n)$ with a cuspidal realization $\CC_\pi$, the $\ell$-th Bessel module $\CF^{\CO_\ell}(\CC_\pi)$ of $\CC_\pi$ enjoys the following property:
There exists an integer $\ell_0$ in $\{0,1,\cdots,\Fr\}$, where $\Fr$ is the $F$-rank of $G_n$, such that
\begin{enumerate}
\item  the $\ell_0$-th Bessel module
$\CF^{\CO_{\ell_0}}(\CC_\pi)$ of $\CC_\pi$ is nonzero, but for any $\ell\in\{0,1,\cdots,\Fr\}$ with $\ell>\ell_0$, the $\ell$-th Bessel module
$\CF^{\CO_\ell}(\CC_\pi)$ is identically zero; and
\item the  $\ell_0$-th Bessel module $\CF^{\CO_{\ell_0}}(\CC_\pi)$ is cuspidal in the sense that its constant terms along all the parabolic subgroups of
$H_{\ell_0^-}^{\CO_{\ell_0}}$ are zero.
\end{enumerate}
\end{prop}

We note that when the cuspidal multiplicity of $\pi$ is two, the index $\ell_0$ of $\pi$ in Proposition \ref{cfc} may depend on a particular realization
$\CC_\pi$ of $\pi$ in the cuspidal spectrum $L_\cusp^2(G_n)$. Hence we write $\ell_0=\ell_0(\CC_\pi)$ to be a {\sl first occurrence index} of $\pi$.
Of course, if the cuspidal multiplicity of $\pi$ is one, $\pi$ has the unique first occurrence index, which may be written as $\ell_0=\ell_0(\pi)$.

By Proposition \ref{cfc}, for any $\pi\in\CA_\cusp(G_n)$, the $\ell_0$-th Bessel module $\CF^{\CO_{\ell_0}}(\pi)$, or more precisely,
$\CF^{\CO_{\ell_0}}(\CC_\pi)$, as a representation of
$H_{\ell_0^-}^{\CO_{\ell_0}}(\BA)$, is nonzero and can be embedded as a submodule in the cuspidal spectrum
$L^2_\cusp(H_{\ell_0^-}^{\CO_{\ell_0}})$, and hence can be written as the following Hilbert direct sum of
irreducible cuspidal automorphic representations of $H_{\ell_0^-}^{\CO_{\ell_0}}(\BA)$:
\begin{equation}\label{dcom}
\CF^{\CO_{\ell_0}}(\pi)
=
\sigma_1\oplus\sigma_2\oplus\cdots\oplus\sigma_t\oplus\cdots
\end{equation}
where all $\sigma_i\in\CA_\cusp(H_{\ell_0^-}^{\CO_{\ell_0}})$. By the uniqueness of local Bessel models for classical groups
(\cite{AGRS}, \cite{SZ}, \cite{GGP12} and \cite{JSZ}), it is easy to deduce that the decomposition \eqref{dcom} is multiplicity free.
Furthermore, we have the following conjecture.

\begin{conj}[Generic Summand]\label{bpconj}
Assume that $\pi\in\CA_\cusp(G_n)$ has a $G_n$-relevant, generic global Arthur parameter $\phi\in\wt{\Phi}_2(G_n^*)$. Then there exists a
cuspidal realization $\CC_\pi$ of $\pi$ in $L_\cusp^2(G_n)$ with the first occurrence index
$\ell_0=\ell_0(\CC_\pi)$, such that there exists an $F$-rational orbit
$\CO_{\ell_0}=\CO_{\udl{p}_{\ell_0}}$ in the $F$-stable orbits $\CO_{\udl{p}_{\ell_0}}^\st$ associated to the partition $\udl{p}_{\ell_0}$
with the {\bf Generic Summand Property:}
There exists at least one $\sigma$ in $\CA_\cusp(H_{\ell_0^-}^{\CO_{\ell_0}})$ with
an $H_{\ell_0^-}^{\CO_{\ell_0}}$-relevant, generic global Arthur parameter $\phi_\sigma$ in $\wt{\Phi}_2(H_{\ell_0^-}^*)$, and with a cuspidal
realization $\CC_\sig$ of $\sig$ in $L_\cusp^2(H_{\ell_0^-}^{\CO_{\ell_0}})$,
such that the $L^2$-inner product
$$
\left<\CF^{\psi_{\CO_{\ell_0}}}(\varphi_\pi),\varphi_\sigma\right>_{H_{\ell_0^-}^{\CO_{\ell_0}}}
$$
in the Hilbert space $L^2_\cusp(H_{\ell_0^-}^{\CO_{\ell_0}})$ is nonzero for some $\varphi_\pi\in\CC_\pi$ and $\varphi_\sigma\in\CC_\sigma$.
\end{conj}

It is clear that the Generic Summand Conjecture seeks a refined structure of the generalized branching law for automorphic representations with help of
the endoscopic classification theory. We introduce such a property of invariant theoretic nature into the explicit construction of cuspidal automorphic
modules. Some interesting examples of this nature are obtained through a simple relative trace formula approach by W. Zhang in \cite{ZW-1}.
In Section \ref{sec-mcsro}, we consider the situation that a cuspidal automorphic member $\pi$ in $\wt{\Pi}_\phi(G_n)$ has the property that
$\Fp^m(\pi)=\{\udl{p}_\subr\}$, where $\udl{p}_\subr$ is the partition associated to the subregular nilpotent orbit,
and prove in Proposition \ref{prop-subr} that Conjecture \ref{bpconj} holds for this case. Further discussions on the Generic Summand Conjecture:
its variants and applications can be found in our work (\cite{JZ-Howe} and \cite{JZ-UnU1}). In \cite{JZ-SLD}, we establish the local analogy of
the Generic Summand Conjecture for orthogonal groups defined over $p$-adic local fields of characteristic zero.

\subsection{Rationality of $H_{\ell^-}^{\CO_{\ell}}$}\label{sec-pif}
We are going to make more explicit the parametrization of the $F$-rational orbits $\CO_\ell$ in the $F$-stable orbit $\CO_{\udl{p}_\ell}^\st$ for
the family of partitions $\udl{p}_\ell$, which define the family of Bessel modules. This yields more explicit structure about
the groups $H_{\ell^-}^{\CO_\ell}$.

For the partition $\udl{p}_\ell=[(2\ell+1)1^{\Fn-2\ell-1}]$ of type $(\Fn,G_n^*)$, which is $G_n$-relevant,
the unipotent subgroup $V_\ell=V_{\udl{p}_\ell}$ of $G_n$ can be chosen to consist of all unipotent elements of the form:
\begin{equation}\label{nell}
V_{\ell}=\cpair{v=\begin{pmatrix}z&y&x\\&I_{\Fn-2\ell}&y'\\&&z^{*} \end{pmatrix}\in G_n \mid z\in Z_{\ell}},
\end{equation}
where $Z_{\ell}$ is the standard maximal (upper-triangular) unipotent subgroup of $G_{E/F}(\ell)$. It follows that
the $F$-rational nilpotent orbits $\CO_\ell$ in the $F$-stable nilpotent orbit $\CO^\st_{\udl{p}_\ell}$ are
in one to one correspondence with the $G_{E/F}(1)\times G_{n-\ell}$-orbits of $F$-anisotropic vectors in
$(E^{\Fn-2\ell},q)$, viewed as a subspace of $(V,q)$. Hence the generic character $\psi_{\CO_\ell}$ of $V_\ell(\BA)$
may also be explicitly defined as follows.
Fix a nontrivial character $\psi_F$ of $F\bks \BA$
and define a character $\psi_E$ of $E\bks \BA_{E}$ by
$$
\psi_E(x):=\begin{cases}
\psi_F(x) & \text{ if } E=F;\\
\psi_F(\frac{1}{2}\tr_{E/F}(\frac{x}{\sqrt{\varsigma}}))& \text{ if } E=F(\sqrt{\varsigma}).
\end{cases}
$$
Consider the following identification:
$$
V_{\ell}/[V_{\ell},V_{\ell}]\cong\oplus_{i=1}^{\ell-1}\Fg_{\alpha_{i}}\oplus E^{\Fn-2\ell}.
$$
Let $w_{0}$ be an anisotropic vector in $(E^{\Fn-2\ell},q)$ and define a character $\psi_{\ell,w_{0}}$ of $V_{\ell}(\BA_F)$ by
\begin{equation}\label{chw0}
\psi_{\CO_\ell}(v)=\psi_{\ell,w_{0}}(v):=\psi_E(\sum^{\ell-1}_{i=1}z_{i,i+1}+q(y_{\ell}, w_{0})),
\end{equation}
where $y_{\ell}$ is the last row of $y$ as defined in  \eqref{nell}.
The Levi subgroup of $\CP_{\{\ell+1,\dots,\Fr\}}$
normalizes the unipotent subgroup $V_{\ell}$
and acts on the set of such defined characters $\psi_{\ell,w_{0}}$.
The group $H_{\ell^-}^{\CO_\ell}=H_{\ell^-}^{w_0}$ is the identity connected component of the
stabilizer of  $\psi_{\ell,w_{0}}$, which is give by
\begin{equation}\label{Lellw0}
\cpair{\begin{pmatrix} I_{\ell}&&\\&\gamma&\\&&I_{\ell} \end{pmatrix}\in G_n \mid
\gamma J_{\Fn-2\ell}w_{0}=J_{\Fn-2\ell}w_{0}},
\end{equation}
where $\ell^-=[\frac{\Fl^-}{2}]$ with $\Fl^-:=\Fn-2\ell-1$.
As introduced in Section \ref{sec-ccg}, we may write $V_{(\ell)}=E^{\Fn-2\ell}$ and view $(V_{(\ell)},q)$ as a non-degenerate subspace of
$(V,q)$ under the natural embedding.
Hence the group $H_{\ell^-}^{\CO_\ell}=H_{\ell^-}^{w_0}$ can also be identified as $\Isom(V_{(\ell)}\cap w_0^\perp,q)^\circ$.
Write $W_{\Fl^-}:=V_{(\ell)}\cap w_0^\perp$ so that $(W_{\Fl^-},q)$ is an $\Fl^-$-dimensional non-degenerate subspace of
$(V,q)$. It follows that the dimension $\Fd_0^-$ of its anisotropic kernel of the space $(W_{\Fl^-},q)$ is $\Fd_0\pm 1$,
depending on the choice of $w_0$. Note that $(W_{\Fl^-},q)$ is isometric to $(W^{\CO_\ell},q)$ as introduced in Section \ref{sec-fcpt}.
Define $\Fr^{-}$ to be the Witt index of $(W_{\Fl^-},q)$,
which equals $\Fr-\ell$ or $\Fr-\ell-1$, depending on $\Fd_0^-=\Fd_0-1$ or $\Fd_0^-=\Fd_0+1$, respectively.

For further explicit calculation, we may take the representative $w_0$ of the $F$-anisotropic vectors corresponding to
the $F$-rational nilpotent orbits $\CO_\ell$ in the $F$-stable nilpotent orbit $\CO^\st_{\udl{p}_\ell}$ as follows.
The representative $w_0$ is an $F$-anisotropic vector in the space $(E^{\Fn-2\ell},q)$,
which defines the character $\psi_{\ell,w_0}$.
Under the action of the product $G_{E/F}(1)\times G_{n-\ell}$ on the space $(E^{\Fn-2\ell},q)$, in particular,
on the set of $F$-anisotropic vectors $w_0$, we may choose, if $\ell<\Fr$,
\begin{equation} \label{eq:w0}
w_{0}=y_{\kappa}=e_{\Fr}+(-1)^{\Fn+1}\frac{\kappa}{2}e_{-\Fr}
\end{equation}
for some $\kappa\in F^{\times}$, using the following lemma.

\begin{lem}\label{w0}
If the Witt index of $(V_{(\ell)},q)$ is not zero, i.e. if $\ell<\Fr$,
then there exists an element $g$ in $G_{\Fn-2\ell}(F)=\Isom(V_{(\ell)},q)^\circ(F)$
such that
$$
g\cdot w_0=e_{\Fr}+(-1)^{\Fn+1}\frac{\kappa}{2}e_{-\Fr}
$$
for some $\kappa\in F^{\times}$.
\end{lem}

\begin{proof}
The proof is straightforward. We omit the details here.
\end{proof}

It is clear that if $\ell=\Fr$, then the subspace $(E^{\Fn-2\Fr},q)$ is $F$-anisotropic, and hence is not sensitive to the choice of the $F$-anisotropic vector
$w_0$. The structure of $H_{\ell^-}^{\CO_\ell}$ is summarized in the following proposition.
\begin{prop}\label{coform}
For the partition $\udl{p}_\ell$, let $\CO_\ell$ in $\CO_{\udl{p}_\ell}^\st$ be determined by the $F$-anisotropic vector $w_0$ as in Lemma \ref{w0}.
Then the classical group $H_{\ell^-}^{\CO_\ell}=H_{\ell^-}^{w_0}$ is defined by an
$\Fl^-$-dimensional non-degenerate subspace $(W_{\Fl^-},q)$ of $(V,q)$ with a $(\Fd_0-1)$-dimensional $F$-anisotropic kernel if
$y_{-\kappa}$ belongs to the $G_{E/F}(1)\times G_{n-\ell}$-orbit of a nonzero vector in the $F$-anisotropic kernel $(V_0,q)$ of $(V,q)$;
and is defined by an $\Fl^-$-dimensional non-degenerate subspace $(W_{\Fl^-},q)$ of $(V,q)$ with a $(\Fd_0+1)$-dimensional $F$-anisotropic
kernel if $y_{-\kappa}$ does not belong to the $G_{E/F}(1)\times G_{n-\ell}$-orbit of any nonzero vector in the $F$-anisotropic kernel $(V_0,q)$.
\end{prop}

Following the explicit discussions on pure inner forms of $F$-quasisplit classical groups in \cite{GGP12}, it is easy to obtain the following
proposition.

\begin{prop}\label{piform}
Let $H_m^*$ be an $F$-quasisplit classical group as introduced in Section \ref{sec-ccg}. For any pure inner
$F$-form $H_m$ of $H_m^*$, there exist
\begin{itemize}
\item a classical group $G_n$ defined over $F$ that is a pure inner form of an $F$-quasisplit classical group $G_n^*$, and
\item a datum $(\udl{p}_\ell,\CO_\ell)$ for the Fourier coefficients for automorphic forms on $G_n(\BA)$,
\end{itemize}
such that $m=\ell^-$ and $H_m\cong H_{\ell^-}^{\CO_\ell}$. Moreover, the product $G_n\times H_m$ is a relevant pure inner form of the $F$-quasisplit
$G_n^*\times H_m^*$ in the sense of the Gan-Gross-Prasad conjecture.
\end{prop}

We will recall the Gan-Gross-Prasad conjecture and related notions in Section \ref{sec-ggp}.


\section{The Local Gan-Gross-Prasad Conjecture}\label{sec-ggp}


We recall the local Gan-Gross-Prasad conjecture from \cite{GGP12} for the cases considered in this paper.
The version of the local Gan-Gross-Prasad conjecture, which will be stated as Conjecture \ref{smoconj},
was proved by Waldspurger and by M\oe glin and
Waldspurger in a series of papers (see \cite{W} and \cite{MW12}, for instance) for orthogonal groups over $p$-adic local fields. 
Over archimedean local fields, it is proved by Z. Luo for tempered local $L$-parameters in \cite{Luo-thesis}, but the case of general generic local $L$-parameters is still in progress. 
For unitary groups, Beuzart-Plessis (\cite{BP12} and \cite{BP}) proves the conjecture (Conjecture \ref{smoconj})
for tempered local $L$-parameters over all local fields,
and in \cite{He}, H. He proves the conjecture for discrete representations over $\BR$ via a different approach.
The extension to the generic local $L$-parameters was obtained by Gan and Ichino (\cite{GI}) for $p$-adic local fields, but over archimedean local fields, such an extension remains an open problem, 
as far as the authors knew. 
In the proof of the main conjecture (Conjecture \ref{pmc}), we need the local Gan-Gross-Prasad conjecture (as in Conjecture \ref{smoconj}) for generic local parameters at
all local places as an input. In the process towards the proof of Conjecture \ref{pmc},
we are able to prove the global Gan-Gross-Prasad conjecture (with one direction having an extra assumption).
This will be explained in Sections \ref{sec-gggp} and \ref{sec-wfs}.

\subsection{Generic Arthur parameters}\label{sec-gap}
We consider generic local Arthur parameters for the classical groups considered in this paper. This has been extensively discussed
in \cite{GGP12} and in \cite{MW12}. We recall the basics for the case of orthogonal groups, and refer to \cite{GI} for the case of unitary groups.
Let $G_n^*=\SO(V^*,q^*)$ be the special orthogonal group defined by a non-degenerate, $\Fn$-dimensional quadratic space $(V^*,q^*)$ with $n=[\frac{\Fn}{2}]$, which is $F$-quasisplit. We recall that the generic global Arthur parameters for $G_n^*$ are of the form
\begin{equation}\label{gapgn}
\phi=(\tau_1,1)\boxplus\cdots\boxplus(\tau_r,1)
\end{equation}
as in \eqref{gap}, where $\tau_1,\cdots,\tau_r$ are irreducible unitary cuspidal automorphic representations
of $\GL_{a_1}(\BA),\cdots,\GL_{a_r}(\BA)$, respectively, with required constraints to make $\phi$ a global Arthur parameter of $G_n^*$. As
before, the set of generic global Arthur parameters of $G^*_n$ is denoted by $\wt{\Phi}_2(G_n^*)$. It is known that
the global Arthur packet $\wt{\Pi}_\phi(G_n^*)$ associated to a generic global Arthur parameter $\phi$ contains a generic member. We refer to \cite[Theorem 3.3]{JL-Cogdell} for the detail.

With the assumption of the Ramanujan conjecture for general linear groups, at each local place $\nu$ of $F$, the localization $\phi_{\nu}$
of the generic global Arthur parameter $\phi$ must be a tempered local $L$-parameter for $G_n^*(F_{\nu})$. Hence with possible failure of the
Ramanujan conjecture for general linear groups, one has to figure out the possible structure of the localization $\phi_{\nu}$
of the generic global Arthur parameter $\phi$. We recall from a work of M\oe glin and Waldspurger (\cite{MW12})
for special orthogonal groups and refer to \cite{GI} for the unitary group case.

For each local place $\nu$ of $F$, we denote by $\CW_{F_{\nu}}$ the local Weil group of $F_{\nu}$. The local Langlands group of $F_{\nu}$, which is denoted by
$\CL_{F_{\nu}}$, is equal to the local Weil-Deligne group. Hence the local Langlands group $\CL_{F_{\nu}}$, as usual, may be taken to be $\CW_{F_{\nu}}\times\SL_2(\BC)$ or
equivalently $\CW_{F_{\nu}}\times\SU(2)$ if $\nu$ is a finite local place, and
to be the local Weil group $\CW_{F_{\nu}}$ if $\nu$ is an archimedean local place. The local $L$-parameters for $G_n^*(F_{\nu})$ are of the form
\begin{equation}\label{llp}
\phi_{\nu}\ :\ \CL_{F_{\nu}}\rightarrow {^LG_n^*}
\end{equation}
with the property that the restriction of $\phi_{\nu}$ to the local Weil group $\CW_{F_{\nu}}$ is Frobenius semisimple and trivial on an open
subgroup of the inertia group $\CI_{F_{\nu}}$ of $F_{\nu}$, and the restriction to $\SL_2(\BC)$ is algebraic.
By the local Langlands conjecture for general linear groups (\cite{L89}, \cite{H00}, \cite{HT}, and \cite{Sch}),
the localization $\phi_{\nu}$ at a local place $\nu$ of $F$
of a generic global Arthur parameter $\phi$ is a local $L$-parameter, for which
there exists a datum $(L^*_{\nu},\phi^{L^*}_{\nu},\udl{\beta})$ with the following properties:
\begin{enumerate}
\item $L^*_{\nu}$ is a Levi subgroup of $G^*(F_{\nu})$ of the form
$$
L^*_{\nu}=\GL_{n_1}\times\cdots\times\GL_{n_t}\times G_{n_0}^*,
$$
where $\GL_{n_1},\cdots,\GL_{n_t}$ and $G_{n_0}^*$ depend on the local place $\nu$,
\item $\phi^{L^*_\nu}$ is a local $L$-parameter of $L^*$ given by
$$
\phi^{L^*_\nu}:=\phi_1\oplus\cdots\oplus\phi_t\oplus\phi_0\ :\  \CL_{F_{\nu}}\rightarrow {^LL^*},
$$
where $\phi_j$ is a local tempered $L$-parameter of $\GL_{n_j}$ for $j=1,2,\cdots,t$, and $\phi_0$ is a local tempered $L$-parameter
of $G_{n_0}^*$, with dependence on the local place $\nu$,
\item $\udl{\beta}:=(\beta_1,\cdots,\beta_t)\in \BR^t$, such that $\beta_1>\beta_2>\cdots>\beta_t>0$, which is also dependent of the local
place $\nu$.
\end{enumerate}
With the given datum, following \cite{MW12}, which is expressed in terms of the parabolic induction, one can write
$$
\phi_{\nu}
=(\phi_1\otimes|\cdot|_\nu^{\beta_1}\oplus\phi_1^\vee\otimes|\cdot|_\nu^{-\beta_1})\oplus\cdots\oplus
(\phi_t\otimes|\cdot|_\nu^{\beta_t}\oplus\phi_t^\vee\otimes|\cdot|_\nu^{-\beta_t})\oplus\phi_0.
$$
Following \cite{A13} and also \cite{MW12}, the local $L$-packets can be formed for all local $L$-parameters $\phi_{\nu}$ as displayed above, and
are denoted by $\wt{\Pi}_{\phi_{\nu}}(G_n^*)$. A local $L$-parameter $\phi_\nu$ is called {\sl generic} if the associated local $L$-packet
$\wt{\Pi}_{\phi_\nu}(G_n^*)$ contains a generic member, i.e. a member with a non-zero Whittaker model with respect to a certain Whittaker data for $G_n^*$.
Using the notation of \cite{A13}, the set of all generic local $L$-parameters is denoted by $\wt{\Phi}_\unit^+(G_n^*(F_{\nu}))$.
All the members in any generic local $L$-packet are irreducible and unitary.
It is clear that the localization $\phi_{\nu}$ of a generic global Arthur parameter $\phi$ is
a generic local $L$-parameter according the definition in \cite{MW12} since there exists a generic member in the local $L$-packet
$\wt{\Pi}_{\phi_{\nu}}(G_n^*)$. Hence, following \cite{MW12}, the local Gan-Gross-Prasad conjecture can be formulated for the localization
$\phi_{\nu}$ of all generic global Arthur parameters $\phi$ in $\wt{\Phi}_2(G_n^*)$, which will be discussed in the following section.

\subsection{The local Gan-Gross-Prasad conjecture}\label{sec-lggp}
We are going to recall the local Gan-Gross-Prasad conjecture that was explicitly formulated in \cite{GGP12} for general classical groups.
We discuss the case of orthogonal groups with details and refer the unitary group case to \cite{GGP12} and \cite{BP12}, \cite{BP}, and \cite{GI}
for the details.

Assume that in this section, $F$ is a local field of characteristic zero.
Recall that an $F$-quasisplit special orthogonal group $G_n^*=\SO(V^*,q^*)$ and its pure inner $F$-forms $G_n=\SO(V,q)$
share the same $L$-group $^LG_n^*$.
As explained in \cite[Section 7]{GGP12}, if the dimension $\Fn=\dim V=\dim V^*$ is odd,
one may take $\Sp_{\Fn-1}(\BC)$ to be the $L$-group $^LG_n^*$, and
if the dimension $\Fn=\dim V=\dim V^*$ is even, one may take $\RO_{\Fn}(\BC)$ to be $^LG_n^*$ when $\disc(V^*)$ is not a square in $F^\times$ and
take $\SO_{\Fn}(\BC)$ to be $^LG_n^*$ when $\disc(V^*)$ is a square in $F^\times$.

For a relevant pair $G_n=\SO(V,q)$ and $H_m=\SO(W,q)$, and an $F$-quasisplit relevant pair $G_n^*=\SO(V^*,q^*)$ and $H_m^*=\SO(W^*,q^*)$
as recalled in Section \ref{sec-ccg} from \cite{GGP12},
we are going to discuss the local Langlands parameters for the group $G_n^*\times H_m^*$ and its relevant pure inner $F$-form
$G_n\times H_m$.
As in Section \ref{sec-gap}, we use $\CL_F$ to denote the local Langlands group associated to $F$. We only consider the local Langlands
parameters that satisfy the three properties in Section \ref{sec-gap}:
\begin{equation}\label{llpgh}
\phi\ :\ \CL_F\rightarrow {^LG_n^*}\times{^LH_m^*}.
\end{equation}
Hence they are the localization of the generic global Arthur parameters for the product of the $F$-quasisplit relevant pair
$G_n^*$ and $H_m^*$.
The set of such local Langlands parameters is denoted by $\wt{\Phi}^+_\unit(G_n^*\times H_m^*)$. As in Section \ref{sec-gap},
each local $L$-parameter $\phi$ in $\wt{\Phi}^+_\unit(G_n^*\times H_m^*)$ defines a local $L$-packet
$\wt{\Pi}_\phi(G_n^*\times H_m^*)$.
For any relevant pure inner $F$-form $G_n\times H_m$, if a parameter
$\phi\in \wt{\Phi}^+_\unit(G_n^*\times H_m^*)$
is $G_n\times H_m$-relevant, it defines a local $L$-packet
$\wt{\Pi}_\phi(G_n\times H_m)$,
as in Section \ref{sec-gap}, following \cite{A13} and \cite{MW12}.
If a parameter
$\phi\in\wt{\Phi}^+_\unit(G_n^*\times H_m^*)$
is not $G_n\times H_m$-relevant,
the corresponding local $L$-packet $\wt{\Pi}_\phi(G_n\times H_m)$ is defined to be the empty set. The local Vogan packet for
a local Langlands parameter
$\phi$ belonging to $\wt{\Phi}^+_\unit(G_n^*\times H_m^*)$
is defined to be the union of the local $L$-packets
$\wt{\Pi}_\phi(G_n\times H_m)$
over all pure inner $F$-forms $G_n\times H_m$ of the $F$-quasisplit group $G_n^*\times H_m^*$, and is denoted by
\begin{equation}\label{lvp}
\wt{\Pi}_\phi[G_n^*\times H_m^*].
\end{equation}

In order to state the local Gan-Gross-Prasad conjecture and relevant progress, we have to
introduce the local analogue of the Fourier coefficients as introduced in Section \ref{sec-fcpt}, which is usually called the local Bessel models.
For a given relevant pair $(G_n,H_m)$, take a partition of the form:
$
\udl{p}_\ell=[(2\ell+1)1^{\Fn-2\ell+1}],
$
where $2\ell+1=\dim W^\perp=\Fn-\Fm$.
The $F$-stable nilpotent orbit $\CO_{\udl{p}_\ell}^\st$ corresponding to the partition $\udl{p}_\ell$ defines a unipotent subgroup
$V_{\udl{p}_\ell}$ and a generic character $\psi_{\CO_\ell}$ associated to any $F$-rational orbit $\CO_\ell$ in the $F$-stable orbit $\CO_{\udl{p}_\ell}^\st$.
According to the discussion in Section \ref{sec-pif}, there is an $F$-rational orbit $\CO_\ell$ in the $F$-stable orbit
$\CO_{\udl{p}_\ell}^\st$, such
that the subgroup $H_m=H_{\ell^-}^{\CO_\ell}$ normalizes the unipotent subgroup $V_{\udl{p}_\ell}$ and stabilizes the character $\psi_{\CO_\ell}$.
We define the following subgroup of $G_n$:
$$
R_{\CO_\ell}:=H_m\ltimes V_{\udl{p}_\ell}=H_{\ell^-}^{\CO_\ell}\ltimes V_{\udl{p}_\ell}.
$$
Let $\pi$ be an irreducible admissible representation of $G_n(F)$ and $\sigma$ be an irreducible admissible representation of $H_m(F)$.
The local functionals we considered belong to the following $\Hom$-space
\begin{equation}\label{lfn}
\Hom_{R_{\CO_\ell}(F)}(\pi\otimes\sigma,\psi_{\CO_\ell}).
\end{equation}
This is usually called the space of local Bessel functionals. The uniqueness of local Bessel functionals asserts that
$$
\dim \Hom_{R_{\CO_\ell}(F)}(\pi\otimes\sigma,\psi_{\CO_\ell})\leq 1.
$$
This was proved in \cite{AGRS}, \cite{SZ}, \cite{GGP12}, and \cite{JSZ}. The stronger version in terms of local Vogan packets
for more general classical groups is given as follows, which will be called as the {\sl local Gan-Gross-Prasad conjecture} in the rest of the paper.

\begin{conj}\label{smoconj}
Let $G_n^*$ and $H_m^*$ be a relevant pair of $F$-quasisplit classical groups.
For a given local $L$-parameter $\phi$ in $\wt{\Phi}^+_\unit(G_n^*\times H_m^*)$, the following identity holds:
\begin{equation}\label{smo}
\sum_{\pi\otimes\sigma\in\wt{\Pi}_\phi[G_n^*\times H_m^*]}
\dim \Hom_{R_{\CO_\ell}(F)}(\pi\otimes\sigma,\psi_{\CO_\ell})=1.
\end{equation}
\end{conj}

The known cases of Conjecture \ref{smoconj} can be summarized as follows.

\begin{thm}\label{mwslmo}
Conjecture \ref{smoconj} holds for the following cases:
\begin{enumerate}
\item the relevant orthogonal group pair $G_n^*$ and $H_m^*$ over a p-adic local field $F$, by M\oe glin and Waldspurger in \cite{MW12} for
generic local $L$-parameters;
\item the relevant orthogonal group pair $G_n^*$ and $H_m^*$ over an archimedean local field $F$, by Zhilin Luo in his PhD Thesis \cite{Luo-thesis}, for
tempered local $L$-parameters;
\item the relevant unitary group pair $G_n^*$ and $H_m^*$ over a p-adic local field $F$ or over the real number field $\BR$, by Beuzart-Plessis in
\cite{BP12} and \cite{BP} for tempered local $L$-parameters; and
\item the relevant unitary group pair $G_n^*$ and $H_m^*$ over a p-adic local field $F$, extended by Gan and Ichino in \cite{GI} to generic local $L$-parameters.
\end{enumerate}
\end{thm}

We remark that over the real number field $\BR$, H. He proves in \cite{He} the local Gan-Gross-Prasad conjecture for discrete series representations of unitary groups 
via a different approach.


\section{Bessel Periods and Global Zeta Integrals}\label{sec-bpgzi}


The automorphic analog of the local Bessel functionals is the notion of Bessel periods for a pair of cuspidal automorphic forms or representations. When
one of cuspidal automorphic forms is replaced by a certain Eisenstein series, the Bessel periods become the global zeta integrals that represent
the tensor product $L$-functions. We extend such a construction of the global zeta integrals considered in our previous work (\cite{JZ14}) to a more general
setting that is needed for the main results of this paper. For quasi-split orthogonal groups, a special family of the global zeta integrals was first investigated by Ginzburg, Piatetski-Shapiro and Rallis in \cite{GPSR97}.

\subsection{Global zeta integrals}\label{sec-gzi}
The global zeta integrals that we are going to study are defined for the following three families of classical groups:
\begin{enumerate}
\item $G_{b}=\SO_{2b+1}(V,q)$ and $H_{c}=\SO_{2c}(W,q)$, such that the product $G_{b}\times H_{c}$ is a relevant
pure inner form of an $F$-quasisplit $G_{b}^*\times H_{c}^*$ over $F$.
\item $G_{b}=\SO_{2b}(V,q)$ and $H_{c}=\SO_{2c+1}(W,q)$, such that the product $G_{b}\times H_{c}$ is a relevant
pure inner form of an $F$-quasisplit $G_{b}^*\times H_{c}^*$ over $F$.
\item $G_{b}=\RU_\Fb(V,q)$ and $H_c=\RU_\Fc(W,q)$ with $\Fb$ and $\Fc$ being of different parity and $b=[\frac{\Fb}{2}]$
and $c=[\frac{\Fc}{2}]$, such that the product
$G_b\times H_c$ is a relevant pure inner form of an $F$-quasisplit $G_b^*\times H_c^*$.
\end{enumerate}
In the following, we use the notation that $G_\star=\Isom(V,q)^\circ$ and $H_\square=\Isom(W,q)^\circ$, such that
$G_\star\times H_\square$ is a relevant pure inner form of an $F$-quasisplit $G_\star^*\times H_\square^*$.
Because the global zeta integrals considered in this paper extend what were studied for $F$-quasisplit groups by the authors in \cite{JZ14},
which generalizes the work of Ginzburg, Piatetski-Shapiro and Rallis in \cite{GPSR97}, we will try to follow the arguments and proofs used in \cite{JZ14} and
provide explanation only for the steps that are necessary for understanding of the main results of this section.

In order to formulate the families of global zeta integrals,
we take $\tau$  to be an irreducible unitary automorphic representation
of $G_{E/F}(a)(\BA_F)$ of the following isobaric type:
\begin{equation}\label{tau8}
\tau=\tau_1\boxplus\tau_2\boxplus\cdots\boxplus\tau_r,
\end{equation}
where $\tau_i\in\CA_\cusp(G_{E/F}(a_i))$, $\sum^{r}_{i=1}a_{i}=a$, and $\tau_i\not\cong\tau_j$ if $i\neq j$.

We note that as in \cite{JZ14}, the global unfolding of the global zeta integrals and the unramified calculation for the unramified local zeta integrals in this section only needs to assume that 
$\tau$ is a generic isobaric automorphic representation, which means that some $\tau_i$ and $\tau_j$ could be equivalent.

Take $H_m=\Isom(W,q)^\circ$ with $\dim W=\Fm$ and $m=[\frac{\Fm}{2}]$.
Let $\sigma$ be an irreducible automorphic representation of $H_{m}(\BA)$. Note that we need not assume the cuspidality of $\sig$ in this section.

Let $M_{\hat{a}}=G_{E/F}(a)\times H_{m}$ be an $F$-Levi subgroup of $H_{a+m}$, so that
$P_{\hat{a}}=M_{\hat{a}}U_{\hat{a}}$ is a standard parabolic $F$-subgroup of $H_{a+m}$.
Following from Section I.1.4 in \cite{MW95},
we denote $X_{M_{\hat{a}}}$ to be the group of continuous homomorphisms of $M_{\hat{a}}(\BA)$ into $\BC^\times$ which are trivial on
$M^1_{\hat{a}}:=\cap_{\chi\in\Hom(M_{\hat{a}},\BG_m)}\ker|\chi|$. Since the parabolic subgroup $P_{\hat{a}}$ is maximal, the $\BC$-vector space
$X_{M_{\hat{a}}}$ is one-dimensional.
As in Section 2.2 in \cite{JZ14}, for any $s\in\BC$, we define that $\lam_s(m,h):=|\det m|^s_{\BA_E}$ for $(m,h)\in G_{E/F}(a)(\BA)\times H_{m}(\BA)$.
It is clear that $\lam_s\in X_{M_{\hat{a}}}$. Via the Iwasawa decomposition, we may make the trivial extension of $\lam_s$
to be a function on $H_{a+m}(\BA)$.

For any
\begin{equation}\label{af}
\phi=\phi_{\tau\otimes\sigma}\in \CA(U_{\hat{a}}(\BA)M_{\hat{a}}(F)\bks H_{a+m}(\BA))_{\tau\otimes\sigma}.
\end{equation}
we may set $\phi_s:=\lam_s\cdot\phi$ and form the associated Eisenstein series to be
\begin{equation}\label{es}
E(h,\phi,s)=E(h,\phi_{\tau\otimes\sigma},s)=\sum_{\delta\in P_{\hat{a}}(F)\bks H_{a+m}(F)}\phi_s(\delta g)
\end{equation}
Note that the character $\lam_s$ is normalized as in \cite{Sh10}.
The theory of Langlands on Eisenstein series (\cite{L76} and \cite{MW95}) shows that $E(h,\phi,s)$ converges absolutely for $\Re(s)$ large, has
meromorphic continuation to the complex plane $\BC$, and defines an automorphic form on $H_{a+m}(\BA)$ when $s$ is not a pole.

Take a family of $H_{a+m}$-relevant partitions $\udl{p}_\ell=[(2\ell+1) 1^{\Fm+2a-2\ell-1}]$ of type $(\Fm+2a,H_{a+m}^*)$, with
$\ell\leq a+\Fr_\Fm$, where $\Fr_\Fm:=\Fr(H_m)$ is the $F$-rank of $H_m$ and is the Witt index of $\Fm$-dimensional non-degenerate space
$(W,q)$ that defines $H_m$.
We define the Bessel-Fourier coefficient of the Eisenstein series $E(h,\phi,s)$ on $H_{a+m}(\BA)$:
\begin{equation}\label{bfces}
\CF^{\psi_{\ell,w_{0}}}(E(\cdot,\phi,s))(h)
:=
\int_{N_\ell(F)\bks N_\ell(\BA)}E(nh,\phi,s)\psi_{\ell,w_{0}}^{-1}(n) \ud n,
\end{equation}
where the unipotent subgroup $N_\ell$ of $H_{a+m}$ determined by the partition $\udl{p}_\ell$ is similar to the unipotent subgroup
$V_\ell$ of $G_n$ considered in Section~\ref{sec-fcpt}. We use $N_\ell$ in this section in order to match the notation used in \cite{JZ14},
since we have to recall from there some technical computations of the global zeta integrals.

As in Lemma \ref{w0}, one may choose the representative $w_0$ that defines the character $\psi_{\ell,w_{0}}$ for
Fourier coefficient, as in \eqref{eq:w0}:
\begin{equation} \label{eq:w0a}
w_{0}=y_{\kappa}=e_{a+\Fr_\Fm}+(-1)^{\Fm+1}\frac{\kappa}{2}e_{-(a+\Fr_\Fm)}
\end{equation}
for some $\kappa\in F^{\times}$. Following Proposition \ref{piform}, we have that
$m^-=[\frac{\Fm^-}{2}]$ and $\Fm^-:=2a+\Fm-2\ell-1$ and $\ell< a+\Fr_\Fm$, and that
the stabilizer $G_{m^-}^{w_0}$ and the subgroup $H_m$ form a relevant pair in the sense of the
Gan-Gross-Prasad conjecture (see Section \ref{sec-ggp}). Of course, when $\ell=a+\Fr_\Fm$, the representative $w_0$ is any $F$-anisotropic vector
in the $F$-anisotropic kernel $(W_0,q)$ of $(W,q)$. It is clear that the pair $(G_{m^-}^{w_0},H_m)$ is relevant in this case.
Note that in this case, we must have
$$
\Fm^-=2a+\Fm-2\ell-1=\Fd_0(W)-1,
$$
and the group $G_{m^-}^{w_0}$ is $F$-anisotropic.

As in \cite{JZ14} we define the following semi-direct product of subgroups:
\begin{equation}\label{rgp}
R_{\ell}^{w_{0}}:=G_{m^-}^{w_0}\ltimes N_\ell.
\end{equation}
For an automorphic form  $\varphi_{2a+\Fm}$ on $H_{a+m}(\BA)$, and an automorphic form $\varphi_{\Fm^-}$ on $G_{m^-}^{w_0}(\BA)$,
we define the Bessel period by
\begin{equation}\label{bp}
\CP^{\psi_{\ell,w_{0}}}(\varphi_{2a+\Fm},\varphi_{\Fm^-})
:=
\int_{G_{m^-}^{w_0}(F)\bks G_{m^-}^{w_0}(\BA)}
\CF^{\psi_{\ell,w_{0}}}(\varphi_{2a+\Fm})(g)
\varphi_{\Fm^-}(g) \ud g.
\end{equation}
It is absolutely convergent if one of the automorphic forms $\varphi_{2a+\Fm}$ and $\varphi_{\Fm^-}$ is cuspidal.
In fact, following \cite[Chapter 2]{MW95}, the Fourier coefficients of an automorphic form with moderate
growth is still of moderate growth on the stabilizer of the character that defines the Fourier coefficients. Moreover, if an
automorphic form is rapidly decreasing on a Siegel set, then its Bessel-Fourier coefficient is also rapidly decreasing on a Siegel set of the
stabilizer of the character that defines the Bessel-Fourier coefficient. We refer to \cite[Lemma 10.1]{GRS11}, \cite[Lemma 2.1]{BAS1} and \cite{BAS2}
(and also \cite[Proposition 2.1]{JZ14}) for details.
Otherwise, some regularization may be needed to define this integral in \eqref{bp}. We define the $L^2$-inner product of automorphic functions $\varphi_1$
and $\varphi_2$ over $G_{m^-}^{w_0}(\BA)$ by
\begin{equation}\label{L2ip}
\left<\varphi_1,\varphi_2\right>_{G_{m^-}^{w_0}}
:=
\int_{G_{m^-}^{w_0}(F)\bks G_{m^-}^{w_0}(\BA)}
\varphi_1(x)\ovl{\varphi}_2(x)\ud x
\end{equation}
assuming it converges, where $\ovl{\varphi}(x)=\ovl{\varphi(x)}$ defines the complex conjugation of the function $\varphi$.

For $\pi\in\CA_\cusp(G_{m^-}^{w_0})$,
the {\sl global zeta integral} $\CZ(s,\phi_{\tau\otimes\sigma},\varphi_\pi,\psi_{\ell,w_{0}})$, as in (2.17) of \cite{JZ14},
is defined by the following Bessel period:
\begin{equation}\label{gzi}
\CZ(s,\varphi_\pi,\phi_{\tau\otimes\sigma},\psi_{\ell,w_{0}})
:=
\CP^{\psi_{\ell,w_{0}}}(E(\phi_{\tau\otimes\sigma},s),\varphi_\pi),
\end{equation}
where the Bessel period is written in our convention in this paper by means of the $L^2$-inner product as follows:
$$
\CP^{\psi_{\ell,w_{0}}}(E(\phi_{\tau\otimes\sigma},s),\varphi_\pi)
=
\left<\varphi_\pi,\overline{\CF^{\psi_{\ell,w_{0}}}(E(\phi_{\tau\otimes\sigma},s))}\right>_{G_{m^-}^{w_0}}.
$$
As given in Proposition~2.1 of \cite{JZ14}, $\CZ(s,\phi_{\tau\otimes\sigma},\varphi_\pi,\psi_{\ell,w_{0}})$
converges absolutely and hence is holomorphic at $s$ where the Eisenstein series $E(h,\phi,s)$ has no poles.

From the theory of Eisenstein series and induced representations, it is clear that the Eisenstein series is an automorphic realization of the following
induced representation
\begin{equation}\label{Is}
\RI_s(\tau,\sigma):=\Ind^{H_{a+m}(\BA)}_{P_{\hat{a}}(\BA)}(\tau|\det|^s\otimes\sigma).
\end{equation}
The global zeta integral $\CZ(s,\phi_{\tau\otimes\sigma},\varphi_\pi,\psi_{\ell,w_{0}})$ when $s$ is away from the pole of $E(\phi_{\tau\otimes\sigma},s)$
defines a global Bessel functional $\ell^{\aut}$ belonging to the following $\Hom$-space,
\begin{equation}\label{hom-1}
\Hom_{R_{\ell}^{w_{0}}(\BA)}(\RI_s(\tau,\sigma),\pi^\vee\otimes\psi_{\ell,w_{0}}),
\end{equation}
which has the following restricted tensor product decomposition:
\begin{equation}\label{hom-tensor}
\otimes_\nu
\Hom_{R_{\ell}^{w_{0}}(F_\nu)}(\RI_s(\tau_\nu,\sigma_\nu),\pi_\nu^\vee\otimes\psi_{\ell,w_{0},v}).
\end{equation}
By the local uniqueness of the Bessel models as proved in \cite{AGRS}, \cite{SZ}, \cite{GGP12}, and \cite{JSZ}, this $\Hom$-space has dimension at
most one (at least for $\Re(s)$ large). Hence we expect that this global functional $\ell^\aut$ in \eqref{hom-1} can be written as an Euler product of the local Bessel functionals
\begin{equation}\label{ed-bf}
\ell^\aut=c\cdot\prod_\nu\ell_\nu
\end{equation}
with certain normalization on $\ell_\nu$ when the data are unramified. The global unfolding process (or the global calculation) of the global zeta integral
$\CZ(s,\phi_{\tau\otimes\sigma},\varphi_\pi,\psi_{\ell,w_{0}})$, when $\Re(s)$ is large, is to find explicitly the Euler product factorization in
\eqref{ed-bf} and explicit formula for the local Bessel functionals $\ell_\nu$. This global calculation contains two main steps. The first is to calculate
the Fourier coefficient of the Eisenstein series $\CF^{\psi_{\ell,w_{0}}}(E(\phi_{\tau\otimes\sigma},s))$ and by using the cuspidality of $\pi$ to show
that $\CZ(s,\phi_{\tau\otimes\sigma},\varphi_\pi,\psi_{\ell,w_{0}})$ is equal to an integration associated to the Zariski open dense double coset from
$P_{\hat{a}}\bks H_{a+m}/R_{\ell}^{w_{0}}$. The second step is to show that this remaining integration is in fact an Euler product of local
zeta integrals.

\subsection{Fourier coefficients of Eisenstein series}\label{sec-fces}
This is to calculate the Fourier coefficient
of the Eisenstein series as defined in \eqref{bfces}. Since the calculation is very similar to that in the proof of \cite[Proposition 3.3]{JZ14},
we will not repeat every detail from there, but point out the key steps in the proof.

By assuming that $\Re(s)$ is large, we are able to unfold the Eisenstein series in the integral defining the Fourier coefficient
$\CF^{\psi_{\ell,w_{0}}}(E(\phi_{\tau\otimes\sigma},s))$. This leads to calculate the double coset decomposition
$P_{\hat{a}}\bks H_{a+m}/R_{\ell}^{w_{0}}$. First, we consider the generalized Bruhat decomposition
$P_{\hat{a}}\bks H_{a+m}/P_{\hat{\ell}}$, as a preliminary step towards the calculation.
This decomposition corresponds to the double coset decomposition
$W_{\hat{a}}\bks W_{_{F}\!\Delta}/ W_{\hat{\ell}}$.
Here $W_{_{F}\!\Delta}$ is the Weyl group of $H_{a+m}$ relative to $F$,
which is generated by the simple reflections $s_\alpha$ associated to the roots $\alpha \in {}_{F}\!\Delta$.
Similarly, $W_{\hat{a}}$ is the subgroup of $W_{{}_{F}\!\Delta}$
generated by the simple reflections $s_\alpha$ for $\alpha \in {}_{F}\!\Delta\smallsetminus \{\alpha_a\}$,
so is $W_{\hat{\ell}}$.

We discuss the group $H_{a+m}$ in the following four cases:
\begin{enumerate}
\item If $H_{a+m}$ is the quasi-split even unitary group (i.e.\ $\Fm=2\Fr_\Fm$ and $E=F(\sqrt{\varsigma})$),
then $W_{_{F}\!\Delta}=W(C_{a+\Fr_\Fm})$;
\item If $H_{a+m}$ is a unitary group, but not a quasi-split even unitary group (i.e.\ $2\Fr_\Fm<\Fm$ and $E=F(\sqrt{\varsigma})$),
then $W_{_{F}\!\Delta}=W(B_{a+\Fr_\Fm})$;
\item If $H_{a+m}$ is the split even special orthogonal group (i.e.\ $\Fm=2\Fr_\Fm$ and $E=F$),
then $W_{_{F}\!\Delta}=W(D_{a+\Fr_\Fm})$;
\item If $H_{a+m}$ is a special orthogonal group but not a split even special orthogonal group (i.e.\ $2\Fr_\Fm<\Fm$ and $E=F$),
then $W_{_{F}\!\Delta}=W(B_{a+\Fr_\Fm})$.
\end{enumerate}
Here $W(X_{a+\Fr_\Fm})$ is the Weyl group of the split classical group of type $X$ with rank $a+\Fr_\Fm$.
Following from \cite[Section 3.1]{JZ14}, we put the double coset decomposition
$W_{\hat{a}}\bks W_{{}_{F}\!\Delta}/ W_{\hat{\ell}}$ into three cases for discussions, i.e. {\bf Case (1-1)}, {\bf Case (2-1)} and {\bf Case (2-2)}.
Both {\bf Case (2-1)} and {\bf Case (2-2)} in \cite[Section 3.1]{JZ14} are only for the split even special orthogonal groups.
The result that we are to prove here has already been proved in \cite{JZ14}.
Hence, we assume that $H_{a+m}$ {\it is not the split even special orthogonal group}, which is {\bf Case (1-1)} in \cite[Section 3.1]{JZ14}.
We extend below the proof for {\bf Case (1-1)} in \cite[Section 3.1]{JZ14} to the current general case considered in this section.

In this situation, the double coset decomposition $P_{\hat{a}}\bks H_{a+m}/ P_{\hat{\ell}}$ is in bijection parameterized by
the set of pairs of nonnegative integers
$$
\FE_{a,\ell}=\{\epsilon_{\alpha,\beta}\mid 0\leq \alpha\leq \beta\leq a \text{ and }
a\leq \ell+\beta-\alpha\leq a+\Fr_\Fm\}.
$$
The representatives $\epsilon_{\alpha,\beta}$ are chosen as in \cite[Section 4.2]{GRS11}.
For each double coset $P_{\hat{a}}\epsilon_{\alpha,\beta}P_{\hat{\ell}}$, we take a further decomposition
$P_{\hat{a}}\bks P_{\hat{a}}\epsilon_{\alpha,\beta}P_{\hat{\ell}}/R^{w_0}_{\ell}$, where the group $R^{w_0}_{\ell}$ is defined as in \eqref{rgp}.
It is equivalent to consider the decomposition $P^{\epsilon_{0,\beta}}_{\hat{\ell}}\bks P_{\hat{\ell}}/ R^{w_0}_{\ell}$ with
$P^{\epsilon_{0,\beta}}_{\hat{\ell}}:= \epsilon^{-1}_{0,\beta} P_{\hat{a}}\epsilon_{0,\beta} \cap P_{\hat{\ell}}$.
Let $\CN_{\beta,\ell,w_0}$ be the set of representatives of
$P^{\epsilon_{0,\beta}}_{\hat{\ell}}(F)\bks P_{\hat{\ell}}(F)/ R^{w_0}_{\ell}(F)$,
and set
\begin{equation}\label{eq:W}
W^{\pm}_{\ell,i}={\rm Span}_{E}\{e_{\pm(\ell+1)},e_{\pm(\ell+2)},\dots,e_{\pm(\ell+i)}\}
\end{equation}
for $1\leq i\leq a+\Fr_\Fm-\ell$, which are totally isotropic subspaces of $(W_{\Fm+2a},q)$.
Following the same argument in \cite[Lemmas 3.1 and 3.2]{JZ14}, we can prove that Proposition~3.3 in \cite{JZ14} also holds for the
more general cases in this paper that $H_{a+m}$ may not be $F$-quasisplit.
\begin{prop}\label{bfcc}
For $\Re(s)$ large, the Bessel-Fourier coefficient of the Eisenstein series as in \eqref{bfces},
$\CF^{\psi_{\ell,w_{0}}}(E(\cdot,\phi_{\tau\otimes\sigma},s))(h)$,
is equal to
\begin{equation*}
\sum_{\epsilon_{\beta}}\sum_{\eta}\sum_{\delta}
\int_{N^{\eta}_{\ell}(\BA)\bks N_{\ell}(\BA)}\int_{N^{\eta}_{\ell}(F)\bks N^{\eta}_{\ell}(\BA)}
\phi_s(\epsilon_{\beta}\eta\delta unh)\psi^{-1}_{\ell,w_{0}}(un)\ud u\ud n,
\end{equation*}
where $N^\eta_\ell=N_\ell\cap \eta^{-1}P^{\eps_{0,\beta}}_{\hat{\ell}}\eta$ and
$G^\eta_{m^-}:=G^{w_0}_{m^{-}}\cap \gamma^{-1}P'_w\gamma$; and the summations are over the following representatives:
\begin{itemize}
\item $\eps_{\beta}=\eps_{0,\beta}\in\FE_{a,\ell}^0$, which is the subset of $\FE_{a,\ell}$ consisting of elements with
$\alpha=0$;
\item $\eta=\diag(\eps,\gamma,\eps^{*})$ belongs to $\CN_{\beta,\ell,w_0}^0$, which is the subset of
$\CN_{\beta,\ell,w_0}$ consisting of elements with $\alpha=0$, $\eps=\begin{pmatrix}
0&I_{\ell-t}\\ I_{t}&0
\end{pmatrix}$, and $t=a-\beta$, and has the property that
if $\beta>\max\cpair{a-\ell,0}$, then $\gamma w_{0}$ is orthogonal to $W^{-}_{\ell,\beta}$ for $\gamma\in P'_{w}(F)\bks H_{a+m-\ell}(F)/G^{w_0}_{m^{-}}(F)$
where $P'_w=H_{a+m-\ell}\cap \epsilon^{-1}_{0,\beta} P_{\hat{a}} \epsilon_{0,\beta}$; and
\item $\delta$ belongs to $G^{\eta}_{m^-}(F)\bks G^{w_0}_{m^{-}}(F)$.
\end{itemize}
\end{prop}

\subsection{Euler product decomposition}\label{sec-epd}
Next, we apply the expression in Proposition \ref{bfcc} to the further calculation of the global zeta integral \eqref{gzi}
and have
\begin{eqnarray}
&&\CZ(s,\phi_{\tau\otimes\sig},\varphi_\pi,\psi_{\ell,w_0})\label{eq:P-E1} \\
&=&\sum_{\epsilon_{\beta};\eta}
\int_{h}\varphi_{\pi}(h)\int_{n}\int_{[N^{\eta}_{\ell}]}
\phi_s(\epsilon_{\beta}\eta unh)
\psi^{-1}_{\ell,w_{0}}(un)\ud u\ud n\ud h, \nonumber
\end{eqnarray}
where the summation over $\epsilon_\beta$ and $\eta$ is a finite sum as in Proposition \ref{bfcc}, and
the integration $\int_h$ is over $G^{\eta}_{m^{-}}(F)\bks G^{w_0}_{m^{-}}(\BA)$ and
$\int_n$ is over $N^\eta_\ell(\BA)\bks N_\ell(\BA)$.
Similar to \cite[Lemma 3.4]{JZ14}, for each $\eta=\diag(\epsilon,\gamma,\epsilon^*)$,
if the stabilizer $G^{\eta}_{m^{-}}$ is a proper maximal $F$-parabolic subgroup of $G^{w_0}_{m^{-}}$,
then the summand over such $\eta$ vanishes due to the cuspidality of $\varphi_\pi$.

To proceed with the calculation, we need to study the double coset decomposition $P'_w\bks H_{a+m-\ell}/ G^{w_0}_{m^{-}}$ as given in Proposition \ref{bfcc},
and
extend the calculation in \cite[Section 3.2]{JZ14} to the current setting. With the choice of the $w_0$, it is easy to see that the group $H_{a+m-\ell}$
has its $F$-rank no less than one. We may apply \cite[Proposition 4.4]{GRS11} to the current situation, and show that only one integral associated to the Zariski open dense double coset in $P_{\hat{a}}\bks H_{a+m}/R_{\ell}^{w_{0}}$ remains and all other integrals in the
summation are zero for any choice of data.
In other words, similar to Proposition 3.6 in \cite{JZ14}, we still obtain the following expression for the global zeta integrals,
which have two different forms according to the two cases: $a\leq \ell$ and $a>\ell$.

If $a\leq \ell$, we must have that $\beta=0$, $\eta=\diag\{\epsilon,I_{\Fm},\epsilon^*\}$ with
$\epsilon=\ppair{\begin{smallmatrix}
0&I_{\ell-a} \\ I_{a} &	0
\end{smallmatrix}}$ and
\begin{align}
\CZ(s,\phi_{\tau\otimes\sig},\varphi_\pi,\psi_{\ell,w_0})=&
\int_{G^{w_0}_{m^{-}}(F)\bks G^{w_0}_{m^{-}}(\BA)}\varphi_\pi(h) \int_{N^\eta_\ell(\BA)\bks N_\ell(\BA)} \nonumber \\
&\int_{[N^\eta_\ell]}\phi_s(\eps_{0,\beta}\eta unh)  \psi^{-1}_{\ell,w_0}(un) \ud u \ud n \ud h, \label{eq:gzi-a<l}
\end{align}
where  $[N^\eta_\ell]=N^\eta_\ell(F)\bks N^\eta_\ell(\BA)$.

If $a>\ell$, we must have that $\beta=a-\ell$,
$\eta=\diag\{I_\ell,\gamma_0,I_\ell\}$ and
\begin{align}
\CZ(s,\phi_{\tau\otimes\sig},\varphi_\pi,\psi_{\ell,w_0})=&
\int_{G^\eta_{m^{-}}(F)\bks G^{w_0}_{m^{-}}(\BA)}\varphi_\pi(h) \int_{N^\eta_\ell(\BA)\bks N_\ell(\BA)} \nonumber \\
&\int_{[N^\eta_\ell]}\phi_s(\eps_{0,\beta}\eta unh)  \psi^{-1}_{\ell,w_0}(un) \ud u \ud n \ud h,\label{eq:gzi-a>l}
\end{align}
where
$\gamma_0$ is a representative in the open double coset of
$$P'_w(F)\bks H_{a+m-\ell}(F)/G^{w_0}_{m^{-}}(F),$$
with the property that $\gamma_0 w_0$ is not orthogonal to $W^-_{\ell,\beta}$.

It remains to show that those global integrals are in fact integrals over adelic domains and can be written as Euler products of local zeta integrals.
In order to continue the calculation, we have to recall the relevant calculations in \cite{JZ14} with replacement of notation used here.
Section \ref{sec-al} will deal with the case of $a>\ell$ and hence is for the integral in \eqref{eq:gzi-a>l}. Section \ref{la} will
deal with the case of $a\leq\ell$ and hence is for the integral in \eqref{eq:gzi-a<l}.

\subsection{Case $a>\ell$}\label{sec-al}
We are studying the integral in \eqref{eq:gzi-a>l}.
For convenience, we recall the open coset representative $\epsilon_{0,\beta}$ with $\beta=a-\ell$ in Equation (4.14) in \cite{GRS11},
\begin{equation}\label{eq:eps-beta}
\eps_{0,\beta}= w_q^{\ell}\cdot \begin{pmatrix}
0&I_{a-\ell}&0&0&0\\
0&0&0&0&I_\ell\\
0&0&I_\Fm&0&0\\
I_\ell&0&0&0&0\\
0&0&0&I_{a-\ell}&0	
\end{pmatrix}.
\end{equation}
Note that when  $E=F$, $w_q$ is defined by
$$
\begin{cases}
	-I_{\Fm+2a}  &\text{ if  $\Fm$ is odd,}\\
	\diag\{I_{\frac{\Fm}{2}+a-1},\ppair{\begin{smallmatrix} 1& 0\\ 0 &-1 \end{smallmatrix}},I_{\frac{\Fm}{2}+a-1}\}  &\text{ if $\Fm$ is even and $H_{a+m}$ is not split,}\\
	\diag\{I_{\Fr_\Fm+a-1},\ppair{\begin{smallmatrix}0 &1\\ 1&0 \end{smallmatrix}},I_{\Fr_\Fm+a-1}\}  &\text{ if
$\Fm=2\Fr_\Fm$ and $H_{a+m}$ is split;}\\
\end{cases}
$$
and  when $E=F(\sqrt{\varsigma})$, $w_q=I_{\Fm+2a}$.

First, we write the integral in \eqref{eq:gzi-a>l} as
\begin{equation}\label{eq:gzi-a>l-2}
\CZ(s,\phi_{\tau\otimes\sig},\varphi_\pi,\psi_{\ell,w_0})
=
\int_{G^\eta_{m^{-}}(F)\bks G^{w_0}_{m^{-}}(\BA)}\varphi_\pi(h)\Phi_s(h)\ud h
\end{equation}
where the function $\Phi_s(h)$ is defined, as in \cite[(3.34)]{JZ14}, to be
\begin{equation}\label{eq:Phi-0}
\Phi_s(h):=\int_{N^\eta_\ell(\BA)\bks N_\ell(\BA)}
\int_{[N^\eta_\ell]}\phi_s(\eps_{0,\beta}\eta unh)  \psi^{-1}_{\ell,w_0}(un) \ud u \ud n. 	
\end{equation}
To calculate the function $\Phi_s(h)$, we first consider $(\eps_{0,\beta}\eta)u(\eps_{0,\beta}\eta)^{-1}$,
similar to Section 3.3 of \cite{JZ14}.  Note that $N^{\eta}_{\ell}$ consists of elements of the form
$$
u=\begin{pmatrix}
c&0&0&0&y_6&0&0\\
&I_{\Fr_{\Fm}}&0&0&0&0&0\\
&&I_{a-\ell}&0&0&0&y'_6\\
&&&I_{\Fm-2\Fr_{\Fm}}&0&0&0\\
&&&&I_{a-\ell}&0&0\\
&&&&&I_{\Fr_{\Fm}}&0\\
&&&&&&c^*	
\end{pmatrix}\in N_\ell
$$
where $c\in Z_\ell$. Then the conjugation $(\eps_{0,\beta}\eta)u(\eps_{0,\beta}\eta)^{-1}$ is of the form
$$
\begin{pmatrix}
I_{a-\ell}&y'_6&&&\\
0&c^*&&&\\
&&I_{\Fm}&&\\
&&&c&y_6\\
&&&0&I_{a-\ell}	
\end{pmatrix}.
$$
Hence the stabilizer $(\eps_{0,\beta}\eta) N^\eta_\ell (\eps_{0,\beta}\eta)^{-1}$ as a subgroup of $P_{\hat{a}}$ is in fact contained
in $G_{E/F}(a)$-part of the Levi subgroup of $P_{\hat{a}}$. We denote it by $Z'_\ell$. We may write elements of $Z'_\ell$ as $\hat{z}'$ with
$z'=\ppair{\begin{smallmatrix}
I_{a-\ell}&y\\
0&z	
\end{smallmatrix}}$. Accordingly, the character $\psi^{-1}_{\ell,w_0}(u)$ becomes
$$
\psi_{Z'_\ell,\kappa}(z'):=\psi((-1)^{\Fm+1}\frac{\kappa}{2}z_{\beta,\beta+1}+z_{\beta+1,\beta+2}+\cdots+z_{a-1,a}).
$$
Hence we obtain
\begin{equation}\label{eq:Phi}
\Phi_s(h)=\int_{N^\eta_\ell(\BA)\bks N_\ell(\BA)}\phi^{\psi_{Z'_\ell,\kappa}}_s(\eps_{0,\beta}\eta n h)\psi^{-1}_{\ell,w_0}(n)\ud n, 	
\end{equation}
with
\begin{equation}\label{FC-Z-lk}
\phi^{\psi_{Z'_\ell,\kappa}}_s(h):=\int_{[Z'_\ell]}\phi_s(\hat{z}'h)\psi_{Z'_\ell,\kappa}(z')\ud z'.
\end{equation}

Next, we need to calculate the integration over $G^\eta_{m^{-}}(F)\bks G^{w_0}_{m^{-}}(\BA)$ in \eqref{eq:gzi-a>l} and in \eqref{eq:gzi-a>l-2}, in order
to show that the global zeta integral is in fact an integration over an adelic domain.
Similar to the decomposition in \cite[Equation (3.33)]{JZ14}, we also have the decomposition
$$
G^\eta_{m^{-}}=(G_{E/F}(W^+_{\Fr_{\Fm}+a-1,\beta-1}) \times H(q_{\eta^{-1}W_{(a)}}))\ltimes V_{\beta-1,\eta},
$$
where $W^+_{\Fr_{\Fm}+a-1,\beta-1}$ is defined in \eqref{eq:W}, $W_{(a)}=(W^+_a\oplus W^-_a)^\perp$, and $V_{\beta-1,\eta}$ is the unipotent radical of the stabilizer $G^\eta_{m^{-}}$, as described on Page 573 of \cite{JZ14}.
More precisely, in Proposition \ref{bfcc}, take $\eta=\eta_{I,\gamma_0}$, and then $V_{\beta-1,\eta}$ consists of the elements of the form
\begin{equation}\label{eq:stab-H-L}
\eta^{-1} \begin{pmatrix}
I_{\ell}&&&&&&\\
&I_{\beta-1}&d_{1}&u&v_{1}&v&\\
&&1&0&0&v'_{1}&\\
&&&I_{\Fm}&0&u'&\\
&&&&1&d'_{1}&\\
&&&&&I_{\beta-1}&\\
&&&&&&I_{\ell}
\end{pmatrix}\eta
\end{equation}
with $d_{1}+(-1)^{\Fm+1}\frac{\kappa}{2}v_{1}=0$,
where $d_{1}$ and $v_{1}$ are column vectors of dimension  $\beta-1$.
Let $Z^\eta_{\ell,\beta-1}$ be the maximal unipotent subgroup of $G_{E/F}(W^+_{\Fr_{\Fm}+a-1,\beta-1})$,
consisting of elements of following type:
$$
\eta^{-1}\cdot \diag\{I_\ell,d,I_{\Fm+2},d^*,I_\ell\}\cdot \eta
$$
with $d\in Z_{\beta-1}$.

Write $N^\eta_{\ell,\beta-1}:=Z^\eta_{\ell,\beta-1} V_{\beta-1,\eta}$. It is a unipotent subgroup of $G^{w_0}_{m^-}$ associated to the nilpotent orbit
with partition $[(2(a-\ell-1)+1)1^{\Fm}]$. Fixing the anisotropic vector $y_{-\kappa}$ that defines the character of $N^\eta_{\ell,\beta-1}$, we deduce
that the corresponding stabilizer in $G^{w_0}_{m^-}$ is $\Isom(\eta^{-1}W_{(a)},q)^\circ$. Hence $\Isom(\eta^{-1}W_{(a)},q)^\circ=\eta^{-1}H_m\eta$.
The elements of $N^\eta_{\ell,\beta-1}$ have the form
\begin{equation}\label{eq:N-eta-ell-beta-1}
(\eps_{0,\beta}\eta)^{-1}\begin{pmatrix}
d&d_1&0&u&0&v_1&v\\
&1&0&0&0&0&v'_1\\
&&I_\ell&0&0&0&0\\
&&&I_{\Fm}&0&0&u'\\
&&&&I_\ell&0&0\\
&&&&&1&d'_1\\
&&&&&&d^*	
\end{pmatrix}(\eps_{0,\beta}\eta),	
\end{equation}
where $d\in Z_{\beta-1}$.
Remark that $Z^\eta_{\ell,\beta-1}$ is the set of all matrices of the form \eqref{eq:N-eta-ell-beta-1} with all entries 0 except $d$.
Denote $Z_{\beta,\eta}$ (resp. $C_{\beta-1,\eta}$) to be the subgroup of $N^\eta_{\ell,\beta-1}$
consisting of all matrices in \eqref{eq:N-eta-ell-beta-1} with all entries 0 except $d$ and $d_1$ (resp. with $d=I_{\beta-1}$ and $d_1=0$).
Then $N^\eta_{\ell,\beta-1}=Z_{\beta,\eta}C_{\beta-1,\eta}$.

Similar to \cite[Page 575]{JZ14}, we have the following isomorphism
\begin{equation}\label{eq:H-eta}
C_{\beta-1,\eta}\bks G^\eta_{m^{-}}\cong P^1_\beta\times H^\eta_m
\qquad \text{ ($H^\eta_m:=\eta^{-1}H_m\eta$),}	
\end{equation}
where $H^\eta_m$ is a subgroup of $G^{w_0}_{m^{-1}}$ and
$P^1_\beta$ is the mirabolic subgroup of $G_{E/F}(\beta)$ containing $Z_{\beta,\eta}$.
Continuing with the global zeta integral as displayed in \eqref{eq:gzi-a>l-2}, we obtain that the global zeta integral $\CZ(s,\cdot)$ is equal to
\begin{equation}\label{eq:gzi-a>l-3}
\int_{P^1_\beta(F)H^\eta_m(F)C_{\beta-1,\eta}(\BA)\bks G^{w_0}_{m^-}(\BA)}
\Phi_s(h)\int_{[C_{\beta-1,\eta}]}\varphi_{\pi}(ch)\ud c\ud h.	
\end{equation}
This is the integral similar to that displayed in (3.36) of \cite{JZ14}. We note that there is  a typo in the integration domain in Equation (3.36) of
\cite{JZ14}, and the integral in \eqref{eq:gzi-a>l-3} gives the correct version.

Following closely the argument in \cite{JZ14},
we apply the Fourier expansion on $\varphi_\pi$ along the mirabolic subgroup $P^1_\beta$ repeatedly and obtain the same expansion as that displayed in
Equation (3.38) in \cite{JZ14}.
Plugging so obtained expansion into \eqref{eq:gzi-a>l-3} and combining the integrals, we obtain
\begin{align*}
 &\CZ(s,\phi_{\tau\otimes\sig},\varphi_\pi,\psi_{\ell,w_0})\\
=&\int_{Z_{\beta,\eta}(F)H^\eta_m(F)C_{\beta-1,\eta}(\BA)\bks G^{w_0}_{m^-}(\BA)}
\Phi_s(h)\CF^{\psi^{-1}_{\beta-1,y_{-\kappa}}}(\varphi_\pi)(h)\ud h,
\end{align*}
where $\CF^{\psi^{-1}_{\beta-1,y_{-\kappa}}}(\varphi_\pi)$ is the $(\beta-1)$-th Bessel coefficient with respect to $\psi_{\beta-1,y_{-\kappa}}$ (with
$\beta=a-\ell$),
as defined in \eqref{fcg}, by
\begin{equation}\label{eq:F-pi}
\CF^{\psi^{-1}_{\beta-1,y_{-\kappa}}}(\varphi_\pi)(h)
=
\int_{N^\eta_{\ell,\beta-1}(F)\bks N^\eta_{\ell,\beta-1}(\BA)}\varphi_{\pi}(nh)\psi_{\beta-1,y_{-\kappa}}(n)\ud n.
\end{equation}
Since $\CF^{\psi^{-1}_{\beta-1,y_{-\kappa}}}(\varphi_\pi)$ is left $(Z_{\beta,\eta},\psi^{-1}_{\beta-1,y_{-\kappa}})$-equivariant, $\CZ(s,\cdot)$
is equal to (see \cite[(3.40)]{JZ14})
\begin{equation}\label{eq:zeta-H-a>l}
\int_{H^{\eta}_{m}(F)N^\eta_{\ell,\beta-1}(\BA)\bks G^{w_0}_{m^-}(\BA)}
\CF^{\psi^{-1}_{\beta-1,y_{-\kappa}}}(\varphi_\pi)(h)
\int_{[Z_{\beta,\eta}]}\Phi_s(zh)\psi^{-1}_{\beta-1,y_{-\kappa}}(z)\ud z\ud h.	
\end{equation}
Let us now focus on the inner integral
$\int_{[Z_{\beta,\eta}]}\Phi_s(zh)\psi^{-1}_{\beta-1,y_{-\kappa}}(z)\ud z$, which by definition (as in \eqref{eq:Phi}) is equal to
$$
\int_{[Z_{\beta,\eta}]}\int_{N^\eta_\ell(\BA)\bks N_\ell(\BA)}\phi^{\psi_{Z'_\ell,\kappa}}_s(\eps_{0,\beta}\eta n z h)\psi^{-1}_{\ell,w_0}(n)\ud n~
\psi^{-1}_{\beta-1,y_{-\kappa}}(z)\ud z.
$$
By the definition of $\phi^{\psi_{Z'_\ell,\kappa}}_s$ (see \eqref{FC-Z-lk}), we may combine the two integrals over $[Z_{\beta,\eta}]$ and $[N^{\eta}_\ell]$.
As a subgroup of $P_{\hat{a}}$, $(\eps_{0,\beta}\eta)N^\eta_{\ell}Z_{\beta,\eta}(\eps_{0,\beta}\eta)^{-1}$ consists of elements of the form
$$
\begin{pmatrix}
d&d_1&(y_6)'_{*,*}&&&&\\
0&1&(y_6)'_{\beta,*}&&&&\\
0&0&c^*&&&&\\	
&&&I_{\Fm}&&&\\
&&&&c&(y_6)_{*,\beta}&(y_6)_{*,*}\\
&&&&0&1&d'_1\\
&&&&0&0&d
\end{pmatrix},
$$
where the notation is  the same as in \cite[(3.42)]{JZ14}. It follows that $N^\eta_{\ell}Z_{\beta,\eta}\cong Z_a$, where
$Z_a$ is the maximal upper-triangular unipotent subgroup of $G_{E/F}(a)$, which is regarded canonically as a subgroup of $P_{\hat{a}}$.
Combining the integrals over $[Z_{\beta,\eta}]$ and $[N^{\eta}_\ell]$, we define
\begin{equation}\label{eq:W-tau}
\phi^{Z_{a},\kappa}_s(h):=\int_{[Z_a]}\phi_s(zh)\psi_{Z_{a},\kappa}(z)\ud z,
\end{equation}
where the character $\psi_{Z_{a},\kappa}(z)$ is given by
\begin{equation}\label{eq:whittaker-tau}
\psi(-z_{1,2}-\cdots-z_{\beta-1,\beta}+(-1)^{\Fm+1}\frac{\kappa}{2}z_{\beta,\beta+1}+z_{\beta+1,\beta+2}+\cdots+z_{a-1,a})
\end{equation}
with $\beta=a-\ell$, which is a non-degenerate character of $Z_a$. Hence $\phi_s\mapsto \phi^{Z_{a},\kappa}_s$ can be regarded as
an $H_{a+m}(\BA)$-equivariant isomorphism from the induced representation $\RI_s(\tau,\sigma)$ onto the induced representation
$\RI_s(\CW_\tau,\sigma)$, where $\CW_\tau:=\CW_\tau^{\ovl{\psi}_{Z_{a},\kappa}}$ is the global Whittaker model
of $\tau$ with respect to the non-degenerate character $\ovl{\psi}_{Z_{a},\kappa}$.

For $\Re(s)$ sufficiently large, the integrals we considered here are absolutely convergent, which allow us to switch the order of integration.
After combining the integrals $[Z_{\beta,\eta}]$ and $[N^{\eta}_\ell]$ and by \eqref{eq:Phi}, similar to \cite[(3.41)]{JZ14}, we obtain that
\begin{equation}\label{eq:CJ-0}
\int_{[Z_{\beta,\eta}]}\Phi_s(zh)\psi^{-1}_{\beta-1,y_{-\kappa}}(z)\ud z
=
\int_{U^{-}_{a,\eta}(\BA)}\phi^{Z_{a},\kappa}_s(n\eps_{0,\beta}\eta h)\psi_{(\Fm+a+\ell,a-\ell)}(n)\ud n.
\end{equation}
Here $U^{-}_{a,\eta}$ consists of matrices of the form
\begin{equation}\label{eq:U-j-eta}
\begin{pmatrix}
I_{a-\ell}&&&&\\
0&I_{\ell}&&&\\
0&x'_{2}&I_{\Fm}&&\\
x_{1}&x_{3}&x_{2}&I_{\ell}&\\
0&x'_{1}&0&0&I_{a-\ell}
\end{pmatrix},
\end{equation}
which is a section for the domain of integration, $N^{\eta}_{\ell}\bks N_{\ell}$,
under the adjoint action of $\eps_{0,\beta}\eta$. The character
$\psi_{(\Fm+a+\ell,a-\ell)}$ of $U^{-}_{a,\eta}$ is given by
$$
\psi_{(\Fm+a+\ell,a-\ell)}(n)=\psi(n_{\Fm+a+\ell,a-\ell})
$$
where $n_{\Fm+a+\ell,a-\ell}=(x_{1})_{\ell,a-\ell}$.

Note that the adelic integration over $U^{-}_{a,\eta}$ in
\eqref{eq:CJ} converges absolutely due to the same reason as that of the quasi-split orthogonal group case considered in Appendix II to \S 5 of \cite{GPSR97},
and also that in \cite{S93} and \cite[Theorem 3.1]{S-I}, for instance.
Another way to confirm the absolute convergence is that after taking the absolute value of the integrand, the integral is the product of local intertwining operators, which converges absolutely for $\Re(s)$ sufficiently large.

From \eqref{eq:CJ-0}, we define for $h\in G^{w_0}_{m^{-}}(\BA)$
\begin{equation}\label{eq:CJ}
\RJ_s(\phi_s)(h)
:=
\int_{U^{-}_{a,\eta}(\BA)}\phi^{Z_{a},\kappa}_s(n\eps_{0,\beta}\eta h)\psi_{(\Fm+a+\ell,a-\ell)}(n)\ud n.
\end{equation}
Following a similar argument as in Theorem 3.1 in \cite{S-I} for split special orthogonal groups,
we verify the absolute convergence of $\RJ_s(\phi_s)$ for $\Re(s)$ sufficiently large  in the part of the proof of the absolute convergence of local zeta integrals for more general groups over all local fields in \cite{JSdZ}.
Moreover, the function $\RJ_s(\phi_s)$ enjoys the following property.

\begin{prop}\label{Js}
For $\Re(s)$ sufficiently large, 
the mapping $$\RJ_s\ :\ \phi_s\mapsto \RJ_s(\phi_s)$$ composing with the restriction to $G_{m^-}^{w_0}(\BA)$ gives $G_{m^-}^{w_0}(\BA)$-equivariant homomorphism
from $\RI_s(\tau,\sig)$ as defined in \eqref{Is} to $\RI_s^{w_0}(\psi_{\beta-1,y_{-\kappa}},\sig^{w^\ell_q})$, which is the the following smooth induction
$$
\RI_s^{w_0}(\psi_{\beta-1,y_{-\kappa}},\sig^{w^\ell_q})
:=
\Ind^{G^{w_0}_{m^-}(\BA)}_{N^\eta_{\ell,\beta-1}(\BA)H_m^\eta(\BA)}(\psi_{\beta-1,y_{-\kappa}}\otimes\sig^{w^\ell_q},s)
$$
where the character $\psi_{\beta-1,y_{-\kappa}}$ is given as in \eqref{eq:F-pi}.
\end{prop}

\begin{proof}
For $g\in G_{m^-}^{w_0}(\BA)$, the function $\RJ_s(\phi_s)(g)$ is smooth on $G_{m^-}^{w_0}(\BA)$. The left quasi-invariance with respect to
$(N^\eta_{\ell,\beta-1},\psi_{\beta-1,y_{-\kappa}})$ is clear from the calculation above Proposition \ref{Js}. It remains to check
the left equivariant property for $x\in H^\eta_m(\BA)$. By definition, we have
\begin{eqnarray*}
\RJ_s(\phi_s)(xg)
&=&
\int_{U^{-}_{a,\eta}(\BA)}\phi^{Z_{a},\kappa}_s(n\eps_{0,\beta}\eta xg)\psi_{(\Fm+a+\ell,a-\ell)}(n)\ud n\\
&=&
\int_{U^{-}_{a,\eta}(\BA)}\phi^{Z_{a},\kappa}_s(n\eps_{0,\beta}\eta x\eta^{-1}\eta g)\psi_{(\Fm+a+\ell,a-\ell)}(n)\ud n.
\end{eqnarray*}
Since $\eta x\eta^{-1}$ belongs to $H_m(\BA)$, it is enough to understand the group $\eps_{0,\beta} H_m\eps_{0,\beta}^{-1}$.
According to \eqref{eq:eps-beta} where $\eps_{0,\beta}$ is explicitly given,
$$
\eps_{0,\beta} H_m\eps_{0,\beta}^{-1}
=w_q^\ell H_mw_q^{-\ell}.
$$
It is clear that $\RJ_s(\phi_s)(xg)=\sig^{w^\ell_q}(x)\cdot\RJ_s(\phi_s)(g)$. We are done.
\end{proof}

Note that by \eqref{eq:eps-beta}, the adjoint action of $w_q$ on $\sig$ is trivial except when $H_m$ is an even special orthogonal group.
In this case, $\det(w_q)=-1$  and the adjoint action of $w_q$ is the non-trivial action of $\RO(W_{\Fm},q)/H_m$ on $\sig$.
In other words, $w_q$ restricted to $\RO(W_{\Fm},q)$ is a choice of $\varepsilon$ as defined in Page \pageref{pg:eps}.
For simplicity, denote
\begin{equation}\label{eq:sig'}
\sig':=\sig^{w^\ell_q}.	
\end{equation}
We note that if $H_m$ is an even special orthogonal group and $\ell$ is odd, then $\{\sig,\sig^{w_q}\}$ is an $\wt{\RO}(G)$-orbit of
$\sig$ as discussed in Page \pageref{pg:star}. Therefore, for any fixed $h\in G^{w_0}_{m^-}(\BA)$, the function $\RJ_s(\phi_s)(xh)$, as
a function in $x$, belongs to the space $V_{\sig'}$ of cuspidal automorphic forms, which is the space of the cuspidal automorphic representation
$\sigma$, up to a twist by $\eps_{0,\beta}$. Hence we obtain the following composition of $G^{w_0}_{m^-}(\BA)$-equivariant mappings
\begin{equation}\label{CJ}
\RI_s(\tau,\sigma)\rightarrow \RI_s(\CW_\tau,\sigma)\rightarrow \RI_s^{w_0}(\psi_{\beta-1,y_{-\kappa}},\sig').
\end{equation}

We summarize the calculation above and state the formula for the global zeta integral in the following
\begin{prop}\label{pre-ep-a>l}
With the notation above and for $\Re(s)$ large, the global zeta integral $\CZ(s,\phi_{\tau\otimes\sig},\varphi_{\pi},\psi_{\ell,w_{0}})$ has the
following expression:
\begin{equation}\label{eq:zeta}
\CZ(s,\phi_{\tau\otimes\sig},\varphi_{\pi},\psi_{\ell,w_{0}})=
\int_{g}
\int_{[H^{\eta}_{m}]}\CF^{\psi^{-1}_{\beta-1,y_{-\kappa}}}(\varphi_\pi)(xg)
\RJ_s(\phi_s)(xg)\ud x\ud g
\end{equation}
where $\ud g$ is over $R^{\eta}_{\ell,\beta-1}(\BA)\bks G_{m^-}^{w_0}(\BA)$ with $R^{\eta}_{\ell,\beta-1}:=H^{\eta}_{m}\ltimes N^\eta_{\ell,\beta-1}$, and $[H^{\eta}_{m}]:=H^{\eta}_{m}(F)\bks H^{\eta}_{m}(\BA)$, as defined in \eqref{eq:H-eta}.
\end{prop}

Note that the pairing
\begin{equation*}
\CP^{\psi^{-1}_{\beta-1,y_{-\kappa}}}(\varphi_\pi,\varphi_{\sig'})=\int_{[H^\eta_m]}\CF^{\psi^{-1}_{\beta-1,y_{-\kappa}}}(\varphi_\pi)(x)\varphi_{\sig'}(x)\ud x
\end{equation*}
defines a Bessel period for the pair $(\pi,\sig')$,
where $\varphi_{\sigma'}$ is a cuspidal automorphic form in $\CC_\sig$ under the conjugation of $w^\ell_q$,
and belongs to the space
\[
\Hom_{R^{\eta}_{\ell,\beta-1}(\BA)}(\pi\otimes\sig',\psi^{-1}_{\beta-1,y_{-\kappa}}).	
\]
In this way, the inner integral of the integration formula \eqref{eq:zeta} for the global zeta integral
$\CZ(s,\phi_{\tau\otimes\sig},\varphi_{\pi},\psi_{\ell,w_{0}})$ can be written as 
\begin{equation}\label{inn-zeta}
\int_{[H^{\eta}_{m}]}\CF^{\psi^{-1}_{\beta-1,y_{-\kappa}}}(\varphi_\pi)(xg)\RJ_s(\phi_s)(xg)\ud x=\CP^{\psi^{-1}_{\beta-1,y_{-\kappa}}}(g\ast\varphi_{\pi},\RJ_s(\phi_s)(g)),
\end{equation}
where $g\ast\RJ_s(\phi_s)(1)$ is in $\sig'$ by Proposition \ref{Js}. From this expression, we deduce the following easy, but important vanishing result.

\begin{cor}\label{zeroa>l}
If the Bessel period for $(\pi,\sig')$ is zero, then the global zeta integral $\CZ(s,\phi_{\tau\otimes\sig},\varphi_{\pi},\psi_{\ell,w_{0}})$ is zero
for all choices of data.
\end{cor}

From now on, it is meaningful to assume that the Bessel period $\CP^{\psi^{-1}_{\beta-1,y_{-\kappa}}}$ for $(\pi,\sig')$ is nonzero.
By the uniqueness of local Bessel functionals, which
is proved in \cite{AGRS}, \cite{SZ}, \cite{GGP12}, and \cite{JSZ},
we have the Euler factorization:
$\CP^{\psi^{-1}_{\beta-1,y_{-\kappa}}}=\otimes_\nu\CP_\nu^{\psi^{-1}_{\beta-1,y_{-\kappa}}}$. It follows that 
the integral in \eqref{inn-zeta} can be written as an Euler product of local
Bessel functionals when $\varphi_\pi$ and $\phi_s=\phi_{\tau\otimes\sig,s}$ are factorizable vectors. More precisely, we take $\varphi_\pi=\otimes_\nu\varphi_{\pi_\nu}$ and $\phi_{\tau\otimes\sig,s}=\otimes_\nu\phi_{{\tau_{\nu}}\otimes\sig_{\nu},s}$. Then
\begin{equation}\label{factorizeWtau}
\phi^{Z_{a},\kappa}_s(h)=\prod_\nu f_{\CW_{\tau_\nu}^\kappa\otimes\sig_\nu,s}(h_\nu),
\end{equation}
where $f_{\CW_{\tau_\nu}^\kappa\otimes\sig_\nu,s}$ belongs to the space of induced representation
\begin{equation}\label{IndWtau}
\RI_{s,\nu}(\CW_{\tau_\nu},\sigma_\nu)=\Ind^{H_{a+m}(F_\nu)}_{P_{\hat{a}}(F_\nu)}(|\cdot|^s\CW_{\tau_\nu}\otimes\sig_\nu).
\end{equation}
It follows that
\begin{eqnarray}\label{epa>l} \label{innerbp}
&&\int_{[H^{\eta}_{m}]}\CF^{\psi^{-1}_{\beta-1,y_{-\kappa}}}(\varphi_\pi)(xg)\RJ_s(\phi_s)(xg)\ud x\nonumber\\
&=&
\prod_\nu\CP_\nu^{\psi^{-1}_{\beta-1,y_{-\kappa}}}(g_\nu\ast\varphi_{\pi_\nu},\RJ_{s,\nu}(g_\nu\ast \phi_{s})(1)),
\end{eqnarray}
where at each local place $\nu$, $\CP_\nu^{\psi^{-1}_{\beta-1,y_{-\kappa}}}$ is the unique local Bessel functional up to scalar,
and  $\RJ_{s,\nu}$ is the $\nu$-local twisted Jacquet module associated to the adelic integration over
$U^-_{a,\eta}(\BA)$ that defines $\RJ_s$ in \eqref{eq:CJ}. 

Now, for $\Re(s)$ sufficiently large,  we define the local zeta integral for
this case by
\begin{equation}\label{localzetaa>l}
\CZ_\nu(s,\phi_{\tau\otimes\sig},\varphi_{\pi},\psi_{\ell,w_{0}})
:=
\int_{g_\nu}\CP_\nu^{\psi^{-1}_{\beta-1,y_{-\kappa}}}(g_\nu\ast\varphi_{\pi_\nu},\RJ_{s,\nu}(\phi_{s,\nu})(g_\nu))
\ud g_\nu,
\end{equation}
where the integration is taken over $R^{\eta}_{\ell,\beta-1}(F_\nu)\bks G_{m^-}^{w_0}(F_\nu)$.

\begin{thm}[$a>\ell$]\label{thm:j>l}
Let $E(\phi_{\tau\otimes\sig},s)$ be the Eisenstein series on $H_{m+a}(\BA)$ as in
\eqref{es} and let $\pi$ belong to $\CA_\cusp(G_{m^-}^{w_0})$. Then the global zeta integral
$\CZ(s,\phi_{\tau\otimes\sig},\varphi_{\pi},\psi_{\ell,w_{0}})$
converges absolutely and is holomorphic at $s$ where the Eisenstein series $E(h,\phi,s)$ has no poles.

Assume that $\varphi_\pi=\otimes_\nu\varphi_{\pi_\nu}$ and $\phi_s=\otimes_\nu\phi_{{\tau_{\nu}}\otimes\sig_{\nu},s}=\otimes_\nu\phi_{s,\nu}$ are factorizable vectors, which yields factorization in \eqref{factorizeWtau},
and that the pair $(\pi,\sig')$ has a nonzero Bessel period. Then for the real part of $s$ sufficiently large, it can be written as an Euler product:
$$
\CZ(s,\phi_{\tau\otimes\sig},\varphi_{\pi},\psi_{\ell,w_{0}})=
\prod_{\nu}\CZ_\nu(s,\phi_{\tau\otimes\sig},\varphi_{\pi},\psi_{\ell,w_{0}})
$$
where the local zeta integral $\CZ_\nu(s,\phi_{\tau\otimes\sig},\varphi_{\pi},\psi_{\ell,w_{0}})$ is defined in \eqref{localzetaa>l}.
\end{thm}

Note that this Euler decomposition of the global zeta integral in terms of the local zeta integrals is a more explicit realization of the abstract
Euler decomposition as in \eqref{ed-bf}, and
further properties of the local and global zeta integrals will be discussed in Section \ref{sec-rnbp}.

\subsection{Case $a\leq \ell$}\label{la}
We study briefly the global zeta integral as given in \eqref{eq:gzi-a<l}, which is not needed for the current paper, but for completeness and future applications.
In principle, it is similar to the case of $a>\ell$. We follow the discussion in Section 3.4 in \cite{JZ14} to give necessary steps in order to
show that the global zeta integral can be factorized as an Euler product of local zeta integrals.

First, we have
$$
N^{\eta}_\ell=\cpair{ \begin{pmatrix}
c& 0& 0& 0&0\\
& b& y_4& z_4&0\\
& & I_{\Fm+2a-2\ell}& y'_4&0\\
& & & b^*&0\\
& & & &c^* 	
\end{pmatrix}\colon c\in Z_{a},~b\in Z_{\ell-a}
}.
$$
Write
$$
N_{a,\ell-a}=\cpair{
	\begin{pmatrix}
I_a& 0& 0& 0&0\\
& b& y_4& z_4&0\\
& & I_{\Fm+2a-2\ell}& y'_4&0\\
& & & b^*&0\\
& & & &I_a 	
\end{pmatrix}\colon b\in Z_{\ell-a}
}\subset N_{\ell-a}.
$$
Denote $\psi_{m,\ell-a;y_{\kappa}}$ to be the restriction to the subgroup $N_{a,\ell-a}$ of the character $\psi_{\ell,y_{\kappa}}$.
By the decomposition $N^\eta_{\ell}=Z_{a}N_{a,\ell-a}$,
the inner integration over $[N^\eta_{\ell}]$ in \eqref{eq:gzi-a<l} can be written as
\begin{equation}\label{eq:CF-a<l}
\int_{[N^\eta_{\ell}]}\phi_s(\epsilon_{0,0}\eta u h)\psi^{-1}_{\ell,w_0}(u)\ud u=
\CF^{\psi^{-1}_{m,\ell-a;y_{\kappa}}}(\phi_s^{\psi_{Z_a,\kappa}})(\epsilon_{0,0}\eta  h),	
\end{equation}
where
$$
\phi^{\psi_{Z_a,\kappa}}_s(h)=\int_{[Z_a]}\phi_s(z)\psi_{Z_a,\kappa}(z)\ud z
$$
with
$$
\psi_{Z_a,\kappa}(z)=\psi(z_{1,2}+z_{2,3}+\cdots+z_{a-1,a}).
$$
Here $\CF^{\psi^{-1}_{m,\ell-a;y_{\kappa}}}$ defines a Bessel-Fourier coefficient of $\sig$. As in \eqref{factorizeWtau}, we have
\begin{equation}\label{factorW}
\phi^{Z_{a},\kappa}_s(h)=\prod_\nu f_{W_{\tau_\nu}^\kappa\otimes\sig_\nu,s}(h_\nu),
\end{equation}
with $f_{W_{\tau_\nu}^\kappa\otimes\sig_\nu,s}$ belonging to the space of induced representation
$$
\RI_{s,\nu}(\CW_{\tau_\nu},\sigma_\nu)=\Ind^{H_{a+m}(F_\nu)}_{P_{\hat{a}}(F_\nu)}(|\cdot|^s\CW_{\tau_\nu}\otimes\sig_\nu).
$$

After changing variables, we obtain
\begin{align}
&\CZ(s,\phi_{\tau\otimes\sig},\varphi_{\pi},\psi_{\ell,w_{0}}) \label{eq:gzi-a<l-1} \\
=&\int_{N^\eta_{\ell}(\BA)\bks N_{\ell}(\BA)} \int_{[G^{w_0}_{m^-}]}
\varphi_\pi(h) \CF^{\psi^{-1}_{m,\ell-a;y_{\kappa}}}(\phi^{\psi_{Z_a,\kappa}}_s)(\epsilon_{0,0}\eta hn )\psi^{-1}_{\ell,w_0}(n)\ud h \ud n. \nonumber
\end{align}

Note that the integral \eqref{eq:gzi-a<l-1} is absolutely convergent for $\Re(s)$ sufficient large. The inner integration over $[G^{w_0}_{m^-}]$
converges absolutely because of rapid decay of the cuspidal automorphic form $\varphi_\pi$. The outer integration over the quotient $N^\eta_{\ell}(\BA)\bks N_{\ell}(\BA)$ converges absolutely due to the reason that explained for \eqref{eq:CJ}.
For convenience, we write down explicitly the quotient $N^\eta_{\ell}\bks N_\ell$ and the restriction of $\psi_{\ell,y_{\kappa}}$.
The quotient $N^\eta_{\ell}\bks N_\ell$ is isomorphic to the subgroup consisting of elements
$$
\begin{pmatrix}
I_a&x_1&x_2&x_3&x_4\\
&I_{\ell-a}&0&0&x'_3\\
&&I_{\Fm+2a-2\ell}&0&x'_2\\
&&&I_{\ell-a}&x'_1\\
&&&&I_{a}	
\end{pmatrix}.
$$
The restriction of $\psi_{\ell,y_\kappa}$ is $\psi((x_1)_{a,1})$.

It is clear that the inner integration in the variable $h$ in \eqref{eq:gzi-a<l-1} gives a Bessel period for the pair $(\sigma, \pi)$.
Hence we obtain the following.

\begin{cor}\label{zeroa<l}
If the Bessel period for $(\sig,\pi)$ is zero, then the global zeta integral $\CZ(s,\phi_{\tau\otimes\sig},\varphi_{\pi},\psi_{\ell,w_{0}})$ is zero
for all choices of data.
\end{cor}

By the uniqueness of the local Bessel models, this Bessel period may be written as an Euler product of local Bessel functionals for factorizable
input data. More precisely, we take $\varphi_\pi=\otimes_\nu\varphi_{\pi_\nu}$ and $\phi_s=\otimes_\nu\phi_{s,\nu}$ and write
\begin{equation}\label{epa<l}
\int_{[G^{w_0}_{m^-}]}
\varphi_\pi(h)\RJ_s(\phi_s)(hn) \ud h
=
\prod_\nu\CP_\nu^{\psi^{-1}_{m,\ell-a;y_\kappa,\nu}}(n_\nu\ast f_{W_{\tau_\nu}^\kappa\otimes\sig_\nu,s},\varphi_{\pi_\nu}),
\end{equation}
where $\RJ_s(\phi_s)(hn):=\CF^{\psi^{-1}_{m,\ell-a;y_{\kappa}}}(\phi^{\psi_{Z_a,\kappa}}_s)(\epsilon_{0,0}\eta hn )$, and for each local place $\nu$,
$\CP_\nu^{\psi^{-1}_{m,\ell-a;y_\kappa,\nu}}$ is the unique functional, up to scalar, in the space
$$
\Hom_{G^{w_0}_{m^-}(F_\nu)\ltimes N_{a,\ell-a}(F_\nu)}(\pi_\nu\otimes\sigma_\nu, \psi^{-1}_{m,\ell-a;y_{\kappa},\nu}).
$$
In this way, we define the local zeta integral by
\begin{equation}\label{localzetaa<l}
\CZ_\nu(s,\phi_{\tau\otimes\sig},\varphi_\pi,\psi_{\ell,w_0})
:=
\int_{n_\nu}\CP_\nu^{\psi^{-1}_{m,\ell-a;y_\kappa,\nu}}(n_\nu\ast f_{W_{\tau_\nu}^\kappa\otimes\sig_\nu,s},\varphi_{\pi_\nu})
\psi^{-1}_{\ell,w_0,\nu}(n_\nu)\ud n_\nu,
\end{equation}
where the integration is taken over $N^\eta_{\ell}(F_\nu)\bks N_{\ell}(F_\nu)$,
and obtain the following

\begin{thm}[$a\leq\ell$]\label{thm:j<l}
With the notation as in Theorem \ref{thm:j>l}, the global zeta integral
$\CZ(s,\phi_{\tau\otimes\sig},\varphi_{\pi},\psi_{\ell,w_{0}})$
converges absolutely and is holomorphic at $s$ where the Eisenstein series $E(h,\phi,s)$ has no poles.

Assume that $\varphi_\pi=\otimes_\nu\varphi_{\pi_\nu}$ and $\phi_s=\otimes_\nu\phi_{{\tau_{\nu}}\otimes\sig_{\nu},s}=\otimes_\nu\phi_{s,\nu}$ are factorizable vectors, which yields the factorization in \eqref{factorW}, and that the pair $(\sig,\pi)$ has a nonzero Bessel period. Then for the real part of $s$ sufficiently large, it can be written as an Euler product:
$$
\CZ(s,\phi_{\tau\otimes\sig},\varphi_{\pi},\psi_{\ell,w_{0}})=\prod_{\nu}\CZ_\nu(s,\phi_{\tau\otimes\sig},\varphi_{\pi},\psi_{\ell,w_{0}})
$$
where the local zeta integral $\CZ_\nu(s,\phi_{\tau\otimes\sig},\varphi_{\pi},\psi_{\ell,w_{0}})$ is defined in \eqref{localzetaa<l}.
\end{thm}
Note that this Euler decomposition of the global zeta integral in terms of the local zeta integrals is a more explicit realization of the abstract
Euler decomposition as in \eqref{ed-bf}. Since this case is not directly used in this paper, we refer more detailed explanation on a special case
to Section 3.4 {\cite{JZ14}}.

\subsection{Unramified local zeta integrals and local $L$-factors}\label{sec-ulzi-llf}
We define the local $L$-factors for the cases under consideration and recall the results from the unramified computations of the local zeta integrals
as considered in \cite{JZ14}, \cite{S-I}, \cite{S-II} and \cite{JSdZ}.

Note that the group $G_{m^-}^{w_0}$
from the construction in Section \ref{sec-pif} yields all the groups $G_{n}$ as listed in the beginning of this section.
Hence there exists a datum such that
$G_{m^-}^{w_0}$ is isomorphic to a given $G_n$ over $F$. From now on, we assume that $\pi\in\CA_\cusp(G_{m^-}^{w_0})$ and
$\sigma\in\CA_\cusp(H_m)$ have generic global Arthur parameters, respectively.

As in \eqref{tau8}, we have
$\tau=\tau_1\boxplus\tau_2\boxplus\cdots\boxplus\tau_r$,
which is an irreducible generic isobaric automorphic representation of $G_{E/F}(a)(\BA_F)$. We define
\begin{equation}\label{cl}
\CL(s,\tau_\nu,\pi_\nu,\sig_\nu;\rho)
=\frac{L(s+\frac{1}{2},\tau_{\nu}\times\pi_\nu)}
{L(s+1,\tau_{\nu}\times\sig_\nu)L(2s+1,\tau_{\nu},\rho)},
\end{equation}
where
$\rho=\wedge^2$ if $H_{m+a}$ is an even orthogonal group;
$\rho=\sym^2$ if $H_{m+a}$ is an odd orthogonal group;
$\rho=\Asai\otimes \xi^m$ if $H_{m+a}$ is a unitary group.

Some remarks on the local $L$-functions are in order. At archimedean local places or at unramified local places, the local $L$-functions in
\eqref{cl} are well defined. The main concern here is at the ramified finite local places. Formally, one may take the G.C.D. of the
ramified local zeta integrals as the definition or take the one from the normalization of the local intertwining operators from the
Eisenstein series in the global zeta integrals. This of course needs the full theory of the local zeta integrals, which is not available
at this moment for general representations $\pi$ and $\sigma$. On the other hand,
since both $\pi$ and $\sigma$ are assumed to be cuspidal and to have generic
global Arthur parameters (\cite[Chapter 9]{A13} and \cite{KMSW}), we may follow \cite{A13}, \cite{Mk15} and \cite{KMSW} to define
the local $L$-functions in \eqref{cl} at ramified finite local places in terms of the local $L$-functions of the corresponding
localization of the global Arthur parameters. We refer to \cite{M12} for discussion with more general parameters when the groups are $F$-quasisplit.

We note that only when $H_m$ is an even special orthogonal group,
the twisted representation $\sig'_\nu$ (see \eqref{eq:sig'}) may not be equivalent to $\sig_\nu$ if $w^{\ell}_{q}\ne I$.
However, their corresponding local $L$-parameters are $\RO_\Fm(\BC)$-conjugate, since $H^\vee_m(\BC)=\SO_{\Fm}$ is the complex dual group of $H_m$.
It follows that $\CL(s,\tau_\nu,\pi_\nu,\sig_\nu;\rho)$ and the local $L$-functions $L(s,\tau_\nu\times\sig_\nu)$ are the same when
the local factors of $\sig_\nu$ replaced by those of $\sig'_\nu$.

Recall that $\sig'=\sig^{w^{\ell}_{q}}$ in \eqref{eq:sig'} when $a>\ell$,
and also denote $\sig'=\sig$ when $a\leq \ell$ for notational consistence.
The Euler products in Theorems \ref{thm:j>l} and \ref{thm:j<l} can be uniformly rewritten as
$$
\CZ(s,\phi_{\tau\otimes\sig'},\varphi_{\pi},\psi_{\ell,w_{0}})
=
\prod_{\nu}\CZ_\nu(s,\phi_{\tau\otimes\sig'},\varphi_{\pi},\psi_{\ell,w_{0}}).
$$

Next, we state the result of unramified calculation for the local zeta integrals. The full detail of the computation in this generality
will appear in our joint work with D. Soudry (\cite{JSdZ}), based on the idea of Soudry as developed in his work (\cite{S-ICM, S-I, S-II}). Many special cases have been treated in \cite{GPSR97} and \cite{JZ14}, for instance.

\begin{thm}[Unramified Computation]\label{urmL}
With all data being unramified,
the local unramified zeta integral $\CZ_\nu(s,\phi_{\tau\otimes\sig'},\varphi_{\pi},\psi_{\ell,w_{0}})$
has the following expression:
\begin{eqnarray}
\CZ_\nu(s,\phi_{\tau\otimes\sig'},\varphi_{\pi},\psi_{\ell,w_{0}})=\CL(s,\tau_\nu,\pi_\nu,\sig_\nu;\rho)
\end{eqnarray}
where $f_{W_{\tau_\nu}^\kappa\otimes\sig_\nu,s}$, $\phi_{\sig_\nu}$ and $\varphi_{\pi_\nu}$ are the spherical vectors,
which are so normalized that the corresponding spherical functions are equal to $1$ at the identity element.
\end{thm}

Let $S$ be a finite set of places consisting of  all ramified places of relevant data and all archimedean places such that for $\nu\notin S$ all data are unramified.
Following Theorem \ref{urmL}, we obtain that
\begin{eqnarray}\label{formula1}
\CZ(s,\phi_{\tau\otimes\sig'},\varphi_{\pi},\psi_{\ell,w_{0}})
&=&\prod_{\nu\in S}\CZ_\nu(s,\cdot)\cdot\prod_{\nu\notin S}\CZ_\nu(s,\cdot)\nonumber\\
&=&\CZ_S(s,\cdot)\cdot\CL^S(s,\tau,\pi,\sig;\rho)
\end{eqnarray}
Here we set $\CZ_\nu(s,\cdot):=\CZ_\nu(s,\phi_{\tau\otimes\sig'},\varphi_{\pi},\psi_{\ell,w_{0}})$, $\CZ_S(s,\cdot):=\prod_{\nu\in S}\CZ_\nu(s,\cdot)$,
and $\CL^S(s,\cdot):=\prod_{\nu \notin S}\CL(s,\cdot_\nu)$.

\subsection{On even special orthogonal groups}\label{sec-esog}
We explain with more details the twists that we get in the case of even special orthogonal groups.
We follow the notation from Section 2 of \cite{JZ14}.
First, $P_j$ is the standard parabolic subgroup of $\SO_{4a+2m}$ with Levi subgroup isomorphic to $\GL_{\ell}\times\SO_{4n+2m-2\ell}(W_\ell)$.
Here $\SO_{4n+2m-2\ell}(W_\ell)$ preserves the quadratic space
$$
W_\ell=\Span\{e^{\pm}_{\ell+1},\dots,e^{\pm}_{\Fr_{\Fm}-1},e^{\pm}_{\Fr_{\Fm}}\}\oplus V_0.
$$
$G$ is the stabilizer of $y_{\kappa}$ preserving the quadratic space
$$
W_\ell\cap y^{\perp}_{\kappa}=\Span\{e^{\pm}_{\ell+1},\dots,e^{\pm}_{\Fr_{\Fm}-1},y_{-\kappa}\}\oplus V_0.
$$
The anisotropic kernel of $W_\ell\cap y^{\perp}_{\kappa}$ is a subspace of $Fy_{-\kappa}\oplus V_0$.
The inner period over $\pi$ is arisen from the open double coset of $P_j\bks \SO_{4n+2m}/G\cdot N_\ell$.
Recall that we choose the following representative $\eta$ for this coset
$$
\begin{pmatrix}
I_\ell&&&&&&\\
&0&I_{j-\ell}&&&&\\
&I_{\Fr_{\Fm}-j}&0&&&&\\
&&&I_{V_0}&&&\\
&&&&0&I_{\Fr_{\Fm}-j}&\\
&&&&I_{j-\ell}&0&\\
&&&&&&I_\ell	
\end{pmatrix}.
$$
Recall that $G^\eta=G\cap (\eta^{-1}P_j\eta)$.
Then under the conjugation of $\eta$,  $\SO_{2m}(\eta^{-1}W_j)$ is the subgroup of $G^\eta$, which preserves
$$
\Span\{e^{\pm}_{\ell+1},\dots,e^{\pm}_{\Fr_{\Fm}-j+\ell}\}\oplus V_0.
$$
For example, when $j=\ell+1$, then $(G,\SO_{2m}(\eta^{-1}W_j))$ is the Gross-Prasad pair.
That is, $\SO_{2m}(\eta^{-1}W_j)$ is the stabilizer of the anisotropic vector $y_{-\kappa}$.
Thus $\SO_{2m}(\eta^{-1}W_j)$ is isomorphic to $\SO_{2m}(W_j)$.


\section{Reciprocal Non-vanishing of Bessel Periods}\label{sec-rnbp}


The reciprocal non-vanishing of Bessel periods is to address the non-vanishing property of the Bessel periods for the pair $(\CE_{\tau\otimes\sig},\pi)$
and for the pair $(\pi,\sig)$, where $\CE_{\tau\otimes\sig}$ is the iterated residue at $s=\frac{1}{2}$ of the Eisenstein series
$E(\cdot,\phi_{\tau\otimes\sigma},s)$ as defined in \eqref{es}, and $\sig$ may have to be replaced by $\sig'$ as in \eqref{eq:sig'}.

\subsection{Residue of the Eisenstein series}
We recall the Eisenstein series $E(\cdot,\phi_{\tau\otimes\sigma},s)$ from \eqref{es}.
Assume as before that $\sig\in\CA_\cusp(H_m)$ has a generic global Arthur parameter $\phi_\sig$, and a cuspidal realization $\CC_\sig$ in case when the
cuspidal multiplicity is not one.
Let $\tau=\tau_1\boxplus\tau_2\boxplus\cdots\boxplus\tau_r$ be the irreducible unitary generic isobaric automorphic representation of $G_{E/F}(a)(\BA_F)$
associated to distinct $\tau_1, \tau_2,\cdots,\tau_r$, as given in \eqref{tau8}. Assume that the generic global Arthur parameter $\phi_\tau$ determined by
$\tau$ has a different parity with $\phi_\sig$. It follows that the $L$-function
$$
L(s,\tau\times\sig)=L(s,\phi_\tau\times\phi_\sig),
$$
as in \cite{A13}, is holomorphic at $s=\frac{1}{2}$.

We calculate the constant terms of $E(\cdot,\phi_{\tau\otimes\sigma},s)$.
According to the cuspidal support of $E(\cdot,\phi_{\tau\otimes\sigma},s)$,
among all of the constant terms that are not identically zero, the term that carries the highest order of the pole at $s=\frac{1}{2}$ is given by
the following global intertwining operator integral
\begin{equation}\label{ioi}
\CM(\omega_0,\tau\otimes\sigma,s)(\phi_{\tau\otimes\sigma})(g)
:=
\int_{U_{\hat{a}}(\BA)}
\lambda_s \phi_{\tau\otimes\sigma}(\omega_0^{-1}ng)dn,
\end{equation}
where $U_{\hat{a}}$ is the unipotent radical of the standard maximal parabolic subgroup $P_{\hat{a}}=M_{\hat{a}}U_{\hat{a}}$ with $M_{\hat{a}}=G_{E/F}(a)\times H_m$, and
the Weyl group element $\omega_0$ takes $U_{\hat{a}}$ to is opposite $U^-_{\hat{a}}$.
Following the calculation of Langlands (\cite{L71} and also \cite{Sh10}),
one may choose the factorizable section $\phi=\phi_{\tau\otimes\sigma}$ so that
$\CM(\omega_0,\tau\otimes\sigma,s)(\phi)$ can be written as
\begin{equation}\label{iof}
\CM(\omega_0,\tau\otimes\sigma,s)_S(\phi_S)\cdot
\frac{L^S(s,\tau\times\sigma)L^S(2s,\tau,\rho)}{L^S(1+s,\tau\times\sigma)L^S(1+2s,\tau,\rho)}\lam_{-s}\phi^S_{\omega_0(\tau\otimes\sig)},	
\end{equation}
where $\CM(\omega_0,\phi_{\tau\otimes\sigma},s)_S$ is the finite product of the local intertwining operators over $\nu\in S$, and
$\phi_S=\prod_{\nu\in S}\phi_{\tau_\nu\otimes\sig_\nu}$ and $\phi^S=\otimes_{\nu\not\in S}\phi_{\tau_\nu\otimes\sig_\nu}$.
Since the cuspidal automorphic representation $\sigma$ is assumed to have a generic global Arthur parameter, we define, following \cite{A13},
the local $L$-factors at $\nu\in S$ in terms of $\tau$ and the generic global Arthur parameter of $\sigma$. Then
we take the Shahidi normalization by defining, for each $\nu\in S$,
\begin{equation}\label{nliodfn}
\CN(\omega_0,\tau\otimes\sigma,s)_{\nu}
:=
\beta_{\nu}(s,\tau,\sigma,\psi_F;\rho)\cdot\CM(\omega_0,\tau\otimes\sigma,s)_\nu,
\end{equation}
where the local normalizing factor $\beta_{\nu}(s,\tau,\sigma,\psi_F;\rho)$ is defined to be
\begin{equation}\label{lnf}
\frac{L_{\nu}(1+s,\tau\times\sig)L_{\nu}(1+2s,\tau,\rho)\epsilon_{\nu}(s,\tau\times\sigma,\psi_F)\epsilon_{\nu}(2s,\tau,\rho,\psi_F)}
{L_{\nu}(s,\tau\times\sigma)L_{\nu}(2s,\tau,\rho)}.
\end{equation}
Hence we obtain the following:
$$
\CM(\omega_0,\tau\otimes\sigma,s)
=
\frac{\CN(\omega_0,\tau\otimes\sigma,s)\cdot L(s,\tau\times\sigma)L(2s,\tau,\rho)}
{L(1+s,\tau\times\sigma)L(1+2s,\tau,\rho)\epsilon(s,\tau\times\sigma)\epsilon(2s,\tau,\rho)}.
$$
We call $\CN(\omega_0,\tau\otimes\sigma,s)_{\nu}$ the normalized local intertwining operators.

\begin{thm}\label{nlio}
Let $\tau=\tau_1\boxplus\cdots\boxplus\tau_r$ be the irreducible isobaric automorphic representation of $G_{E/F}(a)(\BA)$ as in \eqref{tau8},
and $\sigma\in\CA_\cusp(H_m)$ of a generic global Arthur parameter $\phi_\sigma$. Then,
for each local place $\nu$ of $F$, the normalized local intertwining operator $\CN(\omega_0,\tau\otimes\sigma,s)_{\nu}$
from the induced space $\Ind^{H_{a+m}(F_\nu)}_{P_{\hat{a}}(F_\nu)}|\cdot|^s\tau_\nu\otimes\sig_\nu$ to $\Ind^{H_{a+m}(F_\nu)}_{P_{\hat{a}}(F_\nu)}|\cdot|^{-s} \tau^*_\nu\otimes\sig_\nu$
is holomorphic and nonzero for $\Re(s)\geq\frac{1}{2}$,
where $\tau^*_\nu=\iota(\tau_\nu)^\vee$ is
the contragredient of $\iota(\tau)$.
\end{thm}

We remark that when $H_m$ is $F$-quasisplit, much stronger result than what stated in Theorem \ref{nlio}
can be proved when $\sigma$ is also assumed to be generic (see \cite{CKPSS04}, for instance and also see \cite{M12}). We will prove
Theorem \ref{nlio} in Appendix B.

By Theorem \ref{nlio}, we have that the normalized global intertwining operator
$\CN(\omega_0,\tau\otimes\sigma,s)$ is holomorphic and nonzero for $\Re(s)\geq\frac{1}{2}$. We are able to study the
order of the pole at $s=\frac{1}{2}$ of the global intertwining operator $\CM(\omega_0,\tau\otimes\sigma,s)$.

In fact, it is easy to write
$$
L(2s,\tau,\rho)=\prod_{j=1}^rL(2s,\tau_j,\rho)\cdot\prod_{1\leq i<j\leq r}L(2s,\tau_i\times\tau_j^\iota).
$$
Since $\tau_1,\cdots,\tau_r$ are conjugate self-dual and distinct for all $1\leq i<j\leq r$, $L(2s,\tau_i\times\tau_j^\iota)$ is holomorphic and
nonzero at $s=\frac{1}{2}$. It follows that the $L$-function $L(2s,\tau,\rho)$ has a pole at $s=\frac{1}{2}$ of order $r$.
Since the generic global Arthur parameter $\phi_\tau$ associated to $\tau$ and the generic global Arthur parameter
$\phi_\sigma$ associated to $\sigma$ are in different parity, the $L$-function $L(s,\tau\times\sigma)$ must be of
symplectic type (\cite{GGP12}), and is holomorphic at $s=\frac{1}{2}$, but may have zero at $s=\frac{1}{2}$.
It follows that the global intertwining operator
$\CM(\omega_0,\tau\otimes\sigma,s)$ has a pole at $s=\frac{1}{2}$ of order at most $r$, and has the pole of order exactly $r$ if and only if
the $L$-function $L(s,\tau\times\sigma)$ is nonzero at $s=\frac{1}{2}$.
This implies that the Eisenstein series
$E(\cdot,\phi_{\tau\otimes\sigma},s)$ has a pole at $\frac{1}{2}$ of order at most $r$, and it has a pole of order $r$ at $s=\frac{1}{2}$ if and only if
$L(s,\tau\times\sigma)$ is nonzero at $s=\frac{1}{2}$. We summarize this result as follows.
\begin{prop}\label{esp}
Assume that $\sig\in\CA_\cusp(H_m)$ has a generic global Arthur parameter $\phi_\sig$.
Let $\tau=\tau_1\boxplus\tau_2\boxplus\cdots\boxplus\tau_r$ be the irreducible unitary generic isobaric automorphic representation of $G_{E/F}(a)(\BA_F)$
associated to distinct $\tau_1, \tau_2,\cdots,\tau_r$, as given in \eqref{tau8}. Assume that the generic global Arthur parameter $\phi_\tau$ determined by
$\tau$ has a different parity with $\phi_\sig$. Then the $L$-function $L(s,\tau\times\sig)$ is holomorphic at $s=\frac{1}{2}$, and
the Eisenstein series $E(\cdot, \phi_{\tau\otimes\sig},s)$ has a pole at $s=\frac{1}{2}$ of order at most $r$.
Moreover, $E(\cdot, \phi_{\tau\otimes\sig},s)$ has a pole at $s=\frac{1}{2}$ of order $r$
if and only if $L(s,\tau_i,\rho)$ has a pole at $s=1$ for $i=1,2,\cdots,r$,
and $L(s,\tau\times\sig)$ is nonzero at $s=\frac{1}{2}$.
\end{prop}

When the Eisenstein series $E(\cdot, \phi_{\tau\otimes\sig},s)$ has a pole at $s=\frac{1}{2}$ of order $r$, we denote by $\CE_{\tau\otimes\sig}$ the
$r$-th iterated residue at $s=\frac{1}{2}$ of $E(\cdot, \phi_{\tau\otimes\sig},s)$.

\subsection{Special data for Bessel periods}\label{sec-sdbps}
We are going to choose a set of special data in order to establish the reciprocal non-vanishing of the Bessel periods for
the pair $(\CE_{\tau\otimes\sig},\pi)$ and for the pair $(\pi,\sig)$.

Take as before the classical group $G_n=\Isom(V_\Fn,q)^\circ$. The group $G_n$ is a pure inner $F$-form of an $F$-quasisplit classical group
$G_n^*=\Isom(V^*_\Fn,q^*)^\circ$ of the same type. Here $\Fn=\dim_EV_\Fn=\dim_EV^*_\Fn$ and $n=[\frac{\Fn}{2}]$.
Recall from Section \ref{sec-dsap} that $N=\Fn^\vee$ is $\Fn$ if $G_n$ is
a unitary group or an even special orthogonal group; and is $\Fn-1$ if $G_n$ is an odd special orthogonal group.

Assume that $\pi\in\CA_\cusp(G_n)$ has a $G_n$-relevant, generic
global Arthur parameter $\phi\in\wt{\Phi}_2(G_n^*)$. As in \eqref{gap}, the generic global Arthur parameter $\phi$ determines
an irreducible unitary generic isobaric automorphic representation $\tau=\tau_1\boxplus\tau_2\boxplus\cdots\boxplus\tau_r$ of $G_{E/F}(N)(\BA_F)$,
as given in \eqref{tau8}. 
Recall that the sign $\kappa$ of $\xi$ is $+1$ for the global $A$-parameters of unitary groups, as explained in Section \ref{sec-dsap}.
Take a $G_n$-relevant partition
\begin{equation}\label{partition-5}
\udl{p}_{\ell_*}=[(2\ell_*+1)1^{\Fn-2\ell_*-1}]
\end{equation}
and consider the $\ell_*$-th Bessel module $\CF^{\CO_{\ell_*}}(\pi)$ of $\pi$, or $\CF^{\CO_{\ell_*}}(\CC_\pi)$ for a cuspidal realization $\CC_\pi$ of $\pi$.
Since $\pi$ is irreducible and cuspidal, $\CF^{\CO_{\ell_*}}(\pi)$ consists of rapidly decreasing automorphic functions on
$H_{\ell_*^-}^{\CO_{\ell_*}}(\BA)$, and hence is a sub-representation of $H_{\ell_*^-}^{\CO_{\ell_*}}(\BA)$ in the space of $L^2$-automorphic
functions on $H_{\ell_*^-}^{\CO_{\ell_*}}(\BA)$. Note that the group $H_{\ell_*^-}^{\CO_{\ell_*}}$ is a pure inner $F$-form of an $F$-quasisplit
$H_{\ell_*^-}^*$ with
$\ell_*^-=[\frac{\Fl_*^-}{2}]$ and $\Fl_*^-=\Fn-2\ell_*-1$.

We further assume that $\CF^{\CO_{\ell_*}}(\CC_\pi)$ is nonzero and
has the property that there exists  a $\sigma\in\CA_\cusp(H_{\ell_*^-}^{\CO_{\ell_*}})$ with a generic
global Arthur parameter, such that the inner product
\begin{equation}\label{bppi5}
\CP^{\psi_{\CO_{\ell_*}}}(\varphi_\pi,\varphi_\sig)
=\left<\CF^{\psi_{\CO_{\ell_*}}}(\varphi_\pi),\ol{\varphi}_\sigma\right>_{H_{\ell_*^-}^{\CO_{\ell_*}}}\neq 0,
\end{equation}
for some $\varphi_\pi\in\CC_\pi$ and $\varphi_\sigma\in\CC_\sigma$, where  $\CC_\sig$ is a cuspidal realization of $\sig$.

Note that the index $\ell_*$ may not be the {\sl first occurrence index} as described in
Conjecture \ref{bpconj}. In this generality, the discussion in this section can also be applied to the proof of one of the directions of
the global Gan-Gross-Prasad conjecture in Subsection \ref{sec-gggp}.

We take in this section that $m:=\ell_*^-$ and $\Fm:=\Fl_*^-=\Fn-2\ell_*-1$.
In the definition of global zeta integrals in Section \ref{sec-gzi}, we take
\begin{equation}\label{data1}
H_m=H_{\ell_*^-}^{\CO_{\ell_*}}\ \ \text{ and}\ \ a=N=\Fn^\vee,
\end{equation}
and take $H_{a+m}$ to be the classical group containing the Levi subgroup $G_{E/F}(a)\times H_m$.
To define the global zeta integrals, we take the partition
\begin{equation}\label{data3}
\udl{p}_{\kappa_*}=[(2\kappa_*+1)1^{2a+\Fm-2\kappa_*-1}]
\end{equation}
with $\kappa_*:=a-\ell_*-1$. It is a partition of type $(2a+\Fm,H_{a+m})$.

Since $a-\kappa_*=\ell_*+1>0$, we in the situation of Section \ref{sec-al}.
For any $F$-rational orbit $\CO_{\kappa_*}$ in the $F$-stable orbit $\CO_{\udl{p}_{\kappa_*}}^\st$,
we have the stabilizer $G_{m^-}^{\CO_{\kappa_*}}=G_{m^-}^{w_*}$ as in Proposition \ref{pre-ep-a>l}, with $(\kappa_*)^-=m^-$.
The integer $\Fm^-$, such that $m^-=[\frac{\Fm^-}{2}]$,
can be calculated as follows: By definition, we have
\begin{equation}\label{data2}
\Fm^-=2a+\Fm-2\kappa_*-1=\Fm+2(a-\kappa_*)-1.
\end{equation}
Since $a-\kappa_*=\ell_*+1$, we have $\Fm^-=\Fm+2\ell_*+1$. Since $\Fm=\Fn-2\ell_*-1$, we must have that $\Fm^-=\Fn$ and hence that $m^-=n$.
By Proposition \ref{piform}, and the relation of the three groups $(H_{a+m}, G_n, H_m)$, it is not hard to find the $F$-anisotropic vector
$w_*$ corresponding to the $F$-rational orbit $\CO_{\kappa_*}$ such that $G_n$ can be identified with the stabilizer $G_{m^-}^{\CO_{\kappa_*}}=G_{m^-}^{w_*}$.

Recall that $\CE_{\tau\otimes\sig}$ is the iterated residue at $s=\frac{1}{2}$ of the Eisenstein series $E(\cdot, \phi_{\tau\otimes\sig},s)$. The reciprocal
non-vanishing of the Bessel periods for the pair $(\CE_{\tau\otimes\sig},\pi)$ and for the pair $(\pi,\sig')$ is given below.

\begin{thm}[Reciprocal Non-vanishing of Bessel Periods]\label{th-rnbp}
Assume that $\sig\in\CA_\cusp(H_m)$ has a generic global Arthur parameter $\phi_\sig$.
Let $\tau=\tau_1\boxplus\tau_2\boxplus\cdots\boxplus\tau_r$ be the irreducible unitary generic isobaric automorphic representation of $G_{E/F}(a)(\BA_F)$
with $a=N=\Fn^\vee$, which determines a generic global Arthur parameter $\phi_\tau$ of $G_n^*$.
Assume that the residue $\CE_{\tau\otimes\sig'}$ is nonzero and $\pi\in\CA_\cusp(G_n)$ has a generic global Arthur parameter
$\phi_\tau$. 
Then the Bessel period
$\left<\varphi_\pi,\ol{\CF^{\psi_{\CO_{\kappa_*}}}(\CE_{\tau\otimes\sigma'})}\right>_{G_n}$ for the pair $(\CE_{\tau\otimes\sig'},\pi)$ is nonzero
for some choice of data
if and only if the Bessel period $\left<\CF^{\psi_{\CO_{\ell_*}}}(\varphi_\pi),\ol{\varphi}_{\sigma}\right>_{H_m}$ for the pair $(\pi,\sig)$ is nonzero for
some choice of data.
\end{thm}

By using Corollary \ref{zeroa>l}, it is easy to prove that if the Bessel period
$\left<\varphi_\pi,\ol{\CF^{\psi_{\CO_{\kappa_*}}}(\CE_{\tau\otimes\sigma'})}\right>_{G_n}$ is nonzero for some choice of data, then
the Bessel period $\left<\CF^{\psi_{\CO_{\ell_*}}}(\varphi_\pi),\ol{\varphi}_{\sigma}\right>_{H_m}$ is nonzero for
some choice of data. In fact, if $\left<\varphi_\pi,\ol{\CF^{\psi_{\CO_{\kappa_*}}}(\CE_{\tau\otimes\sigma'})}\right>_{G_n}$ is not identically zero,
then by replacing the residue $\CE_{\tau\otimes\sigma'}$ by the Eisenstein series $E(\cdot, \phi_{\tau\otimes\sig'},s)$, we obtain that
the global zeta integral $\CZ(s,\phi_{\tau\otimes\sig'},\varphi_{\pi},\psi_{\CO_{\kappa_*}})$ is not identically zero for $\Re(s)$ large. Hence
by Corollary \ref{zeroa>l}, the Bessel period $\left<\CF^{\psi_{\CO_{\ell_*}}}(\varphi_\pi),\ol{\varphi}_{\sigma}\right>_{H_m}$ is nonzero for
some choice of data.

The proof of the opposite direction is more technical. We have to know enough analytic properties of the local zeta integrals at the ramified and
the archimedean local places.

\subsection{Normalization of local zeta integrals}\label{sec-nlzi}
We continue our discussion of the global and local zeta integrals from Section \ref{sec-bpgzi} with special data as given in Section \ref{sec-sdbps}.
In particular, we will deal with the case where $a-\kappa_*=\ell_*+1>0$, which is the case of Section \ref{sec-al}. Recall from \eqref{formula1},
the global zeta integral has the following expression:
$$
\CZ(s,\phi_{\tau\otimes\sig'},\varphi_{\pi},\psi_{\CO_{\kappa_*}})
=\CZ_S(s,\phi_{\tau\otimes\sig'},\varphi_{\pi},\psi_{\CO_{\kappa_*}})\cdot\CL^S(s,\tau,\pi,\sig;\rho),
$$
where 
$\CZ_S(s,\cdot)=\prod_{\nu\in S}\CZ_\nu(s,\cdot)$ is the finite Euler product 
with the local zeta integral $\CZ_\nu(s,\cdot)$ as in \eqref{localzetaa>l}, and 
$$\CL^S(s,\tau,\pi,\sig;\rho)=\prod_{\nu\not\in S}\CL(s,\tau_\nu,\pi_\nu,\sig_\nu;\rho)$$ 
with $\CL(s,\tau_\nu,\pi_\nu,\sig_\nu;\rho)$ as in \eqref{cl}.
Recall that $S$  consists of  all ramified places of relevant data and all archimedean places such that for $\nu\notin S$ all data are unramified.
Hence at $\nu\notin S$ we only consider spherical vectors in the discussion.


We normalize the local zeta integrals by
\begin{equation}\label{nlzi-1}
\CZ_\nu^*(s,\phi_{\tau\otimes\sig'},\varphi_{\pi},\psi_{\CO_{\kappa_*}})
:=
\frac{\CZ_\nu(s,\phi_{\tau\otimes\sig'},\varphi_{\pi},\psi_{\CO_{\kappa_*}})}{\CL(s,\tau_\nu,\pi_\nu,\sig_\nu;\rho)}.
\end{equation}
By taking the finite product, we have $\CZ_S(s,\cdot)=\CZ_S^*(s,\cdot)\cdot\CL_S(s,\cdot)$. Hence the formula in \eqref{formula1} becomes
\begin{equation}\label{nfa>l}
\CZ(s,\phi_{\tau\otimes\sig'},\varphi_{\pi},\psi_{\CO_{\kappa_*}})
=\CZ_S^*(s,\phi_{\tau\otimes\sig'},\varphi_{\pi},\psi_{\CO_{\kappa_*}})\cdot\CL(s,\tau,\pi,\sig;\rho).
\end{equation}

\begin{prop}\label{lzi-pp2}
The assumption on $(\pi,\tau,\sig)$ is taken as in Theorem \ref{th-rnbp}.
Then the following hold.
\begin{enumerate}
\item $\CZ_S^*(s,\phi_{\tau\otimes\sig'},\varphi_{\pi},\psi_{\CO_{\kappa_*}})$ is meromorphic in $s\in\BC$ for
any choice of the smooth sections $\phi_{\tau\otimes\sig',s}$.
\item $\CZ_S^*(s,\phi_{\tau\otimes\sigma'},\varphi_\pi,\psi_{\CO_{\kappa_*}})$ is holomorphic at $s=\frac{1}{2}$ for any choice of the smooth sections
$\phi_{\tau\otimes\sig',s}$.
\end{enumerate}
\end{prop}
Note that $\phi_{\tau\otimes\sig',s}$ is called a {\sl smooth section} if
$\phi_{\tau_\nu\otimes\sig'_\nu,s}$ is a smooth holomorphic section at archimedean places and is a flat section at non-archimedean places.

\begin{proof}
Recall from \eqref{gzi}, we have
$$
\CZ(s,\phi_{\tau\otimes\sig'},\varphi_{\pi},\psi_{\CO_{\kappa_*}})
=
\left<\varphi_\pi,\ol{\CF^{\psi_{\CO_{\kappa_*}}}(E(\cdot,\phi_{\tau\otimes\sigma'},s))}\right>_{G_n}.
$$
and hence by \eqref{nfa>l}, we obtain
\begin{equation}\label{bpsa>l}
\left<\varphi_\pi,\ol{\CF^{\psi_{\CO_{\kappa_*}}}(E(\cdot,\phi_{\tau\otimes\sigma'},s))}\right>_{G_n}
=
\CZ_S^*(s,\phi_{\tau\otimes\sig'},\varphi_{\pi},\psi_{\CO_{\kappa_*}})\cdot\CL(s,\tau,\pi,\sig;\rho).
\end{equation}

We first consider the right hand side of the identity in \eqref{bpsa>l}.
In the $L$-function part, we have
\begin{equation}\label{clpart}
\CL(s,\tau,\pi,\sigma;\rho)
=
\frac{L(s+\frac{1}{2},\tau\times\pi)}
{L(s+1,\tau\times\sig)L(2s+1,\tau,\rho)}.
\end{equation}
Since the cuspidal automorphic representations $\pi$ and $\sigma$ are assumed to have generic global Arthur parameters,
the complete $L$-functions $L(s,\tau\times\pi)$ and $L(s,\tau\times\sig)$ are defined
in terms of the global Arthur parameters of $\pi$ and $\sigma$, respectively.
Hence $\CL(s,\tau,\pi,\sigma;\rho)$ is meromorphic in $s$ over $\BC$.

In the left hand side of the identity in \eqref{bpsa>l}, the Fourier coefficient
$\CF^{\psi_{\CO_{\kappa_*}}}(E(\cdot,\phi_{\tau\otimes\sigma'},s))$ is
meromorphic in $s$ over $\BC$ and the inner product
$\left<\varphi_\pi,\ovl{\CF^{\psi_{\CO_{\kappa_*}}}(E(\cdot,\phi_{\tau\otimes\sigma'},s))}\right>_{G_n}$
is well defined when $s$ is away from the poles of $\CF^{\psi_{\CO_{\kappa_*}}}(E(\cdot,\phi_{\tau\otimes\sigma'},s))$,
since $\varphi_\pi$ is cuspidal.
Hence the left hand side of the identity is meromorphic in $s$ over $\BC$. It follows that the finite product of the
normalized local zeta integrals, $\CZ_S^*(s,\phi_{\tau\otimes\sigma'},\varphi_\pi,\psi_{\CO_{\kappa_*}})$,
is a meromorphic function in $s$ over $\BC$
for any choice of $\phi_{\tau\otimes\sig'}$ with the given $\varphi_\sig$ and for the given $\varphi_\pi$. This proves Part (1).

In order to prove Part (2), we need more specific information from both sides. In the expression \eqref{clpart}, the following product
$$
L(s+\frac{1}{2},\tau\times\pi)=\prod_{i=1}^r L(s+\frac{1}{2},\tau_{i}\times\pi)
$$
has a pole at $s=\frac{1}{2}$ of order $r$, since the cuspidal automorphic representation $\pi$ has the generic global Arthur
parameter $\phi_\tau$ with
$
\tau=\tau_1\boxplus\cdots\boxplus\tau_r,
$
as in \eqref{tau8}. It is clear that the product
$$
L(s+1,\tau\times\sig)=\prod_{i=1}^r L(s+1,\tau_{i}\times\sigma)
$$
and the product
$$
L(2s+1,\tau,\rho)=\prod_{i=1}^r L(2s+1,\tau_{i},\rho)
\times \prod_{1\leq i<j\leq r} L(2s+1,\tau_{i}\times\tau_{j}^\iota)
$$
are holomorphic and nonzero at $s=\frac{1}{2}$.
It follows that the $L$-function part $\CL(s,\tau,\pi,\sigma;\rho)$ has a pole at $s=\frac{1}{2}$ of
order $r$.

In order to show that the finite product of the normalized local zeta integrals,
$\CZ_S^*(s,\phi_{\tau\otimes\sigma'},\varphi_\pi,\psi_{\CO_{\kappa_*}})$, is holomorphic at $s=\frac{1}{2}$, it is enough to show that
the inner product in the left hand side of \eqref{bpsa>l}
$$
\left<\varphi_\pi,\ol{\CF^{\psi_{\CO_{\kappa_*}}}(E(\cdot,\phi_{\tau\otimes\sigma'},s))}\right>_{G_n},
$$
has a pole at $s=\frac{1}{2}$ of order at most $r$ for any smooth sections $\phi_{\tau\otimes\sigma',s}$.

By Proposition \ref{esp}, the Fourier coefficient
$\CF^{\psi_{\CO_{\kappa_*}}}(E(\cdot,\phi_{\tau\otimes\sigma'},s))$
of the Eisenstein series $E(\cdot,\phi_{\tau\otimes\sigma'},s)$ has a pole at $s=\frac{1}{2}$ of order at most $r$
for any smooth sections $\phi_{\tau\otimes\sigma',s}$. The inner product of the
Fourier coefficient with the cuspidal automorphic form $\varphi_\pi$ can not increase the order of the pole at $s=\frac{1}{2}$. It follows that
the left hand side of \eqref{bpsa>l} has a pole at $s=\frac{1}{2}$ of order at most $r$.
Therefore, we obtain that the finite product of the normalized local zeta integrals,
$\CZ_S^*(s,\phi_{\tau\otimes\sigma'},\varphi_\pi,\psi_{\CO_{\kappa_*}})$,
must be holomorphic at $\frac{1}{2}$. This proves Part (2).
\end{proof}

In order to obtain more properties of the local zeta integrals or the normalized ones for the global purpose in this paper,
we have to introduce the global condition that the Bessel period for $(\pi,\sig)$ is nonzero.
For $\varphi_\pi\in\CC_\pi$ and $\varphi_\sig\in\CC_\sig$, the Bessel period $\left<\CF^{\psi_{\CO_{\ell_*}}}(\varphi_\pi),\ol{\varphi}_{\sigma}\right>_{H_m}$, as in \eqref{bppi5} with the given data, defines a nonzero element in the one-dimensional space
$
\otimes_\nu\Hom_{H_m(F_\nu)}(\CJ_{\psi_{\CO_{\ell_*}}}(\pi_\nu)\otimes\sig_\nu,\BC),
$
where $\CJ_{\psi_{\CO_{\ell_*}}}(\pi_\nu)$ is the local twisted Jacquet module of $\pi_\nu$ with respect to $(V_{\udl{p}_{\ell^*}},\psi_{\CO_{\ell_*}})$.
Let $\Fb^{\psi_{\CO_{\ell_*}}}_\nu$ be the a nonzero functional in
$
\Hom_{H_m(F_\nu)}(\CJ_{\psi_{\CO_{\ell_*}}}(\pi_\nu)\otimes\sig_\nu,\BC),
$
which is unique, up to scalar. We normalize the functional $\Fb^{\psi_{\CO_{\ell_*}}}_\nu$, so that at the unramified local places,
$$
\Fb^{\psi_{\CO_{\ell_*}}}_\nu(\varphi_{\pi_\nu},\varphi_{\sig_\nu})=1,
$$
where $\varphi_{\pi_\nu}$ and $\varphi_{\sig_\nu}$ are normalized spherical vectors in $\pi_\nu$ and $\sig_\nu$, respectively. Here a spherical vector
is normalized if its corresponding spherical function has value $1$ at the identity. And at the ramified local places $\nu\in S$, the local
functional $\Fb^{\psi_{\CO_{\ell_*}}}_\nu$ will be normalized according to \eqref{eq:nonvanish-pf-i} in Appendix \ref{A}.
Hence we obtain the following identity: for factorizable factors
$\varphi_\pi=\otimes_\nu\varphi_{\pi_\nu}$ and $\varphi_\sig=\otimes_\nu\varphi_{\sig_\nu}$,
\begin{equation}\label{bpgl-1}
\left<\CF^{\psi_{\CO_{\ell_*}}}(\varphi_\pi),\ol{\varphi}_{\sigma}\right>_{H_m}
=
c_{\pi,\sig}\cdot\prod_\nu\Fb^{\psi_{\CO_{\ell_*}}}_\nu(\varphi_{\pi_\nu},\varphi_{\sig_\nu}),
\end{equation}
by the uniqueness of the local Bessel models (\cite{AGRS}, \cite{SZ}, \cite{GGP12} and \cite{JSZ}).

\begin{prop}\label{nlzip:nonzero}
The assumption on $(\pi,\tau,\sig)$ is taken as in Theorem \ref{th-rnbp}.
Fix an arbitrary $s=s_0\in\BC$.
For any $\varphi_{\sig_\nu}\in \sig_\nu$ and $\varphi_{\pi_\nu}\in \pi_\nu$ such that the local pairing
$\Fb^{\psi_{\CO_{\ell_*}}}_\nu(\varphi_{\pi_\nu},\varphi_{\sig_\nu})$
is nonzero for every $\nu\in S$,
there exists  a section $\phi_{\tau_\nu\otimes\sig'_\nu}$
such that the finite product of
the local zeta integrals,
$\CZ_S(s,\phi_{\tau\otimes\sigma'},\varphi_\pi,\psi_{\CO_{\kappa_*}})$, is a nonzero constant at $s=s_0$.
\end{prop}

The proof of Proposition \ref{nlzip:nonzero} will be given in Appendix A.

\begin{prop}\label{nlzip-key}
The assumption on $(\pi,\tau,\sig)$ is taken as in Theorem \ref{th-rnbp}.
Assume further that $(\pi,\sig)$ has a nonzero Bessel period. Then
there exist factorizable data $\varphi_\pi$, $\varphi_{\sigma}$, and $\phi_{\tau\otimes\sig'}$, such that the inner product
$\left<\CF^{\psi_{\CO_{\ell_*}}}(\varphi_\pi),\ol{\varphi}_{\sigma}\right>_{H_m}$, and
$\CZ_S^*(s,\phi_{\tau\otimes\sigma'},\varphi_\pi,\psi_{\CO_{\kappa_*}})$ at $s=\frac{1}{2}$ are simultaneously nonzero.
\end{prop}

\begin{proof}
By assumption, the Bessel period
$\left<\CF^{\psi_{\CO_{\ell_*}}}(\varphi_\pi),\ol{\varphi}_\sigma\right>_{H_m}$
is not zero for the pair $(\pi,\sig)$. By \eqref{bpgl-1}, for factorizable vectors $\varphi_\pi=\otimes_\nu\varphi_{\pi_\nu}$ and $\varphi_\sig=\otimes_\nu\varphi_{\sig_\nu}$, we have
$$
\left<\CF^{\psi_{\CO_{\ell_*}}}(\varphi_\pi),\ol{\varphi}_\sigma\right>_{H_m}
=c_{\pi,\sigma} \cdot \prod_{\nu}
\Fb_\nu^{\psi_{\CO_{\ell_*}}}(\varphi_{\pi_{\nu}},\varphi_{\sigma_{\nu}}).	
$$
Since $\left<\CF^{\psi_{\CO_{\ell_*}}}(\varphi_\pi),\ol{\varphi}_\sigma\right>_{H_m}$
is not zero,
it follows that $c_{\pi,\sig}\ne 0$ and $\Fb_\nu(\varphi_{\pi_\nu},\varphi_{\sig_\nu})$ is nonzero for every $\nu$.
By Proposition \ref{nlzip:nonzero}, there exists a smooth factorizable section $\phi_{\tau\otimes\sig'}$ such that the finite product of
the local zeta integral $\CZ_S(s,\phi_{\tau\otimes\sigma'},\varphi_\pi,\psi_{\CO_{\kappa_*}})$ is nonzero at $s=\frac{1}{2}$.
Hence the normalized $\CZ_S^*(s,\phi_{\tau\otimes\sigma'},\varphi_\pi,\psi_{\CO_{\kappa_*}})$ is also nonzero at $\frac{1}{2}$.
\end{proof}

\subsection{Proof of Theorem \ref{th-rnbp}}\label{sec-thm53}
We already proved one direction of Theorem \ref{th-rnbp}. Now we are ready to prove the other direction of Theorem \ref{th-rnbp}.

By the assumptions in Theorem \ref{th-rnbp}, the equation in \eqref{bpsa>l} reads
$$
\left<\varphi_\pi,\ol{\CF^{\psi_{\CO_{\kappa_*}}}(E(\cdot,\phi_{\tau\otimes\sigma'},s))}\right>_{G_n}
=
\CZ_S^*(s,\phi_{\tau\otimes\sig'},\varphi_{\pi},\psi_{\CO_{\kappa_*}})\cdot\CL(s,\tau,\pi,\sig;\rho).
$$
By Proposition \ref{lzi-pp2}, $\CZ_S^*(s,\phi_{\tau\otimes\sig'},\varphi_{\pi},\psi_{\CO_{\kappa_*}})$ is meromorphic in $s$ and is holomorphic
at $s=\frac{1}{2}$ for any section $\phi_{\tau\otimes\sig'}$ depending on the choice of $\varphi_\pi$ and $\varphi_\sig$ with property that
the Bessel period $\left<\CF^{\psi_{\CO_{\ell_*}}}(\varphi_\pi),\ol{\varphi}_\sigma\right>_{H_m}$ is nonzero. Furthermore,
by Proposition \ref{nlzip-key},
there exists a choice of factorizable $\varphi_\pi$, $\varphi_\sig$, and $\otimes_\nu f_{\CW_{\tau_\nu}^\kappa\otimes\sig_\nu'}$, such that both $\left<\CF^{\psi_{\CO_{\ell_*}}}(\varphi_\pi),\ol{\varphi}_\sigma\right>_{H_m}$ is nonzero and
$\CZ_S^*(s,\phi_{\tau\otimes\sig'},\varphi_{\pi},\psi_{\CO_{\kappa_*}})$ is nonzero at $s=\frac{1}{2}$.

With such a choice of data, and with a factorizable $\phi_{\tau\otimes\sig'}=\otimes_\nu\phi_{\tau_\nu\otimes\sig_\nu'}$ corresponding to the above
$\otimes_\nu f_{\CW_{\tau_\nu}^\kappa\otimes\sig_\nu'}$,
the right hand side of \eqref{bpsa>l} has a pole at $s=\frac{1}{2}$ of order $r$. It follows that the left hand side of
\eqref{bpsa>l}, i.e. $\left<\varphi_\pi,\ol{\CF^{\psi_{\CO_{\kappa_*}}}(E(\cdot,\phi_{\tau\otimes\sigma'},s))}\right>_{G_n}$,
has a pole at $s=\frac{1}{2}$ of order $r$. Since the Eisenstein series $E(\cdot,\phi_{\tau\otimes\sigma'},s)$ has a pole at $s=\frac{1}{2}$ of order at most $r$, we must have that $E(\cdot,\phi_{\tau\otimes\sigma'},s)$ has a pole at $s=\frac{1}{2}$ of order $r$.
Since $\varphi_\pi$ is cuspidal, by taking the iterated residue of
$\left<\varphi_\pi,\ol{\CF^{\psi_{\CO_{\kappa_*}}}(E(\cdot,\phi_{\tau\otimes\sigma'},s))}\right>_{G_n}$ at $s=\frac{1}{2}$,
we obtain that the Bessel period
$\left<\varphi_\pi,\ol{\CF^{\psi_{\CO_{\kappa_*}}}(\CE_{\tau\otimes\sigma'})}\right>_{G_n}$
is nonzero with such a chosen data,
where $\CE_{\tau\otimes\sigma'}$ denotes, as before, the iterated residue of the Eisenstein series $E(\cdot,\phi_{\tau\otimes\sigma'},s)$ at the pole
$s=\frac{1}{2}$ of order exactly equal to $r$. This completes the proof of Theorem \ref{th-rnbp}.

\subsection{Global Gan-Gross-Prasad conjecture: one direction}\label{sec-gggp}
We are ready to derive the proof of one of the direction of the global Gan-Gross-Prasad conjecture (Conjectures 24.1 and 26.1 in \cite{GGP12}), as
continuation of the proof of Theorem \ref{th-rnbp} in Section \ref{sec-thm53}.

Assume that the Bessel period
$\left<\CF^{\psi_{\CO_{\ell_*}}}(\varphi_\pi),\varphi_\sigma\right>_{H_m}$ for the pair $(\pi,\sig)$ is nonzero for some $\varphi_\pi$ and $\varphi_\sig$
as in \eqref{bppi5}.  By the same proof as in Section \ref{sec-thm53}, we obtain that the Fourier coefficient
$\CF^{\psi_{\CO_{\kappa_*}}}(\CE_{\tau\otimes\bar{\sigma}'})$ is nonzero,
where $\bar{\sig}'$ is the complex conjugate of $\sig'$.
As a consequence, we obtain that the iterated residual representation $\CE_{\tau\otimes\bar{\sigma}'}$ is nonzero.
By Proposition \ref{esp}, we obtain that $L(s,\tau\times\bar{\sigma})=L(s,\tau\times\bar{\sigma}')$ is holomorphic and nonzero at $s=\frac{1}{2}$.

Because $\bar{\sig}$ is isomorphic to the contragredient $ \sig^\vee$ of $\sig$, it follows that $L(s,\tau\times\sigma^\vee)$ is holomorphic and nonzero at $s=\frac{1}{2}$, and so is $L(s,\tau\times \sigma)$.
This proves one direction of the global Gan-Gross-Prasad conjecture (\cite{GGP12}) in the full generality for the classical groups considered in this paper.

\begin{thm}[Global Gan-Gross-Prasad Conjecture: one direction]\label{gggp1}
For any $\pi\in\CA_\cusp(G_n)$ with a $G_n$-relevant, generic global Arthur parameter in $\wt{\Phi}_2(G_n^*)_{G_n}$, and with a cuspidal realization
$\CC_\pi$ of $\pi$,
assume that the Bessel period
$$
\left<\CF^{\psi_{\CO_{\ell_*}}}(\varphi_\pi),\varphi_\sigma\right>_{H_m}
$$
is nonzero with a choice of $\varphi_\pi\in\CC_\pi$ and $\varphi_\sig\in\CC_\sig$, for some $\sig\in\CA_\cusp(H_m)$ with an $H_m$-relevant, generic global Arthur parameter in $\wt{\Phi}_2(H_m^*)_{H_m}$, and with a cuspidal realization $\CC_\sig$ of $\sig$.
Then the tensor product $L$-function
$L(s,\pi\times\sigma)=L(s,\tau\times\sigma)$
must be holomorphic and nonzero at $s=\frac{1}{2}$.
\end{thm}

Some remarks are in order.

First of all, the original global Gan-Gross-Prasad conjecture in \cite{GGP12} assumes that the cuspidal multiplicity of $\pi\in\CA_\cusp(G_n)$ should be
one. Theorem \ref{gggp1} takes care of the even special orthogonal group case, where the cuspidal multiplicity of $\pi$ could be two.

When $G_n=G_n^*$ and $H_m=H_m^*$ are $F$-quasisplit, and when both $\pi$ and $\sigma$ are generic, i.e. have non-zero Whittaker-Fourier
coefficients, and have simple, generic global Arthur parameters, i.e. their Langlands functorial transfers to the corresponding general linear groups
are cuspidal, Theorem \ref{gggp1} was considered in \cite{GJR04}, \cite{GJR05}, and \cite{GJR09}
by an approach mixing the Arthur truncation method and the Rankin-Selberg method. Recently, it was noticed by experts that there exists a gap
in the proof of Proposition 5.3 in \cite{GJR04}, which was duplicated in \cite{GJR05} and \cite{GJR09}. This technical gap is crucial to the
complete proof of the special case of Theorem \ref{gggp1} considered in those papers, and needs to be filled up.

Meanwhile, the assumption of the genericity
of both $\pi$ and $\sig$ and the assumption of the cuspidality of the functorial transfer of $\pi$ and $\sig$ to general linear groups are
critical to make the arguments and proofs work before Proposition 5.3 in \cite{GJR04}, and the same in \cite{GJR05} and \cite{GJR09}. Those restrictions
disappear in the approach taken in this paper.
It seems to the authors of this paper that the approach taken up
using the general framework (including the twisted automorphic descents and the reciprocal non-vanishing for Bessel periods)
considered in this paper is a more natural and conceptual way to attack the global Gan-Gross-Prasad conjecture.

It is also very important to mention that W. Zhang proved the global Gan-Gross-Prasad conjecture (\cite{ZW-1} and \cite{ZW-2}) for
unitary groups $\RU_n\times\RU_{n-1}$, with certain local assumptions, and with the global assumption on cuspidality of the global Langlands functorial transfers from unitary groups to general linear groups. His approach is based on the Jacquet-Rallis
relative trace formula originally developed in \cite{JR11} for unitary groups. The progress to extend the approach of Zhang to more general situation
has been picked up by Y. Liu (\cite{LiuYF14}) and by H. Xue (\cite{XueH14}). However, this relative trace formula approach is so far not known to be
available for
classical groups that are not unitary groups. The approach taken up in this paper treats both unitary groups and orthogonal groups uniformly.
The same approach is expected to work for symplectic groups and metaplectic groups with replacement of Bessel models by Fourier-Jacobi models. We
refer to our work (\cite{JZ-Howe}) for more details.

The other direction of the global Gan-Gross-Prasad conjecture (\cite{GGP12}) is more delicate and will be discussed
in Section \ref{sec-wfs} with assumption on the structure of Fourier coefficients of the
residual representation $\CE_{\tau\otimes\sigma}$ on $H_{a+m}(\BA)$. See Theorem \ref{gggp2} for details.


\section{Twisted Automorphic Descents}\label{sec-ccam}



We develop here a basic theory of twisted automorphic descents and point out two relevant applications. One is discussed in Subsection \ref{sec-mccam} the explicit construction of cuspidal automorphic modules for  
any irreducible cuspidal automorphic representations of $G_n$ and another is discussed in Subsection \ref{sec-wfs} on the other direction of the global Gan-Gross-Prasad conjecture. 


\subsection{Automorphic descents and certain Arthur packets}\label{sec-nbfc}
For a given $\pi\in\CA_\cusp(G_n)$ with a $G_n$-relevant, generic global Arthur parameter $\phi=\phi_\tau\in\wt{\Phi}_2(G_n^*)$, we
recall that $\phi_\tau$ has the form
$$
(\tau_1,1)\boxplus(\tau_2,1)\boxplus\cdots\boxplus(\tau_r,1)\in\wt{\Phi}_2(G_n^*).
$$
Remark that we choose the sign $\kappa=+1$ for the unitary group case  as in Section \ref{sec-dsap}. 
Take a $\sigma\in\CA_\cusp(H_m)$ with an $H_m$-relevant, generic global Arthur parameter $\phi_\sig\in\wt{\Phi}_2(H_m^*)$, and
define a non-generic global Arthur parameter by
\begin{equation}\label{nggap}
\psi_{\tau,\sig}:=(\tau_1,2)\boxplus(\tau_2,2)\boxplus\cdots\boxplus(\tau_r,2)\boxplus\phi_\sig.
\end{equation}
It is clear that $\psi_{\tau,\sig}$ belongs to $\wt{\Psi}_2(H_{a+m}^*)$, and is $H_{a+m}$-relevant.
Let $\wt{\Pi}_{\psi_{\tau,\sig}}(H_{a+m})$ be the global Arthur packet attached to the global Arthur parameter $\psi_{\tau,\sig}$ in \eqref{nggap}.
As in \cite{JLZ13}, one may easily verify the following property.

\begin{prop}\label{res}
The residual representation $\CE_{\tau\otimes\sigma}$ is square integrable, and, if non-zero, belongs to the global Arthur packet
$\wt{\Pi}_{\psi_{\tau,\sig}}(H_{a+m})$ with the global Arthur parameter $\psi_{\tau,\sig}$ given in \eqref{nggap}.
\end{prop}

From the endoscopic classification of Arthur in \cite{A13}, it is expected that the global Arthur packet $\wt{\Pi}_{\psi_{\tau,\sig}}(H_{a+m})$ contains some 
memebers that belong to $\CA_\disc(H_{a+m})$. If these automorphic members are not of residue type, they are cuspidal. 
The twisted automorphic descent is an approach to understand the structures and  the properties of the global packet $\wt{\Pi}_{\psi_{\tau,\sig}}(H_{a+m})$,
instead of a certain individual member in the global packet $\wt{\Pi}_{\psi_{\tau,\sig}}(H_{a+m})$. 

We assume that a $\Sigma\in\CA_\disc(H_{a+m})$ has the global Arthur parameter $\psi_{\tau,\sig}$ as given in \eqref{nggap},
and has a discrete realization $\CC_\Sigma$. Consider Fourier coefficients associated to the partitions of the form
$$
\underline{p}_{\kappa}=[(2\kappa+1)1^{2a+\Fm-2\kappa-1}]
$$
of $2a+\Fm$ with $0\leq\kappa\leq a+\Fr_\Fm$, where $\Fr_\Fm$ is the $F$-rank of $H_m$.
It is clear that the partition $\underline{p}_{\kappa}$ is of type $(2a+\Fm,H_{a+m}^*)$. As in Section~\ref{sec-fcpt}, we
study the $\psi_{\udl{p}_{\kappa},\CO_\kappa}$-Fourier coefficient of $f_\Sigma\in\CC_\Sigma$, and denote by $\CF^{\CO_\kappa}_{\kappa^-}(\Sigma)$
the $\kappa$-th Bessel module of $G_{\kappa^-}^{\CO_\kappa}(\BA)$ generated by all the Fourier coefficients
$\CF^{\psi_{\CO_\kappa}}(f_\Sigma)$ with all $f_\Sigma\in\CC_\Sigma$.
As in \cite{GRS11} for the case $m=0$ and
in \cite{JLXZ} for $m=1$, we prove the following proposition by
investigating the local structure at one unramified place of the global Arthur parameter $\psi_{\tau,\sig}$ given in \eqref{nggap}.

\begin{prop}\label{fcSigma}
Assume that a $\Sigma\in\CA_\disc(H_{a+m})$ belongs to the global Arthur packet $\wt{\Pi}_{\psi_{\tau,\sig}}(H_{a+m})$
with the parameter $\psi_{\tau,\sig}$ given in \eqref{nggap}. Set $\ell_0:=\frac{\Fn-\Fm-1}{2}$. 
For any integer $\kappa$ with $a-\ell_0-1<\kappa\leq a+\Fr_\Fm$, the $\kappa$-th Bessel modules
$\CF^{\CO_\kappa}_{\kappa^-}(\Sigma)$
are zero for all $F$-rational nilpotent orbits $\CO_\kappa$ in the $F$-stable orbit $\CO^\st_{\udl{p}_{\kappa}}(F)$.
\end{prop}

\begin{proof}
First, the $\kappa$-th Bessel module $\CF^{\CO_\kappa}_{\kappa^-}(\Sigma)$
produces the corresponding local Jacquet module of
$\Sigma_{\nu}$ with respect to $(V_{\udl{p}_{\kappa}},\psi_{\CO_\kappa})$ at any finite local place $\nu$.
At almost all finite local places, $\Sigma_{\nu}$ is unramified and
is completely determined by the $\nu$-component of the global Arthur parameter $\psi_{\tau,\sigma}$. Take one of such unramified
finite local place $\nu$, the generic unramified representation $\tau_{\nu}$ of $G_{E/F}(a)(F_{\nu})$ is conjugate self-dual and
hence is completely determined by $[\frac{a}{2}]$
unramified characters $\mu_1,\cdots,\mu_{[\frac{a}{2}]}$, and $\sigma_{\nu}$ is also an irreducible generic unramified representation of
$F_{\nu}$-quasisplit $H_m(F_{\nu})$. As in \cite[Chapter 5]{GRS11} and in the proof of Proposition 2.3 of \cite{JLXZ},
the unramified local component $\Sigma_{\nu}$ can be realized as the unique irreducible unramified constituent of the following induced representation
\begin{equation}\label{usrp}
\Ind^{H_{a+m}(F_{\nu})}_{P_{\hat{a}}(F_{\nu})}(\tau_{\nu}'\otimes\sigma_{\nu})
\end{equation}
where $\tau_{\nu}'=\Ind^{G_{E/F}(a)(F_{\nu})}_{Q_{[2^{[\frac{a}{2}]}]}(F_{\nu})}(\mu_1\circ\det_2)\otimes\cdots\otimes(\mu_{[\frac{a}{2}]}\circ\det_2)$ in most cases,
and $\det_2$ is the determinant of $G_{E/F}(2)$.
Refer to the discussions in the proof of Lemma \ref{lm:local-descent-unramified}  for all cases of $\tau_{\nu}'$. In the rest of the proof,
the argument works for all cases of $\tau_{\nu}'$, although we only discuss the situation as in \eqref{usrp}.

In \cite[Chapter 5]{GRS11}, the calculation of the local Jacquet modules of the induced representation as in \eqref{usrp}
with respect to $(V_{\udl{p}_{\kappa}},\psi_{\CO_\kappa})$ and for general $\kappa$ has been explicitly carried out.
See \cite[Theorem 5.1]{GRS11}, in particular.
Hence it is not hard to figure out, as in \cite[Section 2]{JLXZ}, that for $\kappa$ with
$a-\ell_0-1<\kappa\leq a+\Fr_\Fm$, such a local Jacquet
module is always zero for the induced representation as in \eqref{usrp},
and so is for the unramified local component $\Sigma_{\nu}$ at the fixed local place $\nu$.
This proves that for all $\kappa$ with $a-\ell_0-1<\kappa\leq a+\Fr_\Fm$,
$\kappa$-th Bessel module $\CF^{\CO_\kappa}_{\kappa^-}(\Sigma)$ must be zero
for all such orbits $F$-rational $\CO_\kappa$ in the $F$-stable orbit $\CO^\st_{\udl{p}_{\kappa}}(F)$.
\end{proof}

The proof uses the structure of unramified local components of $\Sigma$ and hence is independent of the discrete realization of $\Sigma$ if
$\Sigma$ has high discrete multiplicity. The same happens to the proof of the following proposition, which
considers the $\kappa_0$-Bessel modules $\CF^{\CO_{\kappa_0}}_{n}(\Sigma)$
for the case where $\kappa_0=a-\ell_0-1$ and hence $\kappa_0^-=m^-=n$.

\begin{prop}\label{cusp}
Let $\tau$ and $\sigma$ be as in Proposition \ref{fcSigma}.
For an automorphic member $\Sigma$ in the global Arthur packet $\wt{\Pi}_{\psi_{\tau,\sig}}(H_{a+m})$,
$\kappa_0$-Bessel modules $\CF^{\CO_{\kappa_0}}_{n}(\Sigma)$, with $n=\kappa_0^-$, is cuspidal,
as a sub-representation of $G_n^{\CO_{\kappa_0}}(\BA)$ in the cuspidal spectrum $L^2_\cusp(G_n^{\CO_{\kappa_0}})$,
for all $F$-rational nilpotent orbits $\CO_{\kappa_0}$ in the $F$-stable orbit $\CO^\st_{\udl{p}_{\kappa_0}}(F)$
with $\kappa_0=a-\ell_0-1$ and $\ell_0=\frac{\Fn-\Fm-1}{2}$.
\end{prop}

\begin{proof}
It suffices to show that the constant term of $\CF^{\CO_{\kappa_0}}_{n}(\Sigma)$
along every standard parabolic subgroup of $G_n^{\CO_{\kappa_0}}$ is zero.
The proof uses essentially the tower property developed in \cite[Chapter 7]{GRS11}. We take, in particular, Theorem 7.3 of \cite{GRS11}.
As in the proof of Proposition 2.5 of \cite{JLXZ}, it is enough to show the conditions in \cite[Theorem 7.3]{GRS11} hold.
Because of Proposition \ref{fcSigma}, the terms in \cite[Equation (7.35)]{GRS11} are all zero.
If $\Sigma$ is cuspidal, the conditions in \cite[Theorem 7.3]{GRS11} are automatic. Hence in this case, all the constant terms are zero, and
therefore, $\kappa_0$-th Bessel modules $\CF^{\CO_{\kappa_0}}_{n}(\Sigma)$ are cuspidal.

When $\Sigma$ in the global Arthur packet $\wt{\Pi}_{\psi_{\tau,\sig}}(H_{a+m})$ is not cuspidal, it must be a residual
representation with the global Arthur parameter $\psi_{\tau,\sig}$. According to \cite[Theorem \S1.3 ]{M08} and \cite[Theorem B]{M11}, among the residual
representations in the global Arthur packet $\wt{\Pi}_{\psi_{\tau,\sig}}(H_{a+m})$, $\Sigma=\CE_{\tau\otimes\sigma}$ has the least
cuspidal support in the sense that among the cuspidal supports of those residual representations, the Levi subgroup in the cuspidal
support of $\CE_{\tau\otimes\sigma}$ is the smallest one.
It is enough to consider the case when $\Sigma=\CE_{\tau\otimes\sigma}$. The same argument will be applicable to
the other residual representations.

For $\Sigma=\CE_{\tau\otimes\sigma}$, in Formula (7.35) of \cite{GRS11}, it follows from Proposition \ref{fcSigma} that all the summands
in summation are zero. Hence it is enough to check the assumption of Theorem 7.3 of \cite{GRS11}.
By the cuspidal support of $\CE_{\tau\otimes\sigma}$, if the constant term $f^{U_{p-i}}$ is zero (using the notation of
\cite[Theorem 7.3]{GRS11}, with $f\in\CE_{\tau\otimes\sigma}$), we are done. It remains to consider the cases when the constant terms are
not zero. To this end, we may consider the first nonzero constant term, which reduces to the case with
$\tau=\tau_2\boxplus\cdots\boxplus\tau_r$
of $G_{E/F}(a-a_1)(\BA)$. Here we refer to \eqref{tau8} for notation.
In this case, the index for the Fourier coefficient is $\kappa_0+i$ with $i=0,1,\cdots,p-1$, following the notation of \cite[Theorem 7.3]{GRS11}.
Since $\kappa_0=a-\ell_0-1$ and $\ell_0=\frac{\Fn-\Fm-1}{2}$, we must have 
$$
\kappa_0+i=(a-\frac{\Fn}{2})+\frac{\Fm-1}{2}+i.
$$
On the other hand, the term $(f^{U_{p-i}})^{\psi_{\kappa_0+i,\alpha}}$ is a Fourier coefficient on $H_{a-a_1+m}$ with index
$$
\frac{a-a_1+\Fm+\epsilon-1}{2}
$$
where $\epsilon=-1$ if $G_n$ is an odd special orthogonal group; otherwise, $\epsilon=0$. It follows that
$$
\kappa_0+i>\frac{a-a_1+\Fm+\epsilon-1}{2}.
$$
This is because $\frac{a-\epsilon}{2}=\frac{\Fn}{2}$ and $\frac{a_1}{2}+i>0$.
According to the structure of the global Arthur parameter $\psi_{\tau,\sig}$ as in \eqref{nggap},
the term $(f^{U_{p-i}})^{\psi_{\kappa_0+i,\alpha}}$ must be zero. Namely,
the condition in \cite[Theorem 7.3]{GRS11} holds in this reduced case because of Proposition \ref{fcSigma}. Hence by induction,
we obtain that the $\kappa_0$-th Bessel modules $\CF^{\CO_{\kappa_0}}_{n}(\Sigma)$ must also be cuspidal.
This completes the proof.
\end{proof}

It is clear that the $\kappa_0$-th Bessel modules $\CF^{\CO_{\kappa_0}}_{n}(\Sigma)$ could be zero.
We assume that the cuspidal sub-representation $\CF^{\CO_{\kappa_0}}_{n}(\Sigma)$ of
$G_n^{\CO_{\kappa_0}}(\BA)$, occurring in the cuspidal spectrum $L^2_\cusp(G_n^{\CO_{\kappa_0}})$, is non-zero, and write it as a Hilbert direct sum:
\begin{equation}\label{cd}
\CF^{\CO_{\kappa_0}}_{n}(\Sigma)=\pi_1\oplus\cdots\oplus\pi_k\oplus\cdots
\end{equation}
where all $\pi_i$ are irreducible cuspidal automorphic representations of $G_n^{\CO_{\kappa_0}}(\BA)$. In fact, the $\kappa_0$-th Bessel modules
$\CF^{\CO_{\kappa_0}}_{n}(\Sigma)$ are also $\wt{\RO}(G_{\kappa_0^-}^{\CO_{\kappa_0}})$-stable, as indicated in the following proposition. 

\begin{prop}\label{stable}
For $\Sigma\in\CA_\disc(H_{a+m})$, the $\kappa$-th Bessel module $\CF^{\CO_\kappa}_{\kappa^-}(\Sigma)$ is $\wt{\RO}(G_{\kappa^-}^{\CO_\kappa})$-stable.
\end{prop}

\begin{proof}
To prove this, it suffices to consider the case when $H_{a+m}$ is an odd special orthogonal group.
In fact, for all the other cases,
the stabilizer $G_{\kappa^-}^{\CO_\kappa}$ is not an even special orthogonal group
and hence $\wt{\RO}(G_{\kappa^-}^{\CO_\kappa})$ is trivial.

Suppose that $H_{a+m}$ is an odd special orthogonal group. Then $\Fm$ is odd and $\Fm=2m+1$.
In this case, the discrete multiplicity of $\Sigma$ is one. We may take the unique discrete realization $\CC_\Sigma$ of $\Sigma$ in this proof.

By the definition in Section \ref{sec-pif},  $G_{\kappa^-}^{\CO_\kappa}$ is identified as the connected component group
${\rm Isom}(V_{(\kappa)}\cap w^\perp_0,q)^\circ$,
where the anisotropic vector $w_0$ is of form \eqref{eq:w0}, namely,
$$
w_0=e_{a+\Fr_\Fm}+(-1)^{\Fn+1}\frac{x}{2}e_{-(a+\Fr_\Fm)}
$$
for some $x\in F^\times$,
where $\Fr_\Fm=\Fr(H_m)$ is the $F$-rank of $H_m$.
Recall that ${\rm Isom}(V_{(\kappa)},q)^\circ=H_{a+m-\kappa}$ is a subgroup of the Levi subgroup $M_{\hat{\kappa}}$ of $H_{a+m}$.
Assume that $\kappa>0$.
Take the element
$$
\varepsilon=\diag\{-I_{\kappa},I_{a+\Fr_\Fm-\kappa-1},-1,I_{\Fm-2\Fr_\Fm},-1,I_{a+\Fr_\Fm-\kappa-1},-I_{\kappa}\}.
$$
It is easy to check that $\varepsilon\in H_{a+m}$ and it stabilizes $w_0$.
Note that the stabilizer of $w_0$ is $\SO(V_{(\kappa)}\cap w^\perp_0,q)\rtimes \apair{\varepsilon}$,
which is isomorphic to $\RO(V_{(\kappa)}\cap w^\perp_0,q)$.
The adjoint action of $\apair{\varepsilon}$ on $G_{\kappa^-}^{\CO_\kappa}=\SO(V_{(\kappa)}\cap w^\perp_0,q)$ is the same as the action of $\wt{\RO}(G_{\kappa^-}^{\CO_\kappa})$.

Consider the action of $\varepsilon\in H_{a+m}(F)$ on a discrete realization $\CC_\Sigma$ of $\Sigma$,
defined by  $f^\varepsilon(g):=f(\varepsilon^{-1}g\varepsilon)$ for $f\in \CC_\Sigma$.
Since $\varepsilon\in H_{a+m}(F)$,  $f^\varepsilon(g)$ also belongs to $\CC_\Sigma$.
By the definition \ref{fcg}, since $\varepsilon$ stabilizes $\psi_{\kappa,w_0}$,
$$
\CF^{\psi_{\kappa,w_0}}(f^\varepsilon)(h)=\CF^{\psi_{\kappa,w_0}}(f)(\varepsilon^{-1}h \varepsilon):=\CF^{\psi_{\kappa,w_0}}(f)^{\varepsilon}(h),
$$
where the action of $\varepsilon$ on $h\in G_{\kappa^-}^{\CO_\kappa}$ is given as above.
It follows that if $f\in \CC_\Sigma$, then $\CF^{\psi_{\kappa,w_0}}(f)^{\varepsilon}\in \CF^{\psi_{\kappa,w_0}}(\CC_\Sigma)$. As explained in Page \pageref{pg:eps},
the action of $\varepsilon$ on $\CF^{\psi_{\kappa,w_0}}(f)$ coincides the action of $\wt{\RO}(G_{\kappa^-}^{\CO_\kappa})$ on $\CF^{\CO_\kappa}_{\kappa^-}(\CC_\Sigma)$.
That is, $\CF^{\CO_\kappa}_{\kappa^-}(\CC_\Sigma)$ is $\wt{\RO}(G_{\kappa^-}^{\CO_\kappa})$-stable.

When $\kappa=0$, the $\kappa$-th Bessel module $\CF^{\CO_\kappa}_{\kappa^-}(\CC_\Sigma)$ is the restriction of $\CC_\Sigma$ into the even special orthogonal group $\SO_{2a+\Fm-1}(w_0^\perp)$.
Let us extend the representation $\CC_\Sigma$ as the representation of $H_{a+m}\times\apair{-I_{2a+\Fm}}=\RO_{2a+\Fm}(V)$, by letting the action be
trivial on $\apair{-I_{2a+\Fm}}$.
We may choose
$$
\varepsilon=\{I_{a+m}, -1, I_{a+m}\}.
$$
Then the rest of the proof is the same as that for the case $\kappa>0$.
We complete the proof.
\end{proof}

The general calculation of the local Jacquet module of the induced representation of type \eqref{usrp},
as explained in \cite[Chapter 5]{GRS11}, or more precisely, in \cite[Theorems 5.4 and 5.6]{GRS11}, and also as in \cite[Section 4.1]{JLXZ},
can be adopted to prove that those irreducible summands $\pi_i$ in \eqref{cd} 
are actually nearly equivalent to each other, and at almost all local finite places $\nu$, the unramified local component $\pi_{i,\nu}$ of $\pi_i$
shares the same Satake parameter with the unramified local component $\tau_{\nu}$ under the unramified local
Langlands functorial transfer from $G_n^{\CO_{\kappa_0}}(F_{\nu})$
to $G_{E/F}(a)(F_{\nu})$
except that $G_n^{\CO_{\kappa_0}}(F_{\nu})$ is a split even special orthogonal group.
In this case,  the unramified local component $\pi_{i,\nu}$ belongs to the $\wt{\RO}(G_n^{\CO_{\kappa_0}}(F_\nu))$-orbit of the Satake parameters of $G_n^{\CO_{\kappa_0}}(F_{\nu})$,
which are the descents of the Satake parameters of $\tau_\nu$ under the local Langlands functorial transfer.

For the sake of completeness and also for future applications, we apply Theorems 5.4 and 5.6 in \cite{GRS11} to elaborate with some details
the above discussions.
We summarize the results on the local descent at the unramified places as the following lemma. 
\begin{lem}\label{lm:local-descent-unramified}
Let $\nu$ be a finite place such that all data are unramified.
Assume that  $\Sigma_\nu$ is the unique irreducible unramified constituent
of $\Ind^{H_{a+m}(F_\nu)}_{P_{\hat{a}}(F_\nu)}\tau_\nu|\det|^{\frac{1}{2}}\otimes\sig_\nu$,
where  $\tau_\nu$ and $\sig_\nu$ are irreducible, generic and unramified local components of $\tau$ and $\sig$ in \eqref{nggap}.
Then the unramified constituents of $\CF_{n}^{\CO_{\kappa_0}}(\Sigma_\nu)$ have the same Satake parameter with $\tau_\nu$ under 
the local Langlands functorial transfer from $G_n^{\CO_{\kappa_0}}(F_\nu)$ to $G_{E/F}(F_\nu)$.
\end{lem}

\begin{proof}

We consider two cases:
Case (1) $H_m$ is  special even orthogonal,
or $H_m$ is split odd orthogonal and $G^{\CO_\kappa}_{\kappa^-}$ is split,
or $H_m$ is quasi-split odd unitary;
and Case (2) $H_m$ is split special odd orthogonal and  $G^{\CO_\kappa}_{\kappa^-}$ is non-split, or $H_m$ is quasi-split even unitary.

{\bf Case (1).}\
Under the assumption,
if $H_m$ is  split special even orthogonal, then the assumption that the Witt index of $E_\nu y_{-\alpha}+V_0$ is zero in \cite[Theorem 5.4 (1)]{GRS11} holds;
if $H_m$ is  split special odd orthogonal and $G^{\CO_\kappa}_{\kappa^-}$ is split,
then the Witt index of $E_\nu y_{-\alpha}+V_0$ is one, $\delta^1_{h(V),\alpha}=0$ and $\delta^2_{h(V),\alpha}=1$ in the notation of \cite[Theorem 5.4 (2)]{GRS11};
otherwise,  the Witt index of $E_\nu y_{-\alpha}+V_0$ is one, and  $\delta^1_{h(V),\alpha}=1$ and $\delta^2_{h(V),\alpha}=0$ in \cite[Theorem 5.4 (2)]{GRS11}.
In Case (1), $a$ is even by the parity of the dimension of the Arthur parameters involved.

As in \eqref{usrp}, we consider the unramified local component $\Sigma_\nu$ as the unramified constituent of 
\begin{equation}\label{eq:Ind-H-GRS}
\Ind^{H_{a+m}(F_{\nu})}_{P_{[2^{\frac{a}{2}}]}(F_{\nu})}
(\mu_1\circ{\det}_2)\otimes\cdots\otimes(\mu_{\frac{a}{2}}\circ{\det}_2)\otimes
\delta_1\otimes\cdots\otimes\delta_{\Fr_m}.	
\end{equation}
Here $\Fr_m=\Fr(H_m)$ is the $F_\nu$-rank of $H_m$. It follows that $\Fr_m=m-1$ if $H_m(F_\nu)$ is quasi-split and non-split even orthogonal;
and $\Fr_m=m$ otherwise. As before,
$P_{[2^{\frac{a}{2}}]}$ is the standard parabolic subgroup of $H_{a+m}$ whose Levi part is isomorphic to $G_{E/F}(2)^{\times  \frac{a}{2} } \times G_{E/F}(1)^{\times \Fr_m}$;
and $\sig_\nu=\Ind^{H_{m}}_{B_m}\delta_1\otimes\cdots\otimes\delta_{\Fr_m}$.
Let us substitute the following representation for $\tau$ in \cite[Theorem 5.4]{GRS11}
\begin{equation}\label{eq:tau-GRS}
\Ind^{G_{E/F}(a+\Fr_m)(F_{\nu})}_{Q_{[2^{ \frac{a}{2}}]}(F_{\nu})}(\mu_1\circ{\det}_2)\otimes\cdots\otimes(\mu_{ \frac{a}{2}}\circ{\det}_2)\otimes
\delta_1\otimes\cdots\otimes\delta_{\Fr_m},
\end{equation}
where $Q_{[2^{ \frac{a}{2}}]}$ is the standard parabolic subgroup $G_{E/F}(a+\Fr_m)\cap P_{[2^{\frac{a}{2}}]}$.
We regard $G_{E/F}(a+\Fr_m)$ as the subgroup of the standard parabolic subgroup of $P_{(a+\Fr_m)^{\wedge}}$.
The symbols $\tilde{m}$ and $\ell$ in \cite[Theorem 5.4]{GRS11}  are replaced by $a+\Fr_m$ and $\kappa$ respectively in our case.
Then $0\leq \kappa<a+\Fr_m$, which is a part of the conditions in \cite[Theorem 5.4, (1) and (2)]{GRS11}.

If $H_m$ is split special even orthogonal, by \cite[Theorem 5.4 (1)]{GRS11}, one has
$\CF^{\CO_\kappa}_{\kappa^-}(\Sigma_\nu)=0$ for $\kappa> \frac{a}{2}+\Fr_m-1$, and for $\kappa=\frac{a}{2}+\Fr_m-1$
$$
\CF^{\CO_\kappa}_{\kappa^-}(\Sigma_\nu)\prec\Ind^{G^{\CO_\kappa}_{\kappa^-}}_{B_{G,a}}\mu_1\otimes\mu_2\cdots\otimes\mu_{\frac{a}{2}},
$$
where $B_{G,a}$ is the Borel subgroup of $G^{\CO_\kappa}_{\kappa^-}$,
and $\pi_1\prec\pi_2$ means that $\pi_1$ is a constituent of $\pi_2$.
Since $\Fr_m=m$, we have that $G^{\CO_\kappa}_{\kappa^-}$ is isomorphic to the split odd orthogonal $\SO_{a+1}$.

If $H_m$ is split odd orthogonal and $G^{\CO_\kappa}_{\kappa^-}$ is split, by \cite[Theorem 5.4 (2)]{GRS11}, one has
$\CF^{\CO_\kappa}_{\kappa^-}(\Sigma_\nu)=0$ for $\kappa> \frac{a}{2}+\Fr_m$, and for $\kappa=\frac{a}{2}+\Fr_m-1$
$$
\CF^{\CO_\kappa}_{\kappa^-}(\Sigma_\nu)\prec\Ind^{G^{\CO_\kappa}_{\kappa^-}}_{B_{G,a}}\mu_1\otimes\mu_2\cdots\otimes\mu_{\frac{a}{2}-1}\otimes\mu_{\frac{a}{2}}
\oplus \Ind^{G^{\CO_\kappa}_{\kappa^-}}_{B_{G,a}}\mu_1\otimes\mu_2\cdots\otimes\mu_{\frac{a}{2}-1}\otimes\mu^{-1}_{\frac{a}{2}}.
$$
Since $\Fr_m=m$, we have that $G^{\CO_\kappa}_{\kappa^-}$ is isomorphic to the split special even orthogonal $\SO_{a}$ and that the two unramified representations are $\wt{\RO}(G_{\kappa^-}^{\CO_{\kappa_0}}(F_\nu))$-conjugate.

If $H_m$ is quasi-split, but non-split, special even orthogonal or odd unitary, following \cite[Theorem 5.4 (2)]{GRS11}, one has
$\CF^{\CO_\kappa}_{\kappa^-}(\Sigma_\nu)=0$ for $\kappa> \frac{a}{2}+\Fr_m$, and for $\kappa=\frac{a}{2}+\Fr_m$
$$
\CF^{\CO_\kappa}_{\kappa^-}(\Sigma_\nu)\prec\Ind^{G^{\CO_\kappa}_{\kappa^-}}_{B_{G,a}}\mu_1\otimes\mu_2\cdots\otimes\mu_{\frac{a}{2}}.
$$
More precisely, if $H_m$ is quasi-split, but non-split, special even orthogonal,
then $\Fr_m=m-1$ and $G^{\CO_\kappa}_{\kappa^-}$ isomorphic to the split odd orthogonal $\SO_{a+1}$; and
if $H_m$ is odd unitary,
then $\Fr_m=m$ and $G^{\CO_\kappa}_{\kappa^-}$ is isomorphic to the quasi-split odd unitary $\RU_{a+1}$.

{\bf Case (2).}\
Denote $\omega_{\tau,\nu}$ to be the central character of $\tau_\nu$.
Since $\tau_\nu$ is  (conjugate) self-dual, $\omega_{\tau,\nu}$ is a quadratic character, that is $\omega_{\tau,\nu}=1$ or $\lam_0$.
Here $\lam_0$ is  the unique non-trivial unramified quadratic character of $E^\times_\nu$.

Assume that $H_m$ is split odd orthogonal and $G^{\CO_\kappa}_{\kappa^-}$ is non-split.
In this case, the assumption that the Witt index of $F_\nu y_{-\alpha}+V_0$ is zero in \cite[Theorem 5.4 (1)]{GRS11} holds and $a$ is even.

If $\omega_{\tau,\nu}=1$, then the unramified local component $\Sigma_\nu$ is the unramified constituent of the unramified induced representation as in \eqref{eq:Ind-H-GRS}. 
We substitute the representation \eqref{eq:tau-GRS} for $\tau$ in \cite[Theorem 5.4 (1)]{GRS11}.
Then $\CF^{\CO_\kappa}_{\kappa^-}(\Sigma_\nu)=0$ for $\kappa\geq \frac{a}{2}+m$
and the descent $\CF^{\CO_\kappa}_{\kappa^-}(\Sigma_\nu)$ to the orthogonal group $\SO_{a}$ is zero when $\kappa= \frac{a}{2}+m$.
Remark that over the innert finite places, the determinant of the local $L$-parameter of the quasi-split, but non-split $\SO_{a}$ is not 1.
This verifies that if $\omega_{\tau,\nu}=1$ the descent $\CF^{\CO_\kappa}_{\kappa^-}(\Sigma_\nu)$ at this rational orbit $\CO_\kappa$ is zero.

Assume that $\omega_{\tau,\nu}=\lam_0$.
The unramified local component $\Sigma_\nu$ is isomorphic to the unramified constituent of
$$
\Ind^{H_{a+m}(F_{\nu})}_{P_{[2^{\frac{a}{2}-1}]}(F_{\nu})}
(\mu_1\circ{\det}_2)\otimes\cdots\otimes(\mu_{\frac{a}{2}-1}\circ{\det}_2)\otimes 1\otimes\lam_0 \otimes
\delta_1\otimes\cdots\otimes\delta_{m}.	
$$
We may replace the above representation by
$
\Ind^{H_{a+m}(F_{\nu})}_{P_{(a+m-2)^\wedge}(F_{\nu})}\tau_1\otimes\sig_1,
$
where $\sig_1$ is the representation of split $\SO_5(F_\nu)$ induced from the parabolic subgroup which preserves an isotropic line and the character $\lam_0|\cdot|^{\frac{1}{2}}\otimes 1$,
and
$$
\tau_1=\Ind^{G_{E/F}(a+m-2)(F_{\nu})}_{Q_{[2^{ \frac{a}{2}}]}(F_{\nu})}(\mu_1\circ{\det}_2)\otimes\cdots\otimes(\mu_{ \frac{a}{2}-1}\circ{\det}_2)\otimes
\delta_1\otimes\cdots\otimes\delta_{m}.
$$
Applying \cite[Theorem 5.1 (1)]{GRS11}, after the same calculation with Page 104 in \cite{GRS11}, one has
$\CF^{\CO_\kappa}_{\kappa^-}(\Sigma_\nu)=0$ for $\kappa> \frac{a}{2}+m$,
and for $\kappa=\frac{a}{2}+m-1$
$$
\CF^{\CO_\kappa}_{\kappa^-}(\Sigma_\nu)\prec\Ind^{G^{\CO_\kappa}_{\kappa^-}}_{B_{G,a}}\mu_1\otimes\mu_2\cdots\otimes\mu_{\frac{a}{2}-1}\otimes 1,
$$
where $1$ is the trivial representation of the anisotropic part of the torus of $B_{G,a}$.

In the rest, suppose that $H_m$ is quasi-split even unitary. 
Then in our setting $a$ is odd, $\tau_\nu$ is conjugate orthogonal and 
 $\omega_{\tau,\nu}=1$.
The unramified local component $\Sigma_\nu$ is isomorphic to the unramified constituent of
$$
\Ind^{H_{a+m}(F_{\nu})}_{P_{[2^{[\frac{a}{2}]}]}(F_{\nu})}
(\mu_1\circ{\det}_2)\otimes\cdots\otimes(\mu_{[\frac{a}{2}]}\circ{\det}_2)\otimes  |\cdot|^{\frac{1}{2}} \otimes
\delta_1\otimes\cdots\otimes\delta_{m}.
$$
We may replace the above induced representation by
$$
\Ind^{H_{a+m}(F_{\nu})}_{P_{(a+m-1)^\wedge}(F_{\nu})}\tau_1\otimes 1,
$$
where $1$ is the trivial character of quasi-split $\RU_2(F_\nu)$
and
$$
\tau_1=\Ind^{G_{E/F}(a+m-1)(F_{\nu})}_{Q_{[2^{ [\frac{a}{2}]}]}(F_{\nu})}(\mu_1\circ{\det}_2)\otimes\cdots\otimes(\mu_{ [\frac{a}{2}]}\circ{\det}_2)\otimes
\delta_1\otimes\cdots\otimes\delta_{m}. \label{page:tau-last}
$$
Let us apply \cite[Theorem 5.1 (1)]{GRS11} and follow the same calculation with  Page 105 in the proof of Theorem 5.6 in \cite{GRS11} for the case $\omega_{\tau,\nu}=1$.
Then one has
$\CF^{\CO_\kappa}_{\kappa^-}(\Sigma_\nu)=0$ for $\kappa> \frac{a}{2}+m$.
For $\kappa=[\frac{a}{2}]+m$, if $\omega_{\tau,\nu}=1$, then
\begin{equation}\label{eq:desc-unitary-even}
\CF^{\CO_\kappa}_{\kappa^-}(\Sigma_\nu)\prec\Ind^{G^{\CO_\kappa}_{\kappa^-}}_{B_{G,a}}\mu_1\otimes\mu_2\cdots\otimes\mu_{[\frac{a}{2}]-1}\otimes 1,	
\end{equation}
where $1$ is the trivial character of $\RU_1(F_\nu)$.

Therefore, we complete all cases involved in our discussion in this paper, and verify that
 $\pi_{i,\nu}$  shares the same Satake parameter with  $\tau_{\nu}$ under the unramified local Langlands functorial transfer from $G_n^{\CO_{\kappa_0}}(F_{\nu})$
to $G_{E/F}(a)(F_{\nu})$.
\end{proof}

Now we go back to the decomposition in \eqref{cd}.
By the local uniqueness of Bessel models at all local places (\cite{AGRS}, \cite{SZ}, \cite{GGP12}, and \cite{JSZ}),
it is easy to deduce that $\pi_i$ is not equivalent to $\pi_j$ if $i\neq j$, that is, the decomposition in \eqref{cd} is multiplicity free.
Of course, in the situation that the cuspidal spectrum is multiplicity free, the decomposition in \eqref{cd} will be automatically multiplicity
free.
We summarize the discussion as the following theorem.

\begin{thm}\label{wlt}
Assume that $\tau$ and $\sigma$ are as given above.
For an automorphic member $\Sigma$ in the global Arthur packet $\wt{\Pi}_{\psi_{\tau,\sig}}(H_{a+m})$,
assume that the $\kappa_0$-th Bessel module $\CF^{\CO_{\kappa_0}}_{n}(\Sigma)$ is non-zero
for some $F$-rational nilpotent orbit $\CO_{\kappa_0}$ in the $F$-stable orbit $\CO^\st_{\udl{p}_{\kappa_0}}(F)$
with $\ell_0=\frac{\Fn-\Fm-1}{2}$, $\kappa_0=a-\ell_0-1$, and 
 $\kappa_0^-=m^-=n$. Then the following hold:
\begin{enumerate}
\item The $\kappa_0$-th Bessel module $\CF^{\CO_{\kappa_0}}_{n}(\Sigma)$ is cuspidal and can be regarded
as a sub-representation of $G_n^{\CO_{\kappa_0}}(\BA)$ in the cuspidal spectrum $L^2_\cusp(G_n^{\CO_{\kappa_0}})$.
\item In the cuspidal spectrum $L^2_\cusp(G_n^{\CO_{\kappa_0}})$, $\CF^{\CO_{\kappa_0}}_{n}(\Sigma)$ has a multiplicity free, Hilbert direct
sum decomposition
$$
\CF^{\CO_{\kappa_0}}_{n}(\Sigma)=\pi_1\oplus\cdots\oplus\pi_k\oplus\cdots
$$
where all $\pi_i$ belong to $\CA_\cusp(G_n^{\CO_{\kappa_0}})$, and
have a generic global Arthur parameter belonging to the $\wt{\RO}(G_n^{\CO_{\kappa_0}})$-orbit of $\phi_\tau$, which
is $G_n^{\CO_{\kappa_0}}$-relevant and is determined by $\tau$.
Moreover, $\CF^{\CO_{\kappa_0}}_{n}(\Sigma)$ is $\wt{\RO}(G_n^{\CO_{\kappa_0}})$-stable.
\end{enumerate}
\end{thm}

Note that in Part (2) of Theorem \ref{wlt}, the $\wt{\RO}(G_n^{\CO_{\kappa_0}})$-orbit of $\phi_\tau$ contains only one parameter unless $G_n^{\CO_{\kappa_0}}$ is an even special orthogonal group.
In this case, the $\wt{\RO}(G_n^{\CO_{\kappa_0}})$-orbit of $\phi_\tau$ may contain two parameters $\{\phi,\phi_\star\}$, which are
the descents of $\phi_\tau$, as explained in
Page \pageref{pg:eps}. It is worthwhile to remind that in this case, the global Arthur packets $\Pi_\phi(G_n^{\CO_{\kappa_0}})$ and
$\Pi_{\phi_\star}(G_n^{\CO_{\kappa_0}})$ are different.
We may identify the parameter $\phi_\tau$ with either $\phi$ or $\phi_\star$, as in Page \pageref{pg:eps}.

\subsection{Construction of cuspidal automorphic modules}\label{sec-mccam}
The main issue remains from Theorem \ref{wlt} is the non-vanishing assumption that the $\kappa_0$-th Bessel module $\CF^{\CO_{\kappa_0}}_{n}(\Sigma)$ is non-zero for some automorphic member 
$\Sigma$ in the global Arthur packet $\wt{\Pi}_{\psi_{\tau,\sig}}(H_{a+m})$. We refer to Subsection \ref{sec-wfs} and Conjecture \ref{gnvc} in particular for more details. 

We are instead going to discuss the impact of Conjecture \ref{bpconj} in the theory of twisted automorphic descents. To this end, we recall the specific data suggested by Conjecture \ref{bpconj}.
By Proposition \ref{cfc}, for a given $\pi\in\CA_\cusp(G_n)$ with a $G_n$-relevant, generic global Arthur parameter $\phi=\phi_\tau\in\wt{\Phi}_2(G_n^*)$ and with the cuspidal realization $\CC_\pi$, 
there exists the first occurrence index $\ell_0=\ell_0(\CC_\pi)$, such that the (maximal) $\ell_0$-Bessel module $\CF^{{\CO_{\ell_0}}}(\CC_\pi)$ associated to the 
partition $\udl{p}_{\ell_0}=[(2\ell_0+1)1^{\Fn-2\ell_0-1}]$ is cuspidal and 
nonzero, as a representation of $H_{m}^{\CO_{\ell_0}}(\BA)$ occurring in the cuspidal spectrum $L^2_\cusp(H_{m}^{\CO_{\ell_0}})$. In this situation, we take the data that $m=\ell_0^-$, $\Fm=\Fl_0^-=\Fn-2\ell_0-1$, 
and $H_m=H_{m}^{\CO_{\ell_0}}$. By Conjecture \ref{bpconj}, there exists a $\sig\in\CA_\cusp(H_m)$ with an $H_m$-relevant, generic global Arthur parameter $\phi_\sig$ in $\wt{\Phi}_2(H_m^*)$ and 
with the cuspidal realization $\CC_\sig$, such that the inner product $\left<\CF^{\psi_{\CO_{\ell_0}}}(\varphi_\pi),\ol{\varphi}_{\sigma'}\right>_{H_m}$ is nonzero for some $\varphi_\pi\in\CC_\pi$ and 
$\varphi_{\sig'}\in\CC_{\sig'}$, 
where $\sig'=\sig^{w_q^\ell}$ and $\CC_{\sig'}=\CC_\sig^{w_q^\ell}$ defined in \eqref{eq:sig'}.


With the data associated to Conjecture \ref{bpconj}, Theorem \ref{wlt} may be illustrated by the following diagram:
\begin{equation}\label{diag}
\begin{matrix}
\wt{\Phi}_2(G_n^*)_{G_n}& &\phi_\sigma\in\wt{\Phi}_2(H_m)  &&&&\wt{\Pi}_{\psi_{\tau,\sig}}(H_{a+m})& \\
                        & & &&\Longrightarrow&&&\\
\phi_\tau               & &(H_m,\sigma)&&&&\Sigma&\\
             &   &&&&&&\\
\Updownarrow   &&\Uparrow&&&&\Downarrow &\\
 &&&&&&&\\
\wt{\Pi}_{\phi_\tau}(G_n)   & &\ni\pi &       & \underleftrightarrow{(?)}       &           &\CF_n^{\CO_{\kappa_0}}(\Sigma)\subset L^2_\cusp(G_n^{\CO_{\kappa_0}})&\\
\end{matrix}
\end{equation}
In this Diagram, we starts with a generic global Arthur parameter $\phi_\tau$ of $G_n^*$, which is $G_n$-relevant. It gives the global Arthur packet
$\wt{\Pi}_{\phi_\tau}(G_n)$ by the endoscopic classification theory. Now take any cuspidal member $\pi$ in $\wt{\Pi}_{\phi_\tau}(G_n)$. By the
Generic Summand Conjecture (Conjecture \ref{bpconj}), it produces the pair $(H_m,\sigma)$, where $\sig\in\CA_\cusp(H_m)$ with a generic global
Arthur parameter $\phi_\sig$ in $\wt{\Phi}_2(H_m^*)_{H_m}$. Then $\sig$ and $\tau$ together produce the global Arthur packet
$\wt{\Pi}_{\psi_{\tau,\sig}}(H_{a+m})$. Finally, we take the Bessel-Fourier coefficient $\CF_n^{\CO_{\kappa_0}}(\Sigma)$ for any automorphic member
$\Sigma$ in $\wt{\Pi}_{\psi_{\tau,\sig}}(H_{a+m})$, which is a cuspidal automorphic $G_n^{\CO_{\kappa_0}}(\BA)$-module contained in
$L^2_\cusp(G_n^{\CO_{\kappa_0}})$.
The {\sl big question} in the construction is: what can we say about $\pi$ and $\CF_n^{\CO_{\kappa_0}}(\Sigma)$ as
representations of $G_n(\BA)$ and $G_n^{\CO_{\kappa_0}}(\BA)$, respectively? Theorem \ref{wlt} gives an answer to this question with a non-vanishing assumption. 

Without the participation of $\sigma$ and $H_m$, Diagram \eqref{diag} may be reduced to the following diagram:
\begin{equation}\label{diag2}
\begin{matrix}
\wt{\Phi}_2(G_n^*)\ni\phi_\tau  &&\Longrightarrow&&\CE_\tau\in\wt{\Pi}_{\psi_{\tau}}(H_{a}^*)& \\
             &   &&&&\\
\Downarrow   &&&&\Downarrow &\\
 &&&&&\\
\wt{\Pi}_{\phi_\tau}(G_n)\ni\pi &       & \cong       &           &\CF_n^{\CO_{\kappa_0}}(\CE_\tau)\subset L^2_\cusp(G_n^{\CO_{\kappa_0}})&\\
\end{matrix}
\end{equation}
When $G_n=G_n^{\CO_{\kappa_0}}=G_n^*$ is $F$-quasisplit, and  $\CE_\tau$ is the residual representation of $H_a(\BA)$
(with $a=N=\Fn^\vee$) having the global Arthur parameter
$$
\psi_\tau=(\tau_1,2)\boxplus\cdots\boxplus(\tau_r,2),
$$
This reduced diagram (Diagram \eqref{diag2}) yields the automorphic descents of Ginzburg-Rallis-Soudry (\cite{GRS11}) that construct certain generic
cuspidal automorphic representations of an $F$-quasisplit classical group $G_n^*(\BA)$.

By Proposition \ref{piform}, $G_n$ and $G_n^{\CO_{\kappa_0}}$
are pure inner forms. If one of $G_n$ and $G_n^{\CO_{\kappa_0}}$ is not equal to $G_n^*$, then the relation between
$\pi$ and $\CF_n^{\CO_{\kappa_0}}(\Sigma)$ is the generalized Jacquet-Langlands correspondence between $G_n$ and $G_n^{\CO_{\kappa_0}}$. However,
as shown in \cite{JLXZ}, this will not cover the general situation as the $F$-ranks of $G_n$ and $G_n^{\CO_{\kappa_0}}$ must satisfy the condition
given in
Proposition \ref{coform}. The introduction of $\sigma$ and $H_m$ in the construction is to avoid such restriction.

With Conjecture \ref{bpconj} and the participation of $\sigma$ and $H_m$ in the construction as displayed in Diagram \eqref{diag}, which is essential,
the proposed construction may (in principle) produce all irreducible cuspidal
automorphic representations of the classical groups $G_n$ that are pure inner $F$-forms of an $F$-quasisplit classical group $G_n^*$.

In fact, we are going to show in Subsection \ref{sec-mcgc} that the $\kappa_0$-th Bessel module $\CF^{\CO_{\kappa_0}}_{n}(\CE_{\tau\otimes\sigma})$ is non-zero, assuming Conjecture \ref{bpconj}.
If we assume that the stronger uniqueness of the local Bessel models over a local Vogan packet holds at all local places (Conjecture \ref{smoconj}, the known
cases of which is given in Theorem \ref{mwslmo}),
then $\CF^{\CO_{\kappa_0}}_{n}(\CE_{\tau\otimes\sigma})$ is in fact irreducible, when $G_n=G_n^{\CO_{\kappa_0}}$ is {\sl not} an even special orthogonal
group. However, if $G_n=G_n^{\CO_{\kappa_0}}$ is an even special orthogonal group, then $\CF^{\CO_{\kappa_0}}_{n}(\CE_{\tau\otimes\sigma})$
could be a direct sum of two irreducible cuspidal automorphic representations that belong to the $\wt{\RO}(G_n)$-orbit.
In any situation, we set
\begin{equation}\label{tad1}
{\mathcal D}^{\CO_{\kappa_0}}_n(\tau;\sigma):=\CF^{\CO_{\kappa_0}}_{n}(\CE_{\tau\otimes\sigma})
\end{equation}
and call ${\mathcal D}^{\CO_{\kappa_0}}_n(\tau;\sigma)$ a {\sl $\sigma$-twisted automorphic descent} of $\tau$ from
$G_{E/F}(N)$ to $G_n^{\CO_{\kappa_0}}$, or simply a {\sl twisted automorphic descent} of $\tau$, where $N=a=\Fn^\vee$.
The main result in the theory of the {\sl cuspidal automorphic modules} outlined in Diagram \eqref{diag}, by means of {\sl twisted automorphic descents}, is to confirm that the constructed
module ${\mathcal D}^{\CO_{\kappa_0}}_n(\tau;\sigma)$ in \eqref{tad1} is in principle isomorphic to the given irreducible
cuspidal automorphic representation $\pi$. 

In general, we may state the {\sl main conjecture} in the theory of the cuspidal automorphic modules via the twisted automorphic descents as follows.

\begin{conj}[Main Conjecture]\label{pmc}
Let $\tau=\tau_1\boxplus\cdots\boxplus\tau_r$ be an irreducible isobaric representation of
$G_{E/F}(a)(\BA)$ that defines a generic global Arthur parameter $\phi=\phi_\tau\in\wt{\Phi}_2(G_n^*)$ as in \eqref{gapgn}.
Assume that $\phi$ is $G_n$-relevant.
For any $\pi\in\CA_\cusp(G_n)$ belonging to the global Arthur packet $\wt{\Pi}_\phi(G_n)$, there exists a datum $(H_m,\sigma)$
with the following properties:
\begin{enumerate}
\item $H_m$ is a classical group defined over $F$ and is a pure inner $F$-form of an $F$-quasisplit classical group $H_m^*$ such that
 the pairs $(G_n,H_m)$ and $(G_n^*,H_m^*)$ are relevant and the product $G_n\times H_m$ is a relevant pure inner form of the
 product $G_n^*\times H_m^*$; and
\item $\sigma\in\CA_\cusp(H_m)$ belongs to the global Arthur packet $\wt{\Pi}_{\phi'}(H_m)$
associated to an $H_m$-relevant generic global Arthur parameter $\phi'\in\wt{\Phi}_2(H_m^*)$,
\end{enumerate}
such that
\begin{itemize}
\item[(a)] if $G_n=G_n^{\CO_{\kappa_0}}$ is not an even special orthogonal group, or if $G_n=G_n^{\CO_{\kappa_0}}$ is an even special orthogonal group,
but the $\wt{\RO}(G_n)$-orbit of $\pi$ contains only $\pi$, then
the automorphic module ${\mathcal D}^{\CO_{\kappa_0}}_n(\tau;\sigma)$ that is constructed via
the twisted automorphic descent \eqref{tad1} is isomorphic to the given $\pi$:
$$
{\mathcal D}^{\CO_{\kappa_0}}_n(\tau;\sigma)\cong\pi;
$$
\item [(b)] if  $G_n=G_n^{\CO_{\kappa_0}}$ is an even special orthogonal group, and the $\wt{\RO}(G_n)$-orbit of $\pi$ is equal to $\{\pi,\pi_\star\}$, then
$$
{\mathcal D}^{\CO_{\kappa_0}}_n(\tau;\sigma)\cong\pi\oplus\pi_\star.
$$
\end{itemize}
\end{conj}

It is not hard to see that the constructed cuspidal automorphic module ${\mathcal D}^{\CO_{\kappa_0}}_n(\tau;\sigma)$ in Conjecture \ref{pmc} is the special
realization of the module $\CM(\psi, \CF(\pi,G))$ in Principle \ref{prin} in the particular case under consideration. We remark that
the construction outlined in Diagram \eqref{diag} only uses a piece of information from the data $\CF(\pi,G)$.
We will come back to the discussion of Conjecture \ref{pmc} with more details in Section \ref{sec-nmc}.

\subsection{Global Gan-Gross-Prasad conjecture: another direction}\label{sec-wfs}
Let $\tau=\tau_1\boxplus\tau_2\boxplus\cdots\boxplus\tau_r$ be the irreducible isobaric automorphic representation of
$G_{E/F}(a)(\BA)$ as in \eqref{tau8}, with $a=\Fn^\vee=N$, which defines
a generic global Arthur parameter $\phi=\phi_\tau$ in $\wt{\Phi}_2(G_n^*)$. Let $\phi'$ be a generic global Arthur parameter of $H_m^*$.
Assume that $L(\frac{1}{2},\phi\times\phi')\neq 0$. For any member $\sig$ in the global Vogan packet $\wt{\Pi}_{\phi'}[H_m^*]$,
in which all the automorphic members are cuspidal (\cite[Section 3]{JL-Cogdell}),
we have
$$
L(\frac{1}{2},\tau\times\sigma)=L(\frac{1}{2},\tau\times\sig^{w^\ell_q})=L(\frac{1}{2},\phi\times(\phi')^{w^\ell_q})=L(\frac{1}{2},\phi\times\phi')\neq 0.
$$
The {\sl direction of the global Gan-Gross-Prasad conjecture} to be considered in this section
asserts that under the above assumptions, there exists a unique pair $(\pi,\sig)$ in the global Vogan packet
$\wt{\Pi}_{\phi\times\phi'}[G_n^*\times H_m^*]$ with property that $(\pi,\sig)$ with the cuspidal realization $(\CC_\pi,\CC_\sig)$
admits a nonzero Bessel period
(depending on the $F$-rational structure of the unipotent orbits as discussed in Section \ref{sec-pif}). The uniqueness follows from the local
Gan-Gross-Prasad conjecture at all local places (Conjecture \ref{smoconj}). Hence the key point is the existence of such a pair with a nonzero Bessel period.

We are going to prove this direction of the global Gan-Gross-Prasad conjecture by constructing
such a pair via the twisted automorphic descent developed in the early sections of this paper. For a technical reason, we have to take an {\sl assumption},
which we are only able to verify for some special situation for the time being.

Take a member $\sig\in \wt{\Pi}_{\phi'}[H_m^*]$. There is an $F$-inner form $H_m$ of $H_m^*$ such that $\sig\in\CA_\cusp(H_m)$ with a cuspidal realization
$\CC_\sig$. By Proposition \ref{esp}, the residual representation $\CE_{\tau\otimes\sig}$ of $H_{a+m}(\BA)$ is nonzero.
As discussed in \cite{JLZ13}, $\CE_{\tau\otimes\sig}$ is square-integrable. By \cite[Theorem B]{M11}, $\CE_{\tau\otimes\sig}$ is
irreducible. Following from \cite[Section 6]{JLZ13}, $\CE_{\tau\otimes\sig}$ has the global Arthur parameter
$$
\psi_{\tau,\sig}=(\tau_1,2)\boxplus(\tau_2,2)\boxplus\cdots\boxplus(\tau_r,2)\boxplus\phi'.
$$
It is expected that the structure of the Fourier coefficients of $\CE_{\tau\otimes\sig}$ has significant impact to the understanding of the
global Vogan packet $\wt{\Pi}_{\phi\times\phi'}[G_n^*\times H_m^*]$.

As in \cite[Section 4]{J14} and as recalled in Section \ref{sec-fcpt},
the Fourier coefficients of $\CE_{\tau\otimes\sigma}$ are defined in terms of the
$H_{a+m}$-relevant partitions of $(2a+\Fm,H_{a+m}^*)$.
We denote by $\mathfrak{p}(\CE_{\tau\otimes\sigma})$ the set of the $H_{a+m}$-relevant partitions with which the
residual representation $\CE_{\tau\otimes\sigma}$ has a nonzero Fourier coefficient.
To the pair of the generic global Arthur parameters $(\phi,\phi')$ as given above,
we define the following partition:
$$
\udl{p}_{\phi,\phi'}:=
\begin{cases}
[(a+\Fm-1)(a+1)]&\textit{if}\ H_m^*=\SO_{2m}, \Fm=2m;\\
[(a+\Fm)(a-1)1]&\textit{if}\ H_m^*=\SO_{2m+1}, \Fm=2m+1;\\
[(a+\Fm)a]&\textit{if}\ H_m^*\ \textit{is a unitary group}.
\end{cases}
$$
Note that $a=\Fn^\vee$, and the integers $a+\Fm-1$ and $a+\Fm$ are odd, in the respective cases.
The main conjecture in \cite[Section 4]{J14} asserts that
for all $\sig\in \wt{\Pi}_{\phi'}[H_m^*]$, every partition $\udl{p}\in\mathfrak{p}(\CE_{\tau\otimes\sigma})$ has the property that
$\udl{p}\leq \udl{p}_{\phi,\phi'}$.
If we get back to the construction of cuspidal automorphic modules as illustrated in Diagram \eqref{diag}, then we need the following partition:
$$
\udl{p}_{\phi,\phi'}^1:=
\begin{cases}
[(a+\Fm-1)1^{a+1}]&\textit{if}\ H_m^*=\SO_{2m};\\
[(a+\Fm)1^a]&\textit{otherwise}.
\end{cases}
$$

\begin{conj}\label{gnvc}
With notation and assumptions as above, for the given pair of parameters $(\phi,\phi')$, there exists a $\sigma\in \wt{\Pi}_{\phi'}[H_m^*]$ with a cuspidal realization $\CC_\sig$ on $H_m(\BA)$, such that
$\udl{p}_{\phi,\phi'}^1$ belongs to $\mathfrak{p}(\CE_{\tau\otimes\sigma})$, where $\CE_{\tau\otimes\sigma}$ on $H_{a+m}(\BA)$
is defined through $\CC_\sig$.
\end{conj}

For $m=0$, Conjecture \ref{gnvc} was proved in \cite{GRS11}. For $m=1$ and $H_1$ is an $F$-form of $\SO_2$, it is proved in \cite{JLXZ}.
Similar results can be checked for unitary groups, but we do not discuss them here with further details.

\begin{prop}\label{subrnv}
Conjecture \ref{gnvc} holds when $m=0$ and for all $F$-quasisplit classical groups $H_a^*$, and when $m=1$ for even special orthogonal groups
$H_{2n+1}=\SO_{2n+2,2n}$.
\end{prop}

We refer to \cite{J14}, \cite{J17}, \cite{JL-Cogdell} and \cite{JLS} for more discussion of Fourier coefficients of automorphic representations occurring
in the discrete spectrum of classical groups, and of residual representations in particular.

\begin{thm}[Global Gan-Gross-Prasad Conjecture: another direction]\label{gggp2}
Let $\tau$ be the irreducible isobaric automorphic representation of $G_{E/F}(a)(\BA)$ as in \eqref{tau8}, with $a=\Fn^\vee=N$ and
a generic global Arthur parameter $\phi=\phi_\tau$ in $\wt{\Phi}_2(G_n^*)$. Let $\phi'$ be a generic global Arthur parameter of $H_m^*$.
Assume that
$$
L(\frac{1}{2}, \phi\times\phi')\neq 0.
$$
Assume that Conjecture \ref{gnvc} holds for the pair of parameters $(\phi,(\phi')^{w^\ell_q})$.
Then there exist a cuspidal automorphic member $\pi$ in the global Vogan packet $\wt{\Pi}_\phi[G_n^*]$ with a cuspidal realization $\CC_\pi$, and a cuspidal automorphic member $\sig$ in the global Arthur packet
$\wt{\Pi}_{\phi'}(H_m)$ with a cuspidal realization $\CC_\sig$, such that
the pair $(\pi,\sig)$ belongs to the global Vogan packet $\wt{\Pi}_{\phi\times\phi'}[G_n^*\times H_m^*]$ and
the inner product
$$
\left<\CF^{\psi_{\CO_{\ell_0}}}(\varphi_{\pi}),\varphi_{\sigma}\right>_{H_m}\neq 0
$$
for some $\varphi_{\pi}\in\CC_\pi$ and $\varphi_{\sigma}\in\CC_\sigma$, where $\ell_0^-=m$, and
the $\psi_{\CO_{\ell_0}}$-Fourier coefficient $\CF^{\psi_{\CO_{\ell_0}}}(\varphi_{\pi})$ is
defined by an $F$-rational nilpotent orbit $\CO_{\ell_0}$ in the $F$-stable nilpotent orbit $\CO^\st_{\udl{p}_{\ell_0}}$,
associated to the partition $\udl{p}_{\ell_0}=[(2\ell_0+1)1^{\Fn-2\ell_0-1}]$.
\end{thm}

\begin{proof}
By assumption, $L(\frac{1}{2},\phi\times\phi')=L(\frac{1}{2},\phi\times(\phi')^{w^\ell_q})\neq 0$. Let $\sig_0$ be the member in the global Vogan packet $\wt{\Pi}_{(\phi')^{w^\ell_q}}[H_m^*]$, with which Conjecture \ref{gnvc} holds, and $\sig_0\in\CA_\cusp(H_m)$ for some pure inner $F$-form of $H_m^*$,
having the cuspidal realization $\CC_{\sig_0}$.
By Proposition \ref{esp}, the Eisenstein series $E(h,\phi_{\tau\otimes\sig_0},s)$
produces the nonzero iterated residual representation $\CE_{\tau\otimes\sig_0}$ on
$H_{a+m}(\BA)$, with a nonzero Fourier coefficient associated to the partition $\udl{p}_{\phi,\phi'}^1$.
In other words, take
$$\udl{p}_{\kappa_0}:=[(2\kappa_0+1)1^{2a+\Fm-2\kappa_0-1}]$$
with $\kappa_0=a-\ell_0-1$, $\ell_0^-=m$, and $\kappa_0^-=n$. Then
the $\psi_{\CO_{\kappa_0}}$-Fourier coefficient
$\CF_n^{\CO_{\kappa_0}}(\CE_{\tau\otimes\sig_0})$ is nonzero and cuspidal as a sub-representation of $G_n^{\CO_{\kappa_0}}(\BA)$ in the
cuspidal spectrum $L^2_\cusp(G_n^{\CO_{\kappa_0}})$,
where $\CO_{\kappa_0}$ is an $F$-rational nilpotent orbit in the
$F$-stable nilpotent orbit $\CO^\st_{\udl{p}_{\kappa_0}}$ associated to the partition
$\udl{p}_{\kappa_0}$.
Note that the group $G_n^{\CO_{\kappa_0}}$ is a pure inner $F$-form of $G_n^*$, and
by Theorem \ref{wlt}, the $\kappa_0$-th Bessel module $\CF_n^{\CO_{\kappa_0}}(\CE_{\tau\otimes\sig_0})$ is $\wt{\RO}(G_n^{\CO_{\kappa_0}})$-stable,
and every irreducible summand of $\CF_n^{\CO_{\kappa_0}}(\CE_{\tau\otimes\sig_0})$ has a global Arthur parameter belonging to the
$\wt{\RO}(G_n^{\CO_{\kappa_0}})$-orbit $\{\phi=\phi_\tau,\phi_\star\}$ of $\phi_\tau$.

Take $(\pi,\CC_\pi)$ to be one of the irreducible summands, such that $\pi$ belongs to the global Arthur packet $\Pi_\phi(G_n^{\CO_{\kappa_0}})$.
Then for some $\varphi_\pi\in\CC_\pi$, the Bessel period
$$
\left<\varphi_\pi,\CF_n^{\CO_{\kappa_0}}(\CE_{\tau\otimes\sig_0})\right>_{G_n^{\CO_{\kappa_0}}}\neq 0.
$$
By replacing the residue $\CE_{\tau\otimes\sig_0}$ by the corresponding Eisenstein series, we obtain that the global zeta integral
$$
\left<\varphi_\pi,\CF_n^{\CO_{\kappa_0}}(E(\cdot,\phi_{\tau\otimes\sig_0},s)\right>_{G_n^{\CO_{\kappa_0}}}\neq 0
$$
for $\Re(s)$ large. By Corollary \ref{zeroa>l}, the pair $(\pi,\sig_0^{w_q^\ell})$ admits a nonzero Bessel period.
It is clear that $\sig_0^{w_q^\ell}$ belongs to the global Vogan packet $\wt{\Pi}_{\phi'}[H_m^*]$. We take $\sig:=\sig_0^{w_q^\ell}$. Then
the pair $(\pi,\sig)$ belongs to the global Vogan packet $\wt{\Pi}_{\phi\times\phi'}[G_n^*\times H_m^*]$ and has the desired property.
We are done.
\end{proof}

We note that Theorem \ref{gggp2} does not assume that the cuspidal multiplicity of $\pi$ should be one, while the global Gan-Gross-Prasad conjecture
takes this cuspidal multiplicity one assumption in \cite{GGP12}.

Also, for $F$-quasisplit classical groups $G$, a special case of Theorem \ref{gggp2} was also considered in \cite{GJR04} and \cite{GJR05}. It is
clear that within the theory of the construction via twisted automorphic descents of concrete modules for irreducible cuspidal automorphic representations,
the proof of Theorem \ref{gggp2} is more transparent than that in \cite{GJR04} or \cite{GJR05}.

By Proposition \ref{subrnv} and \cite{JLXZ}, the assumption in Theorem \ref{gggp2} is verified for the case of $m=1$ and $H_1$ is an $F$-form of
$\SO_2$. 
Hence, Theorem \ref{gggp2} holds {\sl without the assumption of Conjecture \ref{gnvc}} for this special case.
Combining with Theorem \ref{gggp1}, the global Gan-Gross-Prasad Conjecture holds for this case.

\begin{cor}[Global Gan-Gross-Prasad Conjecture: special case]\label{cor-gggpm1}
Let $G_n^*$ be the $F$-split $\SO_{2n+1}$ and $\phi=\phi_\tau$ be a generic global Arthur parameter in $\wt{\Phi}_2(G^*_n)$ determined by
the irreducible isobaric automorphic representation $\tau$ of $G_{E/F}(a)(\BA)$ as given in \eqref{tau8}. Let
$\phi'$ be a generic global Arthur parameter of $H^*_1$, which is an anisotropic $\SO_2$ over $F$.
Then the following statements are equivalent:
\begin{enumerate}
	\item There exist an automorphic member $\pi$ in $\wt{\Pi}_{\phi_\tau}[G_n^*]$ and an automorphic member $\sig$ in $\wt{\Pi}_{\phi'}[H_1^*]$ such that
the inner product
$$
\left<\CF^{\psi_{\CO_{\ell_0}}}(\varphi_{\pi}),\varphi_{\sigma}\right>_{H_1}\neq 0
$$
for some $\varphi_{\pi}\in\pi$ and $\varphi_{\sigma}\in\sigma$;
	\item $L(\frac{1}{2},\tau\times\phi')\ne 0$.
\end{enumerate}
\end{cor}

Note that the cuspidal multiplicities of $\pi$ and $\sig$ in Corollary \ref{cor-gggpm1} are one.
Hence the cuspidal realizations of $\pi$ and $\sig$ are unique.
Also we would like to mention that Corollary \ref{cor-gggpm1} can be proved for unitary groups, but we will not discuss the details here.
We also note that Corollary \ref{cor-gggpm1} with trivial $\sigma$ was considered in \cite{FM}, via a different approach.


\section{On the Main Conjecture}\label{sec-nmc}


\subsection{The main conjecture: general case}\label{sec-mcgc}
We are going to prove the main conjecture (Conjecture \ref{pmc}), assuming Conjectures \ref{bpconj} and \ref{smoconj}.
More precisely, we show, assuming the conjectures, that for any $\pi\in\CA_\cusp(G_n)$ with a $G_n$-relevant, generic
global Arthur parameter $\phi$ in $\wt{\Phi}_2(G_n^*)$, the cuspidal automorphic module
${\mathcal D}^{\CO_{\kappa_0}}_n(\tau;\sigma)=\CF^{\CO_{\kappa_0}}_{n}(\CE_{\tau\otimes\sigma})$ as constructed through Diagram \eqref{diag} is
a direct sum of the two irreducible cuspidal representations in the $\wt{\RO}(G_n)$-orbit of $\pi$ in the cuspidal spectrum of
$G_n$.
If assume further that the $\wt{\RO}(G_n)$-orbit of $\pi$ contains only $\pi$, then we have
$$
{\mathcal D}^{\CO_{\kappa_0}}_n(\tau;\sigma)\cong\pi.
$$
By Proposition \ref{piform}, the $F$-rational orbit $\CO_{\kappa_0}$ can be chosen such that $G_n^{\CO_{\kappa_0}}=G_n$.
We note that the proof of Conjecture \ref{smoconj} has been in well progress, the known cases of which were explained in Theorem \ref{mwslmo}.

\begin{thm}[Cuspidal Automorphic Modules]\label{th-mcgeneral}
Conjectures \ref{bpconj} and \ref{smoconj} imply Conjecture \ref{pmc}.
\end{thm}

\begin{proof}
Take any cuspidal automorphic member $\pi\in\wt{\Pi}_\phi(G_n)$ with a cuspidal realization $\CC_\pi$ satisfying the conditions in Conjecture \ref{bpconj}.
It follows that
$m:=\ell_0^-$, $H_m:=H_{\ell_0^-}^{\CO_{\ell_0}}$, and $\sigma\in\CA_\cusp(H_m)$ with a generic, $H_m$-relevant global Arthur
parameter $\phi'\in\wt{\Phi}_2(H_m^*)$
and with a cuspidal realization $\CC_\sig$.
They have the property that the inner product
$\left<\CF^{\psi_{\CO_{\ell_0}}}(\varphi_\pi),\ol{\varphi}_\sigma\right>_{H_m}$
is nonzero for some $\varphi_\pi\in\CC_\pi$ and $\varphi_\sigma\in\CC_\sigma$. As proved in Section \ref{sec-pif},
for each local place $\nu$ of $F$, the group $G_n(F_{\nu})\times H_m(F_{\nu})$
is relevant in the sense of the local Gan-Gross-Prasad conjecture as discussed in Section \ref{sec-lggp}, and the local parameter
$\phi_{\nu}\otimes\phi'_{\nu}$
belongs to $\wt{\Phi}^+_{\unit,\nu}(G_n\times H_m)$. By Conjecture \ref{smoconj}, the pair $(\pi_{\nu},\sigma_{\nu})$ must be the unique
distinguished member in the local Vogan packet $\wt{\Pi}_{\phi_{\nu}\otimes\phi'_{\nu}}[G_n^*\times H_m^*]$
as defined in \eqref{lvp}, such that the following space, as defined
in \eqref{lfn},
\begin{equation*}
\Hom_{R_{\ell_0,\CO_{\ell_0}}(F_{\nu})}(\pi_{\nu}\otimes\sigma_{\nu},\psi_{\CO_{\ell_0},\nu})
\end{equation*}
is nonzero. Hence the pair $(\pi, \sig)$ is the unique distinguished member in the global Vogan packet $\wt{\Pi}_{\phi\otimes\phi'}[G_n^*\times H_m^*]$.

We apply the reciprocal non-vanishing for Bessel periods (Theorem \ref{th-rnbp}) to the data $(G_n,H_m;\tau,\pi,\sigma)$,
following the choice in Section \ref{sec-sdbps} and obtain that the Bessel period
$$
\left<\varphi_\pi,\ol{\CF^{\psi_{\CO_{\kappa_0}}}(\CE_{\tau\otimes\sigma'})}\right>_{G_n}\neq 0,
$$
for some choice of data. In particular, this implies that the $\psi_{\CO_{\kappa_0}}$-Fourier coefficient
$\CF^{\psi_{\CO_{\kappa_0}}}(\CE_{\tau\otimes\sigma'})$ is nonzero.

On the other hand, by Theorem \ref{wlt}, $\CF^{\CO_{\kappa_0}}_n(\CE_{\tau\otimes\sigma'})$ with $n=\kappa_0^-$
is nonzero and cuspidal as a sub-representation of $G_n(\BA)$ in the cuspidal spectrum $L^2_\cusp(G_n)$, with $G_n=G_n^{\CO_{\kappa_0}}$,
and hence can be written as a multiplicity-free, Hilbert direct sum:
$$
\CF^{\CO_{\kappa_0}}_n(\CE_{\tau\otimes\sigma'})
=
\pi_1\oplus\pi_2\oplus\cdots\oplus\pi_k\oplus\cdots,
$$
where $\pi_i\in\CA_\cusp(G_n)$ for all $i=1,2,\cdots$. Each irreducible summand $\pi_i$ has a generic global Arthur parameter belonging to the
$\wt{\RO}(G_n)$-orbit $\{\phi=\phi_\tau,\phi_\star\}$ of $\phi_\tau$.
We apply Theorem \ref{th-rnbp} to $\pi_i$ for all $i$. The non-vanishing of the Bessel period $\left<\varphi_{\pi_i},\ol{\CF^{\psi_{\CO_{\kappa_0}}}(\CE_{\tau\otimes\sigma'})}\right>_{G_n}$ implies
the inner product on the right hand side
$$
\left<\CF^{\psi_{\CO_{\ell_0}}}(\varphi_{\pi_i}),\ol{\varphi}_\sigma\right>_{H_m}
$$
is nonzero for some choice of data. Following Section \ref{sec-pif},
the product $G_n^{\CO_{\kappa_0}}\times H_m$ constructed as in Diagram \eqref{diag} is a pure inner $F$-form of
an $F$-quasisplit $G_n^*\times H_m^*$. Then by Theorem \ref{wlt} again, the pair $(\pi_i,\sig)$ belongs to either the
global Vogan packet $\wt{\Pi}_{\phi\otimes\phi'}[G_n^*\times H_m^*]$ or the global Vogan packet
$\wt{\Pi}_{\phi_\star\otimes\phi'}[G_n^*\times H_m^*]$. Since
the pair $(\pi, \sig)$ is the unique distinguished member in the global Vogan packet $\wt{\Pi}_{\phi\otimes\phi'}[G_n^*\times H_m^*]$, and
the pair $(\pi_\star, \sig)$ is the unique distinguished member in the global Vogan packet $\wt{\Pi}_{\phi_\star\otimes\phi'}[G_n^*\times H_m^*]$,
we must have that for each index $i$, $\pi_i$ is isomorphic to either $\pi$ or $\pi_\star$, under the assumption of Conjecture \ref{smoconj}.

Because the direct sum decomposition of $\CF^{\CO_{\kappa_0}}_n(\CE_{\tau\otimes\sigma})$ is multiplicity free, it follows that
$\CF^{\CO_{\kappa_0}}_n(\CE_{\tau\otimes\sigma})$ must be of the form:
$$
\CF^{\CO_{\kappa_0}}_n(\CE_{\tau\otimes\sigma})\cong\pi\oplus\pi_\star,
$$
if the $\wt{\RO}(G_n)$-orbit of $\pi$ has two members $\pi$ and $\pi_\star$. If the $\wt{\RO}(G_n)$-orbit of $\pi$ contains only $\pi$, then
we must have
$$
\CF^{\CO_{\kappa_0}}_n(\CE_{\tau\otimes\sigma})\cong\pi.
$$
We are done.
\end{proof}

\subsection{The main conjecture: regular orbit case}\label{sec-mcro}
In this section, we assume that the group $G_n=G_n^*$ is $F$-quasisplit, and
$\pi\in\CA_\cusp(G_n^*)$ is generic, i.e. has a nonzero Whittaker-Fourier coefficient.
In this case, the global Arthur parameter of $\pi$ can be taken in form
\eqref{gapgn}. Then the Langlands functorial transfer of $\pi$ from $G_n^*$ to $G_{E/F}(N)$ is $\tau$, which is of the form \eqref{tau8}. This
is essentially proved by the work of Cogdell, Kim, Piatetski-Shapiro and Shahidi in \cite{CKPSS04}, with combination of the automorphic
descent of Ginzburg-Rallis-Soudry (\cite{GRS11}). We refer to \cite[Section 3.1]{JL-Cogdell} for the detailed discussions of this and some related issues.

In this case, Conjecture \ref{bpconj} holds automatically without $(H_m,\sigma)$. The residual
representation is $\CE_\tau$ on the $F$-quasisplit $H_a^*(\BA)$.
The automorphic descent of Ginzburg-Rallis-Soudry in \cite{GRS11} shows that
$$
{\mathcal D}^{\CO_{\kappa_0}}_n(\tau;\emptyset)=\CF^{\CO_{\kappa_0}}_{n}(\CE_{\tau})
$$
is a nonzero cuspidal automorphic representation of $G_n^*(\BA)$. As proved in \cite{JS03},
the descent ${\mathcal D}^{\CO_{\kappa_0}}_n(\tau;\emptyset)$ is
in fact irreducible for $G_n^*$,
which is an $F$-split odd special orthogonal group. In general, the structure of ${\mathcal D}^{\CO_{\kappa_0}}_n(\tau;\emptyset)$ follows from
Conjecture \ref{smoconj}. Hence Conjecture \ref{pmc} is proved under Conjecture \ref{smoconj} as a consequence of the proof of Theorem
\ref{th-mcgeneral}.

\begin{cor}[Regular Orbit]\label{generic}
Let $G_n^*$ be $F$-quasisplit. For any $\pi\in\CA_\cusp(G_n^*)$ to be generic with its global Arthur parameter \eqref{gapgn} and $\tau$ as in \eqref{tau8},
then Conjecture \ref{pmc} holds for $\pi$ under the assumption of Conjecture \ref{smoconj}.
\end{cor}

\subsection{The main conjecture: subregular orbit case}\label{sec-mcsro}
We consider in this subsection the irreducible cuspidal automorphic representations $\pi$ of $G_n(\BA)$ such that
the set $\Fp^m(\pi)$ contains the partition $\udl{p}_\subr$ corresponding to the subregular nilpotent orbit of $G_n^*$.
In this situation, it is clear that $\Fp^m(\pi)=\{\udl{p}_\subr\}$. Conjecture \ref{bpconj} can be verified as follows.
The group $H_m$ constructed via Diagram~\eqref{diag} can be determined as below:

If $G_n^*$ is an $F$-quasisplit $\SO_{2n}$, the subregular partition $\udl{p}_\subr$ is $[(2n-3)3]$. The partition with the first occurrence index
$\ell_0$ is $\udl{p}_{\ell_0}=[(2n-3)1^3]$ with $\ell_0=n-2$. Hence $H_m$ is a pure inner $F$-form of $\SO_3$, where $m=\ell_0^-$.
According to \cite[Theorem 11.2]{JLS}, because
$\Fp^m(\pi)=\{\udl{p}_\subr=[(2n-3)3]\}$, the $\ell_0$-th Bessel module $\CF^{\CO_{\ell_0}}(\pi)$ associated to the $F$-rational orbit $\CO_{\ell_0}$
must be nonzero if $H_m=H^{\CO_{\ell_0}}_{\ell_0^-}$ is the split $\SO_3$.
Hence Conjecture \ref{bpconj} holds for this case.

If $G_n^*$ is an $F$-split $\SO_{2n+1}$, then $\udl{p}_\subr$ is $[(2n-1)1^2]$, which is the partition with the first occurrence index $\ell_0=n-1$.
In this case, the group $H_m$ is an $F$-form of $\SO_2$, and hence Conjecture \ref{bpconj} holds.

If $G_n^*$ is an $F$-quasisplit $\RU_{2n}$, then $\udl{p}_\subr$ is $[(2n-1)1]$, which is the partition with the first occurrence index $\ell_0=n-1$.
In this case, the group $H_m$ is equal to $\RU_1$, and hence Conjecture \ref{bpconj} holds.

If $G_n^*$ is an $F$-quasisplit $\RU_{2n+1}$, then the subregular partition $\udl{p}_\subr$ is $[(2n)1]$. The partition with the first occurrence index
$\ell_0$ is $\udl{p}_{\ell_0}=[(2n-1)1^2]$ with $\ell_0=n-1$. Hence $H_m$ is an $F$-form of $\RU_2$. It is clear that Conjecture \ref{bpconj} also holds
for this case, following the proof for the case of $F$-quasisplit $\SO_{2n}$.

We summarized this discussion as

\begin{prop}\label{prop-subr}
Let $\phi=\phi_\tau$ be the generic global Arthur parameter of $G_n^*$ as given in \eqref{gapgn} with $\tau$ as defined in \eqref{tau8}.
If a cuspidal automorphic member $\pi$ in the global Vogan packet $\wt{\Pi}_\phi[G_n^*]$ has the property that $\Fp^m(\pi)=\{\udl{p}_\subr\}$,
then Conjecture \ref{bpconj} holds for $\pi$.
\end{prop}

As a consequence of the proof of Theorem \ref{th-mcgeneral}, we have the following result.

\begin{cor}[Subregular Orbit]\label{sboc}
Assume that $\pi\in\CA_\cusp(G_n)$ has a $G_n$-relevant, generic global Arthur parameter in $\wt{\Phi}_2(G_n^*)$ and the set
$\Fp^m(\pi)$ contains the subregular partition $\udl{p}_\subr$ of type $(\Fn,G_n^*)$. Conjecture \ref{pmc} holds for $\pi$
under the assumption of Conjecture \ref{smoconj}.
\end{cor}

\appendix


\section{Non-vanishing of Local Zeta Integrals}\label{A}


In this appendix, we prove Proposition \ref{nlzip:nonzero}. It is a purely local non-vanishing property of
the finite product of the local zeta integrals, $\CZ_{S}(s,\phi_{\tau\otimes\sig'},\varphi_{\pi},\psi_{\CO_{\kappa_0}})$.
However, the local data have constraints from the global
assumption for $(\pi,\tau,\sig)$ from Theorem \ref{th-rnbp}. From Proposition \ref{lzi-pp2}, $\CZ_{S}(s,\phi_{\tau\otimes\sig'},\varphi_{\pi},\psi_{\CO_{\kappa_0}})$ converges absolutely for $\Re(s)$ large, has a meromorphic continuation to $s\in\BC$,
and is holomorphic at $s=\frac{1}{2}$, What we need to prove Theorem \ref{th-rnbp} is the non-vanishing at $s=\frac{1}{2}$ for a choice of data with
certain global constraints as described in Proposition \ref{nlzip:nonzero}. In fact, we are going to show a more general non-vanishing property for
the local zeta integral $\CZ_{\nu}(s,\phi_{\tau\otimes\sig'},\varphi_{\pi},\psi_{\CO_{\kappa_0}})$ for every $\nu\in S$. These local zeta integrals
converges absolutely for $\Re(s)$ large and has a meromorphic continuation to $s\in\BC$. We give the proof in \cite{JSdZ}, and refer to
\cite{S-I} and \cite{S-II} for the case of the split special orthogonal groups.

Throughout this appendix, all algebraic groups $X$ are defined over $F_\nu$. The $F_\nu$-rational points of $X$ is simply denoted by
$X=X(F_\nu)$ when no confusion is caused.

For $\Re(s)$ large, the local zeta integral in Theorem~\ref{thm:j>l} is defined as in \eqref{localzetaa>l} by
\begin{equation}\label{eq:appendix-CZ}
\CZ_\nu(s,\phi_{\tau\otimes\sig'},\varphi_{\pi},\psi_{\ell,w_{0}})
=
\int_{R^{\eta}_{\ell,\beta-1}\bks G_{m^-}^{w_0}}\CP_\nu^{\psi^{-1}_{\beta-1,y_{-\kappa}}}(g_\nu\ast\varphi_{\pi_\nu},\RJ_{s,\nu}(\phi_{s,\nu})(g_\nu))
\ud g_\nu,
\end{equation}
where $\CP_\nu^{\psi^{-1}_{\beta-1,y_{-\kappa}}}$ is the unique local Bessel functional, up to scalar, in the space
\begin{equation} \label{lbf} 
\Hom_{R^{\eta}_{\ell,\beta-1}(F_\nu)}(\pi_\nu\otimes\sig_\nu,\psi^{-1}_{\beta-1,y_{-\kappa}}).
\end{equation}
This $\Hom$-space is at most one-dimensional,
by the uniqueness of local Bessel models for classical groups (\cite{AGRS}, \cite{SZ}, \cite{GGP12} and \cite{JSZ}).
Alternatively, for a  local Bessel functional $\Fb_\nu$ in \eqref{lbf}, 
we may rewrite $\CZ_\nu(s,\phi_{\tau\otimes\sig'},\varphi_{\pi},\psi_{\ell,w_{0}})$ as
\[
\int_{R^{\eta}_{\ell,\beta-1}\bks G_{m^-}^{w_0}}\int_{U^{-}_{a,\eta}(F_\nu)}\Fb_\nu(g_\nu\ast\varphi_{\pi_\nu},f_{\CW_{\tau_\nu} \otimes\sig_\nu',s}(u\epsilon_\beta \eta g_\nu))
\psi_{\Fm+a+\ell,n-\ell}(u) \ud u
\ud g_\nu.	
\] 
 

Our goal is to construct a  section $f_{\CW_{\tau_\nu}\otimes\sig_\nu',s}$ belonging to
$\RI_{s,\nu}(\CW_{\tau_\nu},\sig'_\nu)$ such that the following non-vanishing holds.

\begin{prop}\label{A1}
Suppose that a non-zero local Bessel functional $\Fb_\nu$ in the $\Hom_{R^{\eta}_{\ell,\beta-1}(F_\nu)}$-space \eqref{lbf}
is not zero at some $\varphi_{\pi_\nu}=v_{\pi_\nu}\in\pi_\nu$ and
$v_{\sig_\nu}\in\sig_\nu$, i.e. $\Fb_\nu(v_{\pi_\nu},v_{\sig_\nu})\neq 0$. Then,
for any given $s=s_0\in\BC$,
there exists a holomorphic section $f_{\CW_{\tau_\nu}\otimes\sig_\nu',s}$ belonging to
$\RI_{s,\nu}(\CW_{\tau_\nu},\sig_\nu')$  such that
the local zeta integral $\CZ_{\nu}(s,\phi_{\tau\otimes\sig'},\varphi_{\pi},\psi_{\ell,w_{0}})$
is nonzero at $s=s_0$.
\end{prop}

It is clear that Proposition \ref{A1} for split orthogonal groups over $p$-adic fields is just Proposition 4.1 of \cite{S-I}.
In the proof of Proposition \ref{A1}, one of the technical issues is to construct the section $f_{\CW_{\tau_\nu}\otimes\sig_\nu',s}$
in the space of the induced representation
$\RI_{s,\nu}(\CW_{\tau_\nu},\sig_\nu')$ with the given constraints. Soudry in his proof of \cite[Proposition 4.1]{S-I} uses
the Iwasawa decomposition to explicitly construct such sections $f_{\CW_{\tau_\nu}\otimes\sig_\nu',s}$. We are going to use the
Bruhat decomposition to proceed
the explicit construction, which works for more general groups over local fields of characteristic 0.

We recall from
Section \ref{sec-gzi} that $H_{a+m}$ is either a special orthogonal group or unitary group.
When $H_{a+m}$ is unitary and $\nu$ splits in the number field $E$,
$H_{a+m}(F_\nu)=\RU_{2a+\Fm}(F_\nu)$ is isomorphic to $\GL_{2a+\Fm}(F_\nu)$. We defer the discussion on this case to the end of this proof.
We first consider the case that  $H_{a+m}(F_\nu)$ is not isomorphic to $\GL_{2a+\Fm}(F_\nu)$.
For convenience,  we consider $\RJ_{s,\nu}$ as a map  
\begin{equation}\label{RJnu}
\RI_{s,\nu}(\CW_{\tau_\nu},\sigma'_\nu)\rightarrow \RI_{s,\nu}^{w_0}(\psi_{\beta-1,y_{-\kappa}},\sig_\nu),
\end{equation}
which is given by the following 
$U^-_{a,\eta}(F_\nu)$-integration 
\begin{equation}\label{Uaeta}
\RJ_{s,\nu}(f_{\CW_{\tau_\nu}\otimes\sig'_\nu,s})(g):=
\int_{U_{a,\eta}^-(F_\nu)}f_{\CW_{\tau_\nu}\otimes\sig'_\nu,s}(n\epsilon_\beta \eta g)\psi_{\Fm+a+\ell,n-\ell}(n) \ud n
\end{equation}
as in \eqref{eq:CJ}.

 It is not hard to show that the integration in \eqref{Uaeta} converges absolutely for $\Re(s)$ large.
It is a little bit more technical to show that it admits a meromorphic continuation to $s\in\BC$ in general, which will be treated in \cite{JSdZ}.
However, for the purpose of this Appendix, we are able to prove this easily for the particular sections $f_{\CW_{\tau_\nu}\otimes\sig'_\nu,s}$
that will be constructed below for the proof of Proposition \ref{A1}.

\begin{rmk}\label{localissue}
For further refined applications of the global zeta integrals considered in this paper, one may be interested in the characterization
of the image of $\RJ_{s,\nu}$ in \eqref{RJnu}. However, for the purpose of this paper, we do not need this. Hence we
will leave this interesting question to be considered in our future work.
\end{rmk}

Let $\Fb_\nu$ be a non-zero local Bessel functional in the $\Hom$-space  \eqref{lbf}. Take
some $v_{\pi_\nu}\in\pi_\nu$ and
$v_{\sig_\nu}\in\sig_\nu$, such that $\Fb_\nu(v_{\pi_\nu},v_{\sig_\nu})\neq 0$. We are going to construct
a section $f_{\CW_{\tau_\nu}\otimes\sig'_\nu,s}$ in
$$
\RI_{s,\nu}(\CW_{\tau_\nu},\sigma'_\nu)=\Ind^{H_{a+m}}_{P_{\hat{a}}}(|\cdot|^s\CW_{\tau_\nu}\otimes\sig'_\nu),
$$
which is compactly supported in the open cell $P_{\hat{a}}{U}_{\hat{a}}^-$ of $H_{a+m}$, modulo $P_{\hat{a}}$ from the left. 
Recall that ${U}_{\hat{a}}^-$ is the unipotent subgroup opposite to the unipotent radical $U_{\hat{a}}$ of $P_{\hat{a}}$,
as defined in Section \ref{sec-gzi}. We define
\begin{equation}\label{eq:appendix-phi}
f_{\CW_{\tau_\nu}\otimes\sig'_\nu,s}
\left( \begin{pmatrix}
g& & \\ & h & \\ & &g^* 	
\end{pmatrix}u\bar{n}' \epsilon_\beta \eta\right):=|\det g|^{s+\rho_a}W_{\tau_\nu}^\kappa(g)f_\nu(\bar{n}')\sig(h)v_{\sig_\nu}
\end{equation}
with $g\in \GL_{a}(E_\nu)$, $h\in H_m(F_\nu)$, $u\in U_{\hat{a}}(F_\nu)$, and $\bar{n}'\in {U}_{\hat{a}}^-(F_\nu)$.
Here  $W_{\tau_\nu}^\kappa(g)$ is a Whittaker function in $\CW_{\tau_\nu}$, $f_\nu(\bar{n}')$ is a smooth, compactly supported function defined on ${U}_{\hat{a}}^-(F_\nu)$,
and $|\cdot|^{2\rho_a}$ is the modular character of the parabolic subgroup $P_{\hat{a}}$.
Over archimedean places, we may take $f_\nu(\bar{n}')$ also to be a positive real-valued function.
Since $H_{a+m}(F_\nu)\ne \GL_{2a+\Fm}(F_\nu)$, $G_{E_\nu/F_\nu}(a)(F_\nu)\ne \GL_a(F_\nu)\times\GL_a(F_\nu)$.
Hence the subgroup $(\Res_{E/F}\GL_a)(F_\nu)$ of the Levi part of $P_{\hat{a}}$ can be written as $\GL_a(E_\nu)$ where $E_\nu$ is either $F_\nu$ or a quadratic field extension over $F_\nu$.
Remark that because of the conjugation by $w^\ell_q$, it is $v_\sig$ on the right hand side of \eqref{eq:appendix-phi}, instead of $v_{\sig'}$.
It is clear that the section $f_{\CW_{\tau_\nu}\otimes\sig'_\nu,s}$
defined in \eqref{eq:appendix-phi} is a smooth section in $\RI_{s,\nu}(\CW_{\tau_\nu},\sigma'_\nu)$.
It is clear that for such a constructed section $f_{\CW_{\tau_\nu}\otimes\sig'_\nu,s}$, the integration in \eqref{Uaeta} is over a compact set.
Hence the integral converges absolutely for every $s\in\BC$, and  admits meromorphic continuation to all $s\in\BC$.

We may assume that the value of $f_{\CW_{\tau_\nu}\otimes\sig'_\nu,s}(g)$ at $g=\epsilon_\beta \eta$:
\begin{equation}
f_{\CW_{\tau_\nu}\otimes\sig'_\nu,s}
(\epsilon_\beta \eta)=W_{\tau_\nu}^\kappa(\RI_a)v_{\sig_\nu},
\end{equation}
where $\RI_a$ is the identity matrix of $\GL_a$. 
It is clear that the functional $\Fb_\nu$ evaluated at $(v_\pi,f_{\CW_{\tau_\nu}\otimes\sig_\nu',s}(\eps_\beta\eta))$ is given by
\begin{equation}\label{PA1}
\Fb_\nu(v_\pi, f_{\CW_{\tau_\nu}\otimes\sig_\nu',s}(\eps_\beta\eta))
=
W_{\tau_\nu}^\kappa(\RI_a)\cdot\Fb_\nu(v_{\pi_\nu},v_{\sig_\nu}).
\end{equation}
We may take the value of $W_{\tau_\nu}^\kappa(\RI_a)$, so that
\begin{equation}\label{eq:nonvanish-pf-i}
W_{\tau_\nu}^\kappa(\RI_a)\cdot\Fb_\nu(v_{\pi_\nu},v_{\sig_\nu})=1.
\end{equation}
This gives the normalization of the local Bessel functional $\Fb_\nu$ at all $\nu\in S$.


To finish the proof of Proposition \ref{A1}, we have to calculate explicitly the relation between $R^\eta_{\ell,\beta-1}\bks G^{w_0}_{m^-}$ and
the open dense set $P_{\hat{a}}{U}_{\hat{a}}^-$, in particular, the following domain
\begin{equation}\label{eq:appendix-intersection-1}
R^\eta_{\ell,\beta-1}\bks G^{w_0}_{m^-}\cap (\eps_\beta\eta)^{-1}(P_{\hat{a}}{U}_{\hat{a}}^-)(\eps_\beta\eta).
\end{equation}
The group $G^{w_0}_{m^-}$ is identified as a subgroup of the Levi subgroup of $P_{\hat{\ell}}$.
According to the structure of the stabilizer of the open cell $P_{\hat{a}}\eps_\beta P_{\hat{\ell}}$ as given in Section \ref{sec-al},
the intersection \eqref{eq:appendix-intersection-1} can be written as the following intersection
\begin{equation}\label{eq:appendix-intersection-2}
R^\eta_{\ell,\beta-1}\bks G^{w_0}_{m^-}\cap \Ad(\eta^{-1})(P'_{w}{U}_{a-\ell}^-).
\end{equation}
Recall that  $P'_{w}=H_{a+m-\ell}\cap \epsilon^{-1}_{0,\beta} P_{\hat{a}} \epsilon_{0,\beta}$ defined in Proposition \ref{bfcc} is the standard parabolic subgroup of $H_{a+m-\ell}$ with Levi decomposition $(G_{E_\nu/F_\nu}(a-\ell)\times H_{m}) \ltimes U_{a-\ell}$,
where $U_{a-\ell}$ is the unipotent radical.
More details can be found in \cite[Section 3.1]{JZ14}.
Because of
$$
R^\eta_{\ell,\beta-1}=G^{w_0}_{m^-}\cap \eta^{-1}P'_{w}\eta,
$$
the intersection set $G^{w_0}_{m^-}\cap \Ad(\eta^{-1})(P'_{w}{U}_{a-\ell}^-)$ is $R^\eta_{\ell,\beta-1}$-left stable.
Thus the intersection modulo $R^\eta_{\ell,\beta-1}$ in \eqref{eq:appendix-intersection-2} is well defined.
It is clear that $P'_{w}{U}_{a-\ell}^-$ is an open subset of $H_{a+m-\ell}$.

With the above choice of $f_{\CW_{\tau_\nu}\otimes\sig'_\nu,s}$, the integral \eqref{eq:appendix-CZ} can be taken over the set
\eqref{eq:appendix-intersection-2}.
To proceed with the integral \eqref{eq:appendix-CZ}, we explicitly describe the intersection
$G^{w_0}_{m^-}\cap \Ad(\eta^{-1})(P'_{w}{U}_{a-\ell}^-)$. It is enough to describe
the set $\Ad(\eta)G^{w_0}_{m^-}\cap (P'_{w}{U}_{a-\ell}^-)$.
Take
\begin{equation}\label{eq:appendix-p}
p=\begin{pmatrix}
g&-Y\cdot h^{-1}&\iota(\hat{Z})g^*\\ &h^{-1}&-Y'g^*\\  &&g^*	
\end{pmatrix}\in P'_{w}
\end{equation}
and
\begin{equation} \label{eq:appendix-n}
\bar{n}=\begin{pmatrix}
I_{a-\ell}&& \\ X'&I_{\Fm}&\\ A&X&I_{a-\ell}	
\end{pmatrix}\in {U}_{a-\ell}^-,
\end{equation}
with $g\in \GL_{a-\ell}(E_\nu)$, $h\in H_{m}$, $Y\in \Mat_{(a-\ell)\times \Fm}$, $\hat{Z}:=\omega_{a-\ell}Z^t \omega_{a-\ell}$, $Y':=-\omega_{a-\ell}\cdot \iota(Y)^t\cdot (J^{\Fm}_{\tilde{\Fm}})^{-1}$, and $g^*=\iota(\hat{g})^{-1}$.
Here $\omega_{a-\ell}$ is the anti-diagonal matrix with the unit entry of the size $(a-\ell)$-by-$(a-\ell)$, $J^{\Fm}_{\tilde{\Fm}}$ is defined in \eqref{eq:J},
and $\iota$ is the Galois element in $\Gamma_{E_\nu/F_\nu}$ (as in Page \pageref{pg:iota}).

Since $G^{w_0}_{m^-}$ fixes the anisotropic vector $w_0=y_{\kappa}$ defined in \eqref{eq:w0a},
 $\Ad(\eta)G^{w_0}_{m^-}$ stabilizes the vector
$$
\Ad(\eta)w_0= \begin{pmatrix}
E_1 \\ 0_{\Fm} \\	E_2
\end{pmatrix}_{(2(a-\ell)+\Fm)\times 1}
$$
where $0_\Fm$ is the $\Fm$-dimensional zero column vector,
$$
E_1=(0,\dots,0,1)^t \text{ and }
E_2=((-1)^{\Fm+1}\frac{\kappa}{2},0,\dots,0)^t \text{ in } \Mat_{(a-\ell)\times 1}.
$$
Here we consider $\Ad(\eta)w_0$ as an anisotropic vector in the Hermitian space defining $H_{a+m-\ell}$.
Then $p\cdot \bar{n}\in P'_{w}{U}_{a-\ell}^-$ is in $\Ad(\eta)G^{w_0}_{m^-}$ if and only if $p\cdot \bar{n}$ fixes the vector $\Ad(\eta)w_0$.
That is to say that both $p$ and $\bar{n}$ satisfy the following equations:
\begin{equation}\label{eq:nonvanish-pf}
ZE_2=(I_{a-\ell}-g)E_1,~
AE_1=(\iota(\hat{g})-I_{a-\ell})E_2,~
X'E_1=h\cdot Y'E_2.
\end{equation}
Since our integral domain is a set of $R^\eta_{\ell,\beta-1}$-right cosets, we identify the quotient set \eqref{eq:appendix-intersection-2} by choosing $h=I_{\Fm}$, $g\in Z_{a-\ell}(E_\nu)\bks\GL_{a-\ell}(E_\nu)$,
$$
Y=\begin{pmatrix}
0_{(\Fm-1)\times (a-\ell)} \\	y
\end{pmatrix} \text{ and }
\iota(\hat{Z})g^*= \begin{pmatrix}
z_1 & 0_{(a-\ell-1)\times (a-\ell-1)}\\ z_2& z_3 	
\end{pmatrix}.
$$
Due to \eqref{eq:nonvanish-pf} and the above choice, the vector $y$ in $Y$, and $z_i$ in $Z$ for $1\leq i\leq 3$ are determined by $X$ and $g$, respectively. Because of this, we write $Y_X$ and $Z_g$ for $Y$ and $Z$, respectively.

To separate  variables, we choose
\begin{equation}\label{eq:nonvanish-pf-ii}
f_\nu\begin{pmatrix}
I_{a-\ell} &&&&\\
0&I_\ell&&&\\
X'&x'_2&I_{\Fm}&&\\
x_1&x_3&x_2&I_\ell&\\
A&x'_1&X&0&I_{a-\ell}	
\end{pmatrix}=f_1(x_1,x_2,x_3)f_2(X,A)
\end{equation}
where $f_1$ and $f_2$ are smooth, compactly supported functions, and  the size of matrices $X$ and $x_i$ are indicated by the matrix in \eqref{eq:nonvanish-pf-ii}.
With the above choices, we are able to
evaluate more explicitly the function $\RJ_{s,\nu}(f_{\CW_{\tau_\nu}\otimes\sig'_\nu,s})(\mathrm{g})$ as defined in \eqref{Uaeta},
for $\mathrm{g}$ in the set \eqref{eq:appendix-intersection-2}.  We decompose $\Ad(\eta)\mathrm{g}=p\cdot\bar{n}$ as given in
\eqref{eq:appendix-p} and \eqref{eq:appendix-n}.
Let us conjugate $p\cdot \bar{n}$ by $\eps_\beta$.
Referring to Equation (3.6) in \cite{JZ14},
as elements in $H_{a+m}$
\begin{eqnarray} \label{eq:p-n}
\eps_\beta p\eps^{-1}_\beta&=&u(Y_X,Z_g)m(g)=\ppair{\begin{smallmatrix}
g&0&-Y_X&0&\iota(\hat{Z}_g)g^*\\
&I_{\ell}&0&0&0\\
&&I_{\Fm}&0&-Y'_Xg^*\\
&&&I_{\ell}&0\\
&&&&g^*	
\end{smallmatrix}}, \\
\eps_\beta \bar{n}\eps^{-1}_\beta&=&\ppair{\begin{smallmatrix}
I_{a-\ell} &&&&\\
0&I_\ell&&&\\
X'&0&I_{\Fm}&&\\
0&0&0&I_\ell&\\
A&0&X&0&I_{a-\ell}	
\end{smallmatrix}},\nonumber
\end{eqnarray}	
where
$$
u(Y_X,Z_g)=\ppair{\begin{smallmatrix}
I_{a-\ell}&0&-Y_X&0&\iota(\hat{Z}_g)\\
&I_{\ell}&0&0&0\\
&&I_{\Fm}&0&-Y'_X\\
&&&I_{\ell}&0\\
&&&&I_{a-\ell}	
\end{smallmatrix}}\text{ and }
m(g)=\ppair{\begin{smallmatrix}
g&&\\&I_{\Fm+2\ell}&\\&& g^*
\end{smallmatrix}}.
$$

By the definition of $\RJ_{s,\nu}(f_{\CW_{\tau_\nu}\otimes\sig'_\nu,s})(\mathrm{g})$
in \eqref{Uaeta}, using the above decomposition of $\mathrm{g}$ in the set \eqref{eq:appendix-intersection-2}, we have
\begin{align}
 &\RJ_{s,\nu}(f_{\CW_{\tau_\nu}\otimes\sig'_\nu,s})(\mathrm{g})\label{eq:Appendix-J-1}\\
=& \int_{U_{a,\eta}^-(F_\nu)}f_{\CW_{\tau_\nu}\otimes\sig'_\nu,s}(n u(Y_X,Z_g) m(g)\cdot \eps_\beta \bar{n}\eps^{-1}_\beta\cdot \epsilon_\beta \eta )
\psi_{\Fm+a+\ell,n-\ell}(n) \ud n.\nonumber
\end{align}
Recall that  the element in $U^-_{a,\eta}$ (see \eqref{eq:U-j-eta}) is of form
$$
n(x_1,x_2,x_3):= \ppair{ \begin{smallmatrix}
I_{a-\ell} &&&&\\
0&I_\ell&&&\\
0&x'_2&I_{\Fm}&&\\
x_1&x_3&x_2&I_\ell&\\
0&x'_1&0&0&I_{a-\ell}	
\end{smallmatrix} }.
$$
By simple manipulations, one has
\begin{align*}
  & n(x_1,x_2,x_3)u(Y_X,Z_g)m(g)\\
 =&m(g)\cdot u_a(-\iota(\hat{B}))\cdot n(x_1g,x_2-x_1Y_X,x_3-Bx'_1)
 \end{align*}
 where $B=x_1\iota(\hat{Z}_g)-x_2Y'_X$ and
 $u_a(-\iota(\hat{B}))=\ppair{\begin{smallmatrix}
  I_{a-\ell}&-\iota(\hat{B})\\0&I_{\ell}	
  \end{smallmatrix}}$ is considered as an element in $\GL_a$ as the subgroup of the Levi subgroup of $P_{\hat{a}}$.
Continuing with \eqref{eq:Appendix-J-1},
by the definition of $f_{\CW_{\tau_\nu}\otimes\sig'_\nu,s}$ in \eqref{eq:appendix-phi} and $f_\nu$  in \eqref{eq:nonvanish-pf-ii},
after changing variables we have
\begin{align}\label{eq:nonvanish-pf-J}
&\RJ_{s,\nu}(f_{\CW_{\tau_\nu}\otimes\sig'_\nu,s})(g)=
|\det g|^{s+\rho_a}W_{\tau_\nu}^\kappa(\begin{pmatrix}
g& \\ &I_{\ell}	
\end{pmatrix})f_2(X,A) v_{\sig}\\
&\times\int_{U^-_{a,\eta}}f_1(x_1,x_2,x_3)\psi_{F}((x_1g^{-1})_{\ell,a-\ell})
\psi^{-1}_{Z_a,\kappa}(u_a(-\iota(\hat{B}_1)))
|\det g|^{-\ell}\ud x_i, \nonumber
\end{align}
where the matrices $x_i$ define the element $n(x_1,x_2,x_3)$ in $U^-_{a,\eta}$
and $B_1=x_1g^{-1}(\iota(\hat{Z}_g)-Y_XY'_X)-x_2Y'_X$.
Although the term $B_1$ is complicated, after we choose suitable $X$ and $A$ defining $\bar{n}(X,A)$, the matrices $Y_X$ and $Z_g$ are zero, so is $B_1$.
Since the function $f_1(x_1,x_2,x_3)$  is chosen to be a smooth and compactly supported function and is independent of complex variable $s$,
the integral $\RJ_{s,\nu}$ is well defined over the whole complex plane for such choice of the section $f_{\CW_{\tau_\nu}\otimes\sig'_\nu,s}$, and so is the local zeta integral \eqref{eq:appendix-CZ}.

Finally, by plugging the formula \eqref{eq:nonvanish-pf-J} into \eqref{eq:appendix-CZ}, we obtain that $\CZ_\nu(,\cdot)$ equals
\begin{align}
&\int_{X,A}\int_{g}|\det g|^{s+\rho_a}W_{\tau_\nu}^\kappa(\begin{pmatrix}
g& \\ &I_{\ell}	
\end{pmatrix})f_2(X,A)
\Fb_\nu(\pi(\eta^{-1} p(g,X)\bar{n}(X,A) \eta)v_\pi,v_\sig)
\label{eq:appendix-Z} \\
&\int_{U^-_{a,\eta}}|\det g|^{-\ell}f_1(x_1,x_2,x_3)\psi_{F}((x_1g^{-1})_{\ell,a-\ell})
\psi^{-1}_{Z_a,\kappa}(u_a(-\iota(\hat{B}_1)))\ud x_i
\ud g\ud X\ud A. \nonumber
\end{align}
The notation in the formula is explained in order.
The integration $\int_g$ is over $Z_{a-\ell}(E_\nu)\bks\{g\in \GL_{a-\ell}(E_\nu)\colon \iota(\hat{g})E_2=AE_1+E_2 \}$, with constraints given
in \eqref{eq:nonvanish-pf}.
Rewrite $\bar{n}$ and $p$ to be $\bar{n}(X,A)$ and $p(g,X)$ respectively to
indicate their dependence on variables $X$, $A$ and $g$, following \eqref{eq:p-n}.
The integration $\int_{X,A}$ is over the set ${U}_{a-\ell}^-$ with  $AE_1\ne -E_2$, due to $AE_1=(\iota(\hat{g})-I_a)E_2$ in \eqref{eq:nonvanish-pf} and $\det(g)\ne 0$.
Indeed, because $AE_1+E_2=\iota(\hat{g})E_2$, if $AE_1=-E_2$, then $\iota(\hat{g})E_2=0_{a-\ell}$, which implies  $\det(g)=0$.

We are going to finish the proof based on the above expression for the local zeta integral $\CZ_\nu(s,\cdot)$.
Suppose that $\CZ_\nu(s,\cdot)$ is identically zero for all choice of data $f_1$ and $f_2$ at the given $s=s_0$.
We vary the function $f_2(X,A)$ first and
consider the rest of the integral   as a continuous function of $X$ and $A$.
Since the integral over ${U}_{a-\ell}^-$ is identically zero,  the remaining integration in the variable $\bar{n}$ as given in \eqref{eq:appendix-n}
is identically zero, that is,
\begin{align}\label{eq:nonvanish-pf-XA}
&\int_g|\det g|^{s+\rho_a-\ell}W_{\tau_\nu}^\kappa(\begin{pmatrix}
g& \\ &I_{\ell}	
\end{pmatrix})\Fb_\nu(\pi(\eta^{-1}p\bar{n}\eta)v_\pi,v_\sig)\\
&\times\int_{U^-_{a,\eta}}f_1(x_1,x_2,x_3)\psi_{F}((x_1g^{-1})_{\ell,a-\ell})
\psi^{-1}_{Z_a,\kappa}(u_a(-\iota(\hat{B}_1)))
\ud x_i
\ud g\equiv 0. \nonumber
\end{align}
Especially the integral on the left hand side of \eqref{eq:nonvanish-pf-XA} is identically zero at $\bar{n}=I_{2(a-\ell)+\Fm}$,
equivalently, at $X=0_{(a-\ell)\times \Fm}$ and $A=0_{(a-\ell)\times(a-\ell)}$.
Because $M=I_\Fm$ and $X'E_1=Y'E_2$ in \eqref{eq:nonvanish-pf}, we must have that $Y_X=0_{(a-\ell)\times \Fm}$ due to $X=0_{(a-\ell)\times \Fm}$, and similarly $Z_g=0_{(a-\ell)\times (a-\ell)}$.
It follows that   $B_1=0_{\ell\times (a-\ell)}$ and the character $\psi^{-1}_{Z_a,\kappa}(u_a(-\iota(\hat{B}_1)))$ disappears.
As $A=0_{(a-\ell)\times(a-\ell)}$ and $AE_1=(\iota(\hat{g})-I_a)E_2$ in \eqref{eq:nonvanish-pf}, $g$ must belong to the standard mirabolic subgroup
of $\GL_{a-\ell}(E_\nu)$, that is, $E^t_1g=E^t_1$.
Since the integration domain of $g$ is modulo $Z_{a-\ell}(E_\nu)$ and $E^t_1g=E^t_1$,
the integral $\int_g$ is over $Z_{a-\ell-1}(E_\nu)\bks \GL_{a-\ell-1}(E_\nu)$.
Since $g$ is in the standard mirabolic subgroup of $\GL_{a-\ell}(E_\nu)$, $g^{-1}$ stabilizes the character $\psi_{F}((x_1)_{\ell,a-\ell})$ of  $U^-_{a,\eta}$,
that is, $(x_1 g^{-1})_{\ell,a-\ell}=(x_1)_{\ell,a-\ell}$.

Furthermore, one may choose a suitable smooth, compactly supported function $f_1$ such that
\begin{equation}\label{eq:appendix-f1}
\int_{U^-_{a,\eta}}f_1(x_1,x_2,x_3)\psi_{F}((x_1)_{\ell,a-\ell})\ud x_1\ud x_2\ud x_3=1.
\end{equation}
Plugging \eqref{eq:appendix-f1} into \eqref{eq:nonvanish-pf-XA}, we have
\begin{equation}\label{eq:appendix-W-F}
\int_{g}|\det g|^{s+\rho_a-\ell}W_{\tau_\nu}^\kappa(
\ppair{\begin{smallmatrix}
g& \\ &I_{\ell+1}	
\end{smallmatrix}})\Fb_\nu(\pi( \ppair{\begin{smallmatrix}
g&&\\ &I_{\Fm+1}&\\ &&g^*	
\end{smallmatrix}})v_\pi,v_\sig)\ud g\equiv 0,
\end{equation}
where $\int_g$ is over $Z_{a-\ell-1}(E_\nu)\bks \GL_{a-\ell-1}(E_\nu)$.

It is clear that the left hand side of \eqref{eq:appendix-W-F} is exactly
the same with (4.7) in \cite{S-I}, up to a nonzero constant.
We note that the reduction to this type of the integrals is a key step in the proof of such non-vanishing of the local Rankin-Selberg integrals.
See \cite{S93} for instance.
Applying the same inductive argument in Sections 6 and 7 of \cite{S93} and the Dixmier-Marlliavin Lemma (\cite{D-M}), we obtain that
$$
W_{\tau_\nu}^\kappa(I_a)\Fb_\nu(v_\pi,v_\sig)=0.
$$
However, this contradicts with \eqref{eq:nonvanish-pf-i}. Therefore, there must exist a choice of data such that $\CZ_\nu(s,\cdot)$
is not zero at the given $s=s_0$.
This completes the proof of Proposition \ref{A1} when $H_{a+m}$ is not isomorphic to $\GL_{2a+\Fm}(F_\nu)$.

If $H_{a+m}(F_\nu)$ is isomorphic to $\GL_{2a+\Fm}(F_\nu)$, due to the splitness of the group,
the matrix calculation such as \eqref{eq:appendix-p} and \eqref{eq:appendix-n} is slight different. See \cite{Zh} for instance.
However, the proof for this case is completely same. Hence we omit the details here.

\begin{rmk}
As in Theorem 4.1 of \cite{S-I}, one may construct a special section such that the non-archimedean local zeta integral is a constant independent of $s$. We omit the details here.
\end{rmk}

\section{On Local Intertwining Operators}\label{B}


Throughout this appendix, let $F$ be a local field of characteristic 0.
Recall that $H^*_{m}$ is a quasi-split classical  group  defined over $F$ and $H_{m}$ is a pure inner $F$-form of $H^*_{m}$.
Let $\phi$ be a local $L$-parameter of $H^*_{m}(F)$
and $\wt{\Pi}_{\phi}(H_{m})$ the associated $L$-packet.
Assume that  $\phi$ is  generic,
that is, $\wt{\Pi}_\phi(H^*_{m})$ contains a generic member, following \cite{MW12}.
Up to a conjugation, assume that $\phi$ is of form as in Section \ref{sec-gap}
\begin{equation}\label{eq:L-parm}
\phi
=(\phi_1\otimes|\cdot|^{\beta_1}\oplus\phi_1^\vee\otimes|\cdot|^{-\beta_1})\oplus\cdots\oplus
(\phi_t\otimes|\cdot|^{\beta_t}\oplus\phi_t^\vee\otimes|\cdot|^{-\beta_t})\oplus\phi_0,
\end{equation}
where $\beta_1>\beta_2>\cdots>\beta_t>0$, all $\phi_i\colon \CL_F\to {}^L\!G_{E/F}(n_i)$ for $1\leq i\leq t$ and
$\phi_0\colon \CL_F\to {}^L\!H_{n_0}$ are tempered local $L$-parameters.
Then, the $L$-packet $\wt{\Pi}_\phi(H_{m})$ is defined to be the set of the Langlands quotients of the induced representations
\begin{equation}\label{eq:varphi-induce}
\Ind^{H_{m}(F)}_{P(F)}\tau(\phi_1)|\det|^{\beta_1}\otimes\cdots\otimes \tau(\phi_t)|\det|^{\beta_t}\otimes\sig_0
\end{equation}
where the parabolic subgroup $P$ has the Levi subgroup isomorphic to $G_{E/F}(n_1)\times\cdots\times G_{E/F}(n_t)\times H_{n_0}$,
$\sig_0$ runs though over the tempered $L$-packet $\wt{\Pi}_{\phi_0}(H_{n_0})$,
and $\tau(\phi_i)$ is the irreducible admissible unitary generic representation of $G_{E/F}(n_i)(F)$ given by the local Langlands correspondence
for the general linear groups.

\begin{prop} \label{prop:MW-generic-packet}
If $\phi$ is a generic $L$-parameter of $H^*_{m}$ as given in \eqref{eq:L-parm}, then all representations in
$\wt{\Pi}_\phi(H_{m})$ can be written as irreducible standard modules,
that is, the induced representations displayed in \eqref{eq:varphi-induce} are irreducible for all pure inner forms $H_{m}$ and $\sig_0\in \wt{\Pi}_{\phi}(H_{n_0})$.
\end{prop}

\begin{proof}
If $F$ is non-archimedean, this proposition is proved by M\oe glin and Waldspurger in \cite{MW12} for orthogonal groups, by Gan and Ichino in \cite[Proposition 9.1]{GI},
and by Heiermann in \cite{H15} for general reductive groups.
If $F$ is archimedean, it is a special case of Theorem 1.24 in the book by Adams, Barbasch and Vogan (\cite{ABV}).
More details can be found in Chapters 14 and 15 of \cite{ABV}.
\end{proof}

Proposition \ref{prop:MW-generic-packet} serves as a base for us to prove Theorem \ref{nlio}.
Recall that the normalized local intertwining operator $\CN(\omega_0, \tau\otimes\sig,s)_\nu$ takes
sections in the induced representation
\begin{equation}\label{ind}
\Ind^{H_{a+m}(F_\nu)}_{P_{\hat{a}}(F_\nu)}(\tau_\nu|\det|^s\otimes\sigma_\nu)
\end{equation}
to sections in the induced representation
\begin{equation}
\Ind^{H_{a+m}(F_\nu)}_{P_{\hat{a}}(F_\nu)}(\tau^*_\nu|\det|^{-s}\otimes\sigma_\nu),
\end{equation}
where $\tau_\nu$ is the local $\nu$-component of the irreducible isobaric automorphic representation $\tau$ as given in \eqref{tau8}, and
$\sigma_\nu$ is the local $\nu$-component of the irreducible cuspidal automorphic representation $\sigma$ in $\CA_\cusp(H_m)$ with an
$H_m$-relevant, generic global Arthur parameter $\phi_\sigma$ as in Theorem \ref{nlio}. It is clear that Theorem \ref{nlio}
follows from the following theorem.

\begin{thm}
Let $\phi^+$ be a local $\nu$-component of an $H_{m}$-relevant, generic global Arthur parameter of $H_{m}^*$.
If $\tau$ is an irreducible admissible unitary generic self-dual representation of $G_{E/F}(a)(F)$ and $\sig$ is
an irreducible representation in the generic local $L$-packet $\wt{\Pi}_{\phi^+}(H_{m})$, then
the normalized local intertwining operator $\CN(\omega_0,\tau\otimes\sig,s)$ is holomorphic and nonzero for $\Re(s)\geq \frac{1}{2}$.
\end{thm}

\begin{proof}
First of all, the local $L$-packet $\wt{\Pi}_{\phi^+}(H_{m})$ has a generic member $\sigma^\circ$ (\cite{A13} and \cite{Mk15})
 when $H_m=H_m^*$ is quasisplit.
If $\sig$ is generic, the proposition follows from Theorem 11.1 in \cite{CKPSS04}.

Assume now that $\sig$ is not generic.
For such a generic local $L$-packet $\wt{\Pi}_{\phi^+}(H_{m})$, by Proposition \ref{prop:MW-generic-packet}, the standard modules as displayed in \eqref{eq:varphi-induce} are irreducible. This is the key point for us to apply the argument in \cite{CKPSS04} in the proof of this proposition.

According to the structure of the generic unitary dual of the general linear groups, given by Vogan in \cite{V86} for the archimedean case and  by Tadic in
\cite{T86} for the non-archimedean case, any generic member $\sig^\circ$ in $\wt{\Pi}_{\phi^+}(H^*_{m})$ is isomorphic to the irreducible generic unitary
induced representation
 $$
 \Ind^{H^*_{m}(F)}_{P(F)}\tau(\phi_1)|\det|^{\beta_1}\otimes\cdots\otimes \tau(\phi_t)|\det|^{\beta_t}\otimes\sig_0
 $$
 where
\begin{equation}\label{ineq:expo}
\frac{1}{2}> \beta_1>\beta_2>\cdots>\beta_t> 0,
\end{equation}
 and all $\tau(\phi_i)$ and $\sig_0$ are irreducible, unitary, generic, and tempered.
By Proposition \ref{prop:MW-generic-packet}, each $\sig$ in $\wt{\Pi}_{\phi^+}(H_m)$ is of form \eqref{eq:varphi-induce} with the exponents satisfying \eqref{ineq:expo}.

Again, by the generic unitary dual of the general linear groups,
$\tau$ is isomorphic to the irreducible induced representation
\begin{equation}\label{eq:induc-GL}
\Ind^{\GL_a(E)}_{P'(E)}\tau_{1}|\det|^{\alpha_1}\otimes\cdots\otimes \tau_{d}|\det|^{\alpha_d}\otimes \tau_{0}\otimes
\tau^*_{d}|\det|^{-\alpha_d}\otimes\cdots\otimes \tau^*_{1}|\det|^{-\alpha_1}
\end{equation}
where all $\tau_i$ are unitary, generic and tempered, and
$$
\frac{1}{2}>\alpha_1>\cdots>\alpha_d>0.
$$

Now we replace the representations $\tau_\nu$ and $\sig_\nu$ in \eqref{ind} by their corresponding realizations in \eqref{eq:induc-GL} and \eqref{eq:varphi-induce}, respectively.
By the transitivity of parabolic induction, the normalized local intertwining operator $\CN(\omega_0, \tau_\nu\otimes\sig_\nu,s)$ can be expressed as a composition of the
local intertwining operators of rank one, which are of form
\begin{align}
&\CN(w_{j,i},\tau_j\otimes\tau(\phi_i),s\pm \alpha_j\pm \beta_i) \label{eq:appendx-N1}\\
&\CN(w'_{j,i},\tau_j\otimes\tau_i,2s\pm \alpha_{j}\pm \alpha_i) \label{eq:appendx-N2}\\
&\CN(w''_{j},\tau(\phi_i)\otimes\sig_0,s\pm \alpha_j), \label{eq:appendx-N3}
\end{align}
where $w_{j,i}$, $w'_{j,i}$, and $w''_{j}$ are the corresponding Weyl elements.
We deal with these three types of the local intertwining operators separately.

The first two types \eqref{eq:appendx-N1} and \eqref{eq:appendx-N2} were studied by M\oe glin and Waldspurger in
\cite{MW89}. For any unitary tempered $\tau$ and $\tau'$ of general linear groups,
the normalized intertwining operator $\CN(w,\tau\otimes\tau',s)$ is holomorphic and nonzero for $\Re(s)>-1$.
Because of the bounds for the exponents, it follows that
$$
\CN(w_{j,i},\tau_j\otimes\tau(\phi_i),s\pm \alpha_j\pm \beta_i)
\text{ and }
\CN(w'_{j,i},\tau_j\otimes\tau_i,2s\pm \alpha_{j}\pm \alpha_i)$$
are holomorphic and nonzero for $\Re(s)\geq 0$.

For the remaining type \eqref{eq:appendx-N3}, by the bound $0<\alpha_j<\frac{1}{2}$,
it is sufficient to show that the normalized intertwining  operator
$\CN(w'',\tau\otimes\sig_0,s)$ is holomorphic and nonzero for $\Re(s)>0$, when $\tau$ and $\sig_0$ are unitary tempered.
Since $\phi_0$ is a generic parameter, there is a generic representation $\sig^\circ_0$ in the tempered local $L$-packet
$\wt{\Pi}_{\phi_0}(H_{n_0})$. Hence we have the identity of local $L$-factors:
$$
L(s,\tau\times\sig_0)=L(s,\tau\times\sig^\circ_0).
$$
Referring to \cite{K05}, $L(s,\tau\times\sig^\circ_0)$ and $L(s,\tau,\rho)$ are holomorphic and nonzero for $\Re(s)>0$,
and so is the normalizing factor.
In addition, following Proposition IV.2.1 in \cite{W03} for the non-archimedean case and Lemma 4.4 in \cite{BW} for the archimedean case,
the non-normalized local intertwining operator $\CM(w'',\tau\otimes\sig_0,s)$ for tempered data is holomorphic and nonzero for $\Re(s)>0$.
It follows that both $\CM(w'',\tau_j\otimes\sig_0,s-\alpha_j)$ and $L(s-\alpha_j,\tau_j\times\sig_0)$ are holomorphic and nonzero
for $\Re(s)\geq \frac{1}{2}$ because $0<\alpha_j<\frac{1}{2}$.
Therefore, the normalized local intertwining operator $\CN(w''_{j},\tau(\phi_i)\otimes\sig_0,s\pm \alpha_j)$
is holomorphic and nonzero for $\Re(s)\geq \frac{1}{2}$.
Putting together the results for all three types, we complete the proof of this proposition.
\end{proof}

\end{document}